\documentclass[11pt]{article}\usepackage{fullpage,amsmath,amsfonts,amsthm,graphicx,amssymb,textcomp,enumerate,multicol,rotating,mathtools,xcolor}
\usepackage{pgffor}
\usepackage{tikz}
\usepackage{tikz-cd}
\usetikzlibrary{decorations.markings,calc}
\usepackage{hyperref}
\usepackage{ifthen}
\usepackage{s2s1tikz, s2s1bigtikz, s2s1commdiag}

\newtheorem{theorem}{Theorem}[section]
\newtheorem{corollary}[theorem]{Corollary}
\newtheorem{lemma}[theorem]{Lemma}
\newtheorem{proposition}[theorem]{Proposition}

\theoremstyle{definition}
\newtheorem{definition}[theorem]{Definition}

\theoremstyle{remark}

\theoremstyle{remark}
\newtheorem{remark}[theorem]{Remark}

\theoremstyle{remark}

\newcommand{\R}{\mathbb{R}}
\newcommand{\N}{\mathbb{N}}
\newcommand{\Z}{\mathbb{Z}}  

\newcommand{\F}{\mathcal{F}}
\newcommand{\E}{\mathcal{E}}

\newcommand{\C}{\mathcal{C}}

\newcommand{\Edd}[1]{\E^{\delta_#1}_{\delta_#1}}

\renewcommand{\L}{\mathcal{L}}
\newcommand{\Lk}[1][]{L_{#1}(\scalebox{.8}{$\vec{k}$})}
\newcommand{\Lzero}[1][]{L_{#1}(\scalebox{.8}{$\vec{0}$})}
\newcommand{\Linf}[1][]{L_{#1}\left(\vec{\infty}\right)}
\renewcommand{\sl}{s\ell}
\newcommand{\K}{\mathcal{K}}
\newcommand{\KLk}[1][]{K_L^{#1}(\scalebox{.8}{$\vec{k}$})}
\newcommand{\KLzero}[1][]{K_L^{#1}(\scalebox{.8}{$\vec{0}$})}
\newcommand{\KLinf}[1][]{K_L^{#1}\left(\vec{\infty}\right)}

\newcommand{\cone}{\mathrm{Cone}}

\newcommand{\h}{\mathrm{h}}
\renewcommand{\th}{\mathrm{th}}
\newcommand{\q}{\mathrm{q}}
\newcommand{\hocolim}{\mathrm{hocolim}}

\newcommand{\longhookrightarrow}{\lhook\joinrel\longrightarrow}

\newcommand{\hspRarrow}[1]{\hspace{#1}\Longrightarrow\hspace{#1}}
\newcommand{\sk}{\mathcal{S}}

\newcommand{\mConeFull}[6]{
\underset{#1,#2 \in #3}{\mathrm{Multicone}} \left( \KC^* \left( #4 \right) \xrightarrow{ #5^* } \h^{ \h_{#3}(#2) - \h_{#3}(#1) - 1} \KC^* \left( #6 \right) \right)
}
\newcommand{\mCone}[6]{
\underset{\h_{#3}(#2) - \h_{#3}(#1) = 1}{\mathrm{Multicone}} \left( \KC^* \left( #4 \right) \xrightarrow{ #5^* } \KC^* \left( #6 \right) \right)
}

\newcommand{\mConehqs}[1]{
\def\tempa{#1}
\mConehqsCont
}
\newcommand{\mConehqsCont}[9]{
\underset{\h_{#2}(#1) - \h_{#2}(\tempa) = 1}{\mathrm{Multicone}} \left( \KChqs{#6}{#7}^* \left( #3 \right) \xrightarrow{ #4^* }
 \KChqs{#8}{#9}^* \left( #5 \right) \right)
}

\newcommand{\mConeFullhqs}[1]{
\def\tempa{#1}
\mConeFullhqsCont
}
\newcommand{\mConeFullhqsCont}[9]{
\underset{\tempa,#1 \in #2}{\mathrm{Multicone}} \left( \KChqs{#6}{#7}^* \left( #3 \right) \xrightarrow{ #4^* } \h^{ \h_{#2}(#1) - \h_{#2}(\tempa) - 1} \KChqs{#8}{#9}^* \left( #5 \right) \right)
}

\newcommand{\mConeFullNOKChqs}[1]{
\def\tempa{#1}
\mConeFullhqsNOKCCont
}
\newcommand{\mConeFullhqsNOKCCont}[9]{
\underset{\tempa,#1 \in #2}{\mathrm{Multicone}} \left( \h^{#6}\q^{#7} \left( #3 \right) \xrightarrow{ #4^* } \h^{ \h_{#2}(#1) - \h_{#2}(\tempa) - 1} \h^{#8}\q^{#9} \left( #5 \right) \right)
}

\newcommand{\topp}[1]{#1^{\mathrm{top}}}
\newcommand{\bott}[1]{#1_{\mathrm{bot}}}

\newcommand{\SSi}{(S^0\times S^2)_i}

\newcommand{\KC}{\widetilde{KC}}
\newcommand{\KChq}[2]{\h^{#1}\q^{#2}KC}
\newcommand{\Khhq}[2]{\h^{#1}\q^{#2}Kh}
\newcommand{\KChqs}[2]{\h^{#1}\q^{#2}\KC}
\newcommand{\KCsimp}{C^*}

\tikzset{->-/.style={decoration={
  markings,
  mark=at position #1 with {\arrow{>}}},postaction={decorate}}}
  
\tikzset{-<-/.style={decoration={
  markings,
  mark=at position #1 with {\arrow{<}}},postaction={decorate}}}

\def\centerarc[#1](#2)(#3:#4:#5)
    { \draw[#1] ($(#2)+({#5*cos(#3)},{#5*sin(#3)})$) arc (#3:#4:#5); }

\title{Khovanov homology for links in $\#^r(S^2\times S^1)$}
\author{Michael Willis \\
Department of Mathematics, UCLA\\
\href{mailto:msw188@ucla.edu}{\texttt{msw188@ucla.edu}}}

\begin{document}

\maketitle
{\let\thefootnote\relax\footnote{This work was funded in part by the NSF grant DMS-1563615}}

\begin{abstract}
We revisit Rozansky's construction of Khovanov homology for links in $S^2\times S^1$, extending it to define Khovanov homology $Kh(L)$ for links $L$ in $M^r=\#^r(S^2\times S^1)$ for any $r$.  The graded Euler characteristic of $Kh(L)$ can be used to recover WRT invariants at certain roots of unity, and also recovers the evaluation of $L$ in the skein module $\sk(M^r)$ of Hoste and Przytycki when $L$ is null-homologous in $M^r$.  The construction also allows for a clear path towards defining a Lee's homology $Kh'(L)$ and associated $s$-invariant for such $L$, which we will explore in an upcoming paper.  We also give an equivalent construction for the Khovanov homology of the knotification of a link in $S^3$ and show directly that this is invariant under handle-slides, in the hope of lifting this version to give a stable homotopy type for such knotifications in a future paper.
\end{abstract}

\section{Introduction}
In \cite{Khov} Mikhail Khovanov introduced the Khovanov homology $Kh(L)$ of any link $L$ in $S^3$, which categorifies the Jones polynomial of $L$.  This construction was generalized for tangles in the 3-ball in two different, but equivalent, ways.  In \cite{Khov2}, Khovanov considered tangles with $2n$ `incoming' and $2m$ `outgoing' strands.  In the spirit of topological quantum field theories (TQFTs), Khovanov defined corresponding `arc algebras' $H_n$ and $H_m$ and assigned to any such tangle a complex of $H_n,H_m$-bimodules with proper gluing properties that could be used to recover $Kh(L)$.  Meanwhile, in \cite{BN}, Dror Bar-Natan constructed a universal categorification of the Temperley-Lieb algebra allowing him to assign to a tangle a formal complex of cobordisms between Temperley-Lieb diagrams in the 2-disk, again from which $Kh(L)$ could be recovered.

Later in \cite{Roz}, Lev Rozansky used Khovanov's framework to assign complexes of $H_n,H_n$-bimodules to tangles in $S^2\times[0,1]$ with $2n$ endpoints at both ends.  From this Rozansky was able to define the Khovanov homology of the closure of such a tangle in $S^2\times S^1$ by passing to the derived category of such complexes and taking the Hochschild homology, as expected under the axioms of a TQFT.  He then showed that the projective resolution needed for this computation could be approximated by concatenating large numbers of full twists to the given tangle and computing the Khovanov homology of the usual closure in $S^3$.

The main goal of this manuscript is to revisit the argument of Rozansky and extend it to define Khovanov homology invariants for links in $M^r:=\#^r(S^2\times S^1)$, the connect sum of $r$ copies of $S^2\times S^1$.

\begin{theorem}\label{thm:Main Thm}
Let $\L$ be a link in $M^r$ having even geometric intersection numbers $n_i$ with all of the belt spheres $S^2_i\subset M^r$, represented by a diagram $L$.  Then there exists a $\Z$-graded chain complex $KC^j(L)$ whose (graded) homology groups $H^{*,j}(\L):=H^*(KC^j(L))$ are invariants of the link $\L$ up to overall grading shifts which vanish if $\L$ is null-homologous in $M^r$.  For links $\L$ in $M^0=S^3$, these homology groups are precisely the traditional Khovanov homology groups of $\L$.
\end{theorem}

Theorem \ref{thm:Main Thm} provides link invariants that categorify both the skein module $\sk(M^r)$ of Hoste and Przytycki \cite{HP} and the WRT invariants of links in $M^r$ in the proper sense (see Section \ref{sec:decat} for precise statements).  To prove Thoerem \ref{thm:Main Thm}, we revise Rozansky's original construction and avoid the use of Khovanov's derived category and Hochschild homology, for which such an extension might be unclear.  Instead, we choose to remain in Bar-Natan's setting throughout our construction and utilize Rozansky's arguments involving infinite full twists, only applying Khovanov's functor at the end to produce homology groups.  We present a rough outline of this construction below.

Consider $M^r$ as built by performing $r$ $S^0$-surgeries on $S^3$.  We draw $M^r$ on the plane by specifying the projections of the surgery spheres $S^0\times S^2$ for the 0-surgeries, using dashed lines to match corresponding spheres.  Then any link $\L\subset M^r$ can be projected onto this plane as a tangle diagram with $n_i$ matching endpoints on the attaching spheres $(S^0\times S^2)_i$ (see Figure \ref{fig:L in Mr} for an example; these notions will be made precise in Section \ref{sec:Defining general Kh}).

To construct the link invariant, we imagine connecting the corresponding endpoints of the link using $n_i$ parallel copies of the dashed lines, giving us a diagram for a link in $S^3$.  We further augment this diagram by inserting a large number of full right-handed twists $\F_{n_i}^{k_i}$ into each of these connections (see Figure \ref{fig:L in Mr}).  Rozansky's arguments in \cite{Roz} give a prescription for defining a limiting chain complex (in Bar-Natan's universal Temperley-Lieb category) for this diagram as the quantities $k_i\rightarrow\infty$.  If we apply Khovanov's functor to this limiting complex, we have an infinite complex of chain groups; the Khovanov homology groups $Kh(\L)$ of $\L\subset M^r$ are then defined to be the (graded) homology groups of this complex.

\begin{figure}
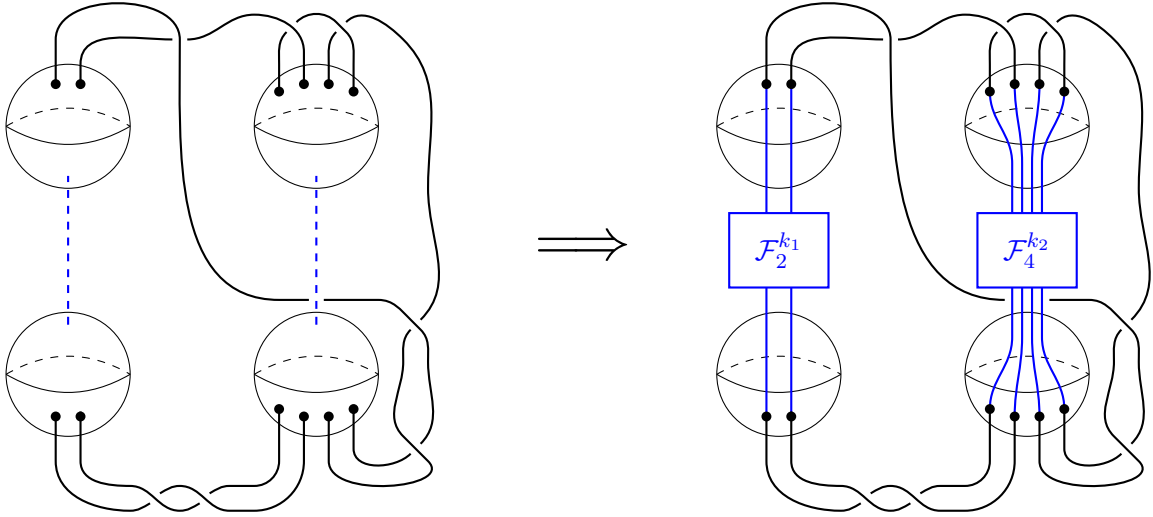

\centering
\LinMr
\hspace{.3in}
\resizebox{.5in}{!}{$\Longrightarrow$}
\hspace{.3in}
\LkinMr{\F_2^{k_1}}{\F_4^{k_2}}
\caption{On the left is an example of a link diagram $L$ for $\L\subset M^2$, with intersection numbers $n_1=2$ and $n_2=4$.  The dashed lines in blue indicate corresponding surgery spheres.  We then convert this into the diagram $\Lk$ on the right by changing the dashed lines into torus braids $\F_{n_i}^{k_i}$, also drawn in blue.  The complex $KC(L)$ is defined by letting the various $k_i$ go to infinity.}
\label{fig:L in Mr}
\end{figure}

In order to show that $Kh(\L)$ is indeed a link invariant, one must show that this construction is invariant with respect to isotopies in $M^r$.  In terms of our diagram $\Lk$, such isotopies incorporate the usual Reidemeister moves in the plane as well as movements of tangle endpoints on corresponding spheres, pushing strands `through' the attaching spheres, and the ability to `pull strands around' the spheres (and through the twists introduced into the diagram for $L$ along these dashed lines).  The usual Reidemeister moves are handled automatically from Khovanov's original construction.  The movements of the tangle endpoints, as well as the ability to push a strand through a sphere, are handled via simple properties of the full twist.  For pulling strands \emph{through} the twisting strands, the basic idea is that, for even numbers of strands $n_i$, the limiting complexes for the torus braids $\F_{n_i}^{k_i}$ as $k_i\rightarrow\infty$ can be written entirely in terms of diagrams where no strands pass from `top to bottom', and thus there are always gaps through which one can `pull' other strands in a consistent manner.

Many of these considerations follow directly from Rozansky's work in $S^2\times S^1$ \cite{Roz}, but we present the details here for completion.  In addition, there are two somewhat more substantial modifications to the arguments regarding the simplifications of multicones (Section 8.1 in \cite{Roz}; see Remark \ref{rmk:VSDR required} here) and the use of `quasi-triviality' of certain tangles (Section 7.3 in \cite{Roz}; see Remark \ref{rmk:JM not quasi-trivial?} here).

\subsection{Future work}
The construction above gives Khovanov homology groups for links up to isotopies within a fixed $M^r$, well-defined up to grading shifts.  It is not hard to generalize all of this to Bar-Natan's deformed Temperley-Lieb category, allowing for a Lee's homology for such links and a definition of an $s$-invariant similar to that of Jacob Rasmussen in \cite{Ras}.  We will explore this invariant and its associated bounds on the genus of cobordisms between links in a future paper with Ciprian Manolescu, Marco Marengon, and Sucharit Sarkar, to appear soon.

We will also show that, for knotifications $\K_\L \subset M^r$ of links $\L\subset S^3$ (used to define knot Floer homology for links \cite{OS}), handle slides of the underlying $M^r$ do not affect the chain homotopy equivalence class of our complexes, and using this will allow us to construct the Khovanov homology for these knotifications directly from the link diagram for $\L$ in $S^3$.  It is our hope that this alternative construction will be amenable to further study, including a lifting to the stable homotopy category following the work of \cite{LS,LLS} in a future paper.

\subsection{Organization of the paper}
This paper will be organized as follows.  Section \ref{sec:Infinite twist} will review the necessary homological algebra as it pertains to the Khovanov complexes in question before going on to produce a simplified Khovanov complex for the full twist on $n$ strands which satisfies certain important properties.  Much of this work is equivalent to similar work in \cite{Roz}, but is presented in a way that emphasizes the precise homological algebra being used.  In Section \ref{sec:Defining general Kh} we explore diagrams for links in $M^r$ before proving Theorem \ref{thm:Main Thm}, making precise the ideas presented here in the introduction.  We present example computations for some simple links in $M^1$ and $M^2$ in Section \ref{sec:Examples}.  Section \ref{sec:decat} discusses the decategorification of the invariant, and its relationships with the skein module $\sk(M^r)$ and WRT invariants.  Finally, in Section \ref{sec:knotification Kh} we present an alternative construction geared towards defining Khovanov homology for knotifications of links in $S^3$ directly from the given link diagram.

\subsection*{Acknowledgements}
The author would like to thank an anonymous referee for several very helpful suggestions that have improved the paper.  He would also like to thank Slava Krushkal, Sucharit Sarkar, Ciprian Manolescu, Andy Manion, and especially Matt Hogancamp for many very enlightening discussions.

\section{The Khovanov complex for the infinite full twist}
\label{sec:Infinite twist}
\subsection{Overview, conventions, and notation}
Let $R$ denote an arbitrary ground-ring.  In this section we will construct a semi-infinite chain complex for the infinite right-handed twist on an even number of strands.  This complex will be viewed as living in Bar-Natan's homotopy category of complexes of Temperley-Lieb diagrams and $R$-linear combinations of dotted cobordisms between them \cite{BN}.  In order to state the desired theorem, we first introduce some notational conventions.

\begin{definition}\label{def:notation list}
The following definitions and notations will be used throughout this paper.
\begin{itemize}
\item The variable $n$ will generally be reserved for the number of strands in a braid; when $n$ is even, it will be written as $n=2p$.
\item Tangles will generally be written as script capital letters.  In particular, we will use the notation $\F_n^k$ to indicate the torus braid of $k$ right-handed full-twists on $n$ strands, shown in Figure \ref{fig:Full Twist} below ($\F_n$ without the superscript will denote a single full twist).
\item The meaning of positive and negative crossings, as well as the meanings of 0-resolutions and 1-resolutions, is presented in Figure \ref{fig:crossing defs}.
\item The variable $n_{\F}^-$ represents the number of negative crossings present in a single copy of $\F_n$ (this is non-zero when the strands of $\F_n$ are not all oriented in the same direction).  Similarly we define $N_{\F}:=2n_{\F}^- - n_{\F}^+$, where $n_{\F}^+$ is the number of positive crossings in $\F_n$.  Later in Section \ref{sec:Defining general Kh}, when there are several $\F_{n_i}$ involved, we will use $n_i^-$ and $N_i$ (see Definition \ref{def:notations for surgery strands}).
\item If $\mathcal{Z}$ is some oriented tangle, then the notation $KC^*(\mathcal{Z})$ will denote the Khovanov complex of $\mathcal{Z}$ as described by Bar-Natan in \cite{BN}.  We will be employing Bar-Natan's conventions for the homological grading of terms in the complex:
\subitem The homological degree of a diagram $\delta$ is computed as the number of 1-resolutions taken to arrive at $\delta$, subtracted by the total number of negative crossings $n^-$ of the original oriented tangle $\mathcal{Z}$.
\subitem The $q$-degree of a diagram $\delta$ is computed as the number of 1-resolutions taken to arrive at $\delta$, subtracted by the normalization shift $N:=2n^--n^+$ coming from the numbers of positive ($n^+$) and negative ($n^-$) crossings in $\mathcal{Z}$.

In this way, $KC^*(\mathcal{Z})$ is a genuine graded invariant of the oriented $\mathcal{Z}$ up to chain homotopy equivalence, with no grading shifts necessary for Reidemeister moves.  The differential raises homological grading by one (ie $KC^*$ is a cochain complex) and respects the $q$-grading.  The asterisk refers to homological grading.
\item \label{it:shifts} Shifts in homological and $q$-grading of a complex $C^*$ will be denoted with $\h$ and $\q$, respectively.  Thus, the complex $\h^a \q^b KC^*(\mathcal{Z})$ indicates the Khovanov complex for $\mathcal{Z}$ shifted up in homological degree by $a$, and in $q$-degree by $b$.  The shifts take place before the use of the asterisk, so that the notation $\h^a \q^b KC^*(\mathcal{Z})$ should be thought of as implying $\big( \h^a \q^b KC\big) ^*(\mathcal{Z})$.
\item For a given chain complex $C^*$, the notation $C^*_{> a}$ will denote the truncated complex
\begin{equation}\label{eqn:truncated cx}
C^*_{> a}:= \begin{cases}
C^* & \mathrm{ if }\hspace{.1in}  *> a\\
\emptyset & \mathrm{ if }\hspace{.1in}  * \leq a
\end{cases}.
\end{equation}
Note that this ensures that $C^*_{> b}$ is a subcomplex of $C^*_{> a}$ as long as $b\geq a$, with the inclusion
\[C^*_{> b} \longhookrightarrow C^*_{> a}\]
inducing an isomorphism on homology in all homological degrees strictly greater than $b+1$.  As with the use of the asterisk, any shifts indicated are meant to take place before the truncation, so that $\h^a\q^b KC^*_{> c}(\mathcal{Z})$ implies $\big( \h^a \q^b KC\big)^*_{> c} (\mathcal{Z})$.
\item Single Temperley-Lieb diagrams within a complex $C^*$ will typically be denoted by greek letters.
We will use the notation $\h_{C^*}(\delta)$ to denote the homological grading of $\delta$ within $C^*$.  In many cases the complex will be clear from the context, and the subscript will be omitted.
\item For a Temperley-Lieb diagram $\delta$, the notation $\th(\delta)$ will indicate the \emph{through-degree} of $\delta$.  This is the number of strands within $\delta$ whose endpoints are on opposite ends of the diagram.  See Figure \ref{fig:Thru Deg} for clarification.
\item Identity cobordisms will be denoted by $I$, together with a subscript if necessary for clarification.  Dotted planar cobordisms between Temperley-Lieb diagrams will typically be written as $\phi_{i,j}:\delta_i\rightarrow\delta_j$.  More general dotted tangle cobordisms will often appear via Reidemeister moves and be notated by $\rho$.
\end{itemize}
\end{definition}

\begin{figure}
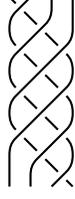

\centering
\FTex
\caption{The full twist, denoted by $\F_n$ (depicted here for $n=4$).}
\label{fig:Full Twist}
\end{figure}

\begin{figure}
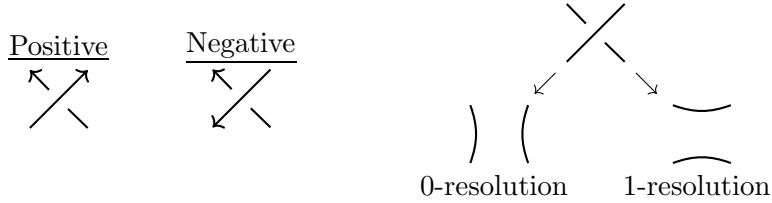

\[\CrossingDefs\]
\caption{On the left we show our conventions for positive and negative crossings.  On the right we show our conventions for a 0-resolution and a 1-resolution of a crossing.}
\label{fig:crossing defs}
\end{figure}

\begin{figure}
\centering
$\delta = \hspace{.04in}\thdegExA \hspace{.04in}\cong\hspace{.04in} \thdegExB 
$
\caption{The diagram $\delta$ shown here has through-degree $\th(\delta)=2$.  Any $\delta$ can be `pulled tight' so that $\th(\delta)$ is the cardinality of $\delta\cap\ell$ for a central horizontal line $\ell$, as in the second picture above.}
\label{fig:Thru Deg}
\end{figure}

With the notations indicated above, we state the goal of this section here.
\begin{theorem}\label{thm:inf twist complex}
Fix an even integer $n=2p$.  Let $\F_n$ denote an oriented full right-handed twist, and define $n_{\F}^-$ and $N_{\F}$ as above.  Then there exists a sequence of complexes $\KCsimp(k)$ satisfying the following properties.
\begin{enumerate}[(i)]
\item For each $k$ we have a chain homotopy equivalence
\[\KCsimp(k) \simeq \KChq{k(n_{\F}^- - 2p^2)}{k(N_{\F} - 2p(p+1))}^*(\F_n^k).\]
\item The truncated complexes $\KCsimp_{> -2k}(k)$ come equipped with inclusions
\[\KCsimp_{> -2}(1) \hookrightarrow \KCsimp_{>-4}(2) \hookrightarrow \cdots \hookrightarrow \KCsimp_{>-2k}(k)\hookrightarrow \cdots\]
\end{enumerate}
Thus there exists a well-defined stable limit
\[C^*(\F_n^\infty):= \hocolim_{k\rightarrow\infty} \KCsimp_{>-2k}(k)\]
that can be computed in any finite degree by truncating a complex that is chain homotopy equivalent to the Khovanov complex of a finite right-handed twist.  Furthermore, $\forall \delta\in C^*(\F_n^\infty)$, we must have $\h(\delta)\leq 0$ and $\th(\delta)=0$.
\end{theorem}
Compare Theorem \ref{thm:inf twist complex} with Theorems 6.6 and 6.7 in \cite{Roz}.  Our complexes $\KCsimp(k)$ here correspond to his $C^\sharp$ complexes in that paper.  The rest of this section is devoted to proving this theorem.

\begin{remark}
The reader who is familiar with Rozansky's earlier work \cite{Rozearly} using infinite twists to construct the categorified projectors of Cooper-Krushkal \cite{CK} may find Theorem \ref{thm:inf twist complex} puzzling.  In fact the limiting complex for the infinite twist depends on the system of maps used to define the limit.  In \cite{Rozearly}, the maps are chosen in such a way as to fix the `left' end of each complex, so that low homological degrees stabilize in the limit while high degrees get `pushed away to infinity'; this limits to a categorified highest weight projector in the sense of \cite{CK}, used to define colored Khovanov homology.  Here, however, the maps are chosen so as to fix the `right' end of each complex, so that high homological degrees stabilize while low degrees get pushed away; this limits to a categorified weight-zero projector, denoted by $P_{n,0}$, as described in \cite{Roz}.
\end{remark}

\subsection{Simplifying Khovanov complexes}
We begin by recalling the defining aspect of $KC^*(\cdot)$.  Given an oriented tangle diagram $\mathcal{Z}$, we construct $KC^*(\mathcal{Z})$ by finding crossings in $\mathcal{Z}$ and defining
\begin{equation}\label{eqn:KC main def}
KC^*(\crossing) = \h^{-n^-}\q^{n^+-2n^-}\left(\underline{\vres} \longrightarrow \q\hres\right)
\end{equation}
where the underlined term is in homological grading zero, the map is a saddle cobordism, and $n^+$ and $n^-$ are either 1 or 0 depending on whether the crossing was positively or negatively oriented.  The full $KC^*(\mathcal{Z})$ is then a large tensor product over all of these two term complexes, with diagrams and cobordisms stitched together in the usual sense of planar algebras (see \cite{BN} for more details).

In order to simplify $KC^*(\mathcal{Z})$, we have several tools at our disposal.  There is the obvious approach of breaking $\mathcal{Z}$ up into separate tangles and trying to simplify their Khovanov complexes individually, before tensoring them all back together again.  There is also the idea of using Equation \ref{eqn:KC main def} to view $KC^*(\mathcal{Z})$ as the cone on a chain map $KC^*(\vres)\rightarrow KC^*(\hres)$, and attempting to simplify $KC^*(\vres)$ and/or $KC^*(\hres)$ separately.  This idea can be expanded into viewing $KC^*(\mathcal{Z})$ as a large multicone of various cobordism maps in the hopes of simplifying various diagrams within the multicone.  However, the presence of the degree shifts causes a slight bit of trouble because the orientation of the original crossing (and thus the entire diagram) cannot be maintained on both of its resolutions.  Therefore we adopt the following convention, also adopted in the author's earlier related paper \cite{MW3}.

\begin{definition}\label{def:n- convention}
The symbol $\KC^*(\cdot)$ will stand for the \emph{renormalized Khovanov complex}
\begin{equation}\label{eqn:renormalized KC}
\KC^*(\cdot):=\KChq{n^-}{N}(\cdot)
\end{equation}
where the symbols $n^-$ and $N:=2n^- - n^+$ will count positive and negative crossings in whatever tangle they are attached to.  Thus we will write Equation \ref{eqn:KC main def} in the form
\begin{equation}\label{eqn:KC cone def}
\KC^*(\crossing) = \cone\left( \KC^*(\vres) \longrightarrow \q\KC^*(\hres) \right)
\end{equation}
and it will be understood that the various $n^-$ and $N$ terms that are implied by the notation are actually different numbers.  With this convention in place, it will not matter what orientations are assigned to the resolved diagrams within the cone of Equation \ref{eqn:KC cone def}.  
\end{definition}
Notice that this convention ensures that, for any tangle $\mathcal{Z}$, $\KC^*(\mathcal{Z})$ has left-most term in homological grading zero (ie taking 0-resolutions at every crossing), and right-most term in homological grading equal to the number of crossings in $\mathcal{Z}$ (ie taking 1-resolutions at every crossing).  However, the trade-off for this normalizing convention is that, since the `true' grading of $KC^*(\cdot)$ is invariant under Reidemeister moves, we must incur grading shifts when performing Reidemeister moves with $\KC^*(\cdot)$.

\begin{lemma}\label{lem:Reid grading shifts}
Suppose $D_1$ and $D_2$ are two tangle (or link) diagrams that are related by Reidemeister moves (ie the diagrams are tangle isotopic, representing the same tangle or link).  Let $n^-_1$ and $n^-_2$ be the number of negative crossings in the diagrams $D_1$ and $D_2$ respectively, and similarly for the $q$-grading renormalizations $N_1$ and $N_2$.  Then using the convention of Definition \ref{def:n- convention},
\begin{equation}\label{eqn:Reid grading shifts}
\KC^*(D_1) \simeq \h^{n_1^- - n_2^-}\q^{N_1 - N_2} \left( \KC^*(D_2) \right).
\end{equation}
\end{lemma}
\begin{proof}
\begin{align*}
\KC^*(D_1) &= \KChq{n^-_1}{N_1}^*(D_1)\\
&= \KChq{n_1^- - n_2^- + n_2^-}{N_1 - N_2 + N_2}^*(D_1)\\
&\simeq \KChq{n_1^- - n_2^- + n_2^-}{N_1 - N_2 + N_2}^*(D_2)\\
&\simeq \h^{n_1^- - n_2^-}\q^{N_1 - N_2} \left( \KChq{n_2^-}{N_2}^*(D_2) \right)\\
&= \h^{n_1^- - n_2^-}\q^{N_1 - N_2} \left( \KC^*(D_2) \right)
\end{align*}
\end{proof}
In other words, Reidemeister moves shift the renormalized homological grading of $\KC^*(\cdot)$ by precisely the number of negative crossings that were removed.  In particular, using negative Reidemeister 1 moves and Reidemeister 2 moves to eliminate crossings both shift homological degree by 1.  Similar statements hold for shifting of $q$-grading, but we will be focusing on the homological grading for the most part here.  Compare Lemma \ref{lem:Reid grading shifts} with the various shifts described by Rozansky throughout \cite{Roz}, and also Proposition 2.19 in \cite{MWMA}.

With the notations of Equations \ref{eqn:renormalized KC} and \ref{eqn:KC cone def} in place, we explore the notion of viewing Khovanov complexes as large multicones.  If we are planning on simplifying individual pieces of such a multicone via chain homotopy equivalences, we will need a construction that can keep track of these homotopies.  For this we recall some general homological algebra (compare to Section 8.1 in \cite{Roz}) for complexes over additive categories.


\begin{definition}\label{def:gen multicone def}
Suppose we are given the following data in a fixed category of chain complexes over some additive category:
\begin{itemize}
\item A finite index set $\C$ with a $\Z$-grading $\h_{\C}:\C\rightarrow\Z$.
\item For all $i\in\C$, a chain complex $(A_i^*,d_i)$.
\item For all $i,j\in\C$, a map (not necessarily a chain map) $f^*_{ij}:A^*_i\rightarrow \h^{\h_{\C}(j)-\h_{\C}(i) -1} A_j^*$ satisfying
\begin{itemize}
	\item $f_{ii}^*:=d_i$
	\item For all $j\neq i$ in $\C$ with $\h_{\C}(j)\leq\h_{\C}(i)$, $f_{ij}^*:=0$
	\item For all $i,k\in\C$, $\sum_{j\in\C} f_{jk}^* f_{ij}^* = 0$.
\end{itemize}
\end{itemize}
Then we can form the \emph{multicone} 
\begin{equation}\label{eq:gen multicone eq}
M=\underset{i,j\in\C}{\mathrm{Multicone}} \left( A_i^* \xrightarrow{f_{ij}^*} \h^{\h_{\C}(j)-\h_{\C}(i) - 1} A_j^* \right)
\end{equation}
 which is a chain complex $(M,d_M)$ whose terms are the direct sum of all of the terms of the complexes $A_i$
\[M:=\bigoplus_{i\in\C} A_i\]
and whose differential is the sum of all of the maps $f_{ij}$
\[d_M:=\sum_{i,j\in\C} f_{ij}.\]
For a term $\alpha\in A_i^*\subset M$, we determine the homological grading as the sum of the contributions of viewing $\alpha$ in $A_i^*$ and viewing $A_i^*$ in $\C$
\[\h_M(\alpha):= \h_{A_i}(\alpha) + \h_{\C}(i).\]
\end{definition}
The reader may verify that this definition gives a well defined chain complex.  When $\h_\C(j)-\h_\C(i)=1$, the maps $f_{ij}^*$ assemble to define chain maps; when $\h_\C(j)-\h_\C(i)=2$, the maps $f_{ij}^*$ assemble to form null-homotopies for the compositions of any two of these chain maps; and so on.  Because the original system $\C$ was finite, this process must eventually end, and so of course the sum in the definition of $d_M$ is finite.  When $\C=\{0,1\}$ and we have the single chain map $f_{01}^*:A_0\rightarrow A_1$, this construction recovers the usual cone on $f_{01}^*$.

\begin{remark}
We employ the term `multicone' in Definition \ref{def:gen multicone def} following \cite{Roz}.  A complex built in this manner is also often referred to as a \emph{totalization} or \emph{convolution} of a \emph{twisted complex}.  See for instance \cite{BK}.
\end{remark}

Note that \emph{any} finite chain complex $\C^*$ can be represented as a multicone by declaring that $\C$ is indexed by the terms in $\C^*$ while the maps $f_{i(i+1)}$ are given by the differentials of $\C^*$.  Meanwhile all of the maps $f_{ij}^*$ with $\h_\C(j)-\h_\C(i)\geq2$ are zero maps (ie no homotopies are needed).

\begin{proposition}\label{prop:gen multicone equiv}
Given a chain complex presented as a multicone $M$ as in Equation \ref{eq:gen multicone eq}, and given chain homotopy equivalences $\iota_i:A^*_i\rightarrow (A_i')^*$ for each $i\in\C$, there exist maps $(f_{ij}')^*$ such that we can form the multicone
\[M':=\underset{i,j\in\C}{\mathrm{Multicone}} \left( (A'_i)^* \xrightarrow{(f_{ij}')^*} \h^{\h_{\C}(j)-\h_{\C}(i) - 1} (A_j')^* \right)\]
that is chain homotopy equivalent to $M$:
\[M\simeq M'.\]
\end{proposition}
\begin{proof} This is a standard result generalizing the fact that the homotopy category of complexes over an additive category is triangulated (Proposition 2 in \cite{BK}).
\end{proof}

Proposition \ref{prop:gen multicone equiv} tells us that we can replace complexes within a multicone with chain homotopy equivalent ones, but it tells us nothing about the maps $(f_{ij}')^*$ that result.  For that we need some stronger assumptions.

\begin{definition}\label{def:VSDR}
A chain map $\iota:A^*\rightarrow (A')^*$ will be called a \emph{very strong deformation retract} if there are maps $\iota^{-1}:(A')^*\rightarrow A^*$ and $H:A^*\rightarrow\h^{-1}A^*$ satisfying:
\begin{enumerate}[(i)]
\item\label{it:SDR} $\iota^{-1}$ is a chain map such that $\iota\circ\iota^{-1}=I_{A'}$ and $\iota^{-1}\circ\iota=I_A + dH +Hd$.  In other words, $\iota$ is a strong deformation retract.
\item\label{it:side cond} The maps involved satisfy the following \emph{side conditions}: $HH=0,\quad \iota H=0,\quad H\iota'=0$.
\end{enumerate}
\end{definition}

Given a strong deformation retract $\iota$ satisfying only item \eqref{it:SDR} above, it is always possible to replace $\iota$ with a very strong deformation retract $\tilde{\iota}$ satisfying item \eqref{it:side cond} as well (see \cite{Stash}, as well as \cite{HomPert} for more considerations on such maps and constructions similar to the ones below).  In the cases of interest in this paper, however, our deformation retracts will all start out very strong.

\begin{proposition}\label{prop:multicone VSDR equiv}
In the situation of Proposition \ref{prop:gen multicone equiv}, suppose all of the maps $\iota_i$ are very strong deformation retracts (with accompanying maps $\iota_i^{-1}$ and $H_i$).  Then the maps $(f_{ij}')^*$ used to build $M'$ (satisfying $M'\simeq M$) are sums of terms of the form $\iota( f + fHf + fHfHf +\cdots)\iota^{-1}$ as indicated in slightly more detail below:
\[f_{i\ell}' = \sum_{(j_1,\dots,j_k)} \iota_\ell f_{j_k \ell} H_{j_k} f_{j_{k-1} j_k} \cdots f_{j_2j_3} H_{j_2} f_{j_1j_2} H_{j_1} f_{ij_1}\iota_i^{-1}\]
where we sum over all sequences of $j$'s that `partition the path' from index $i$ to index $\ell$ (disallowing the differentials $f_{jj}$).
\end{proposition}
\begin{proof}
The finiteness of $\C$ ensures that the sums above are finite.  In this case, one can check explicitly that the maps $f_{i\ell}'$ satisfy the conditions needed to build the multicone $M'$.  Then one can build explicit maps $\iota_{MM'}:M\rightarrow M'$ and $\iota_{M'M}:M'\rightarrow M$ written as follows:
\[\iota_{MM'} = \iota(I+fH + fHfH + \cdots),\qquad \iota_{M'M} = (I+Hf+HfHf+\cdots)\iota^{-1}.\]
All of these maps are to be interpreted as sums over `partitions of paths between various indices' as in the expansion of $\iota(f+fHf+fHfHf+\cdots)\iota^{-1}$ above.  Again, one can check explicitly that these are indeed chain maps (up to this point, the assumption that the various $\iota$ are very strong deformation retracts is not used).  The composition $\iota_{MM'}\circ\iota_{M'M}$ is actually equal to identity thanks to the side conditions \eqref{it:side cond} of Definition \ref{def:VSDR}.  The composition $\iota_{M'M}\circ\iota_{MM'}$ is only homotopic to identity, with a homotopy of the form $H_{M}=H+HfH+HfHfH+\cdots$.  The details of this computation are left to the energetic reader.
\end{proof}
Note that, even if $M^*$ required no non-zero homotopies $f_{ij}^*$ for $\h_\C(j)-\h_\C(i)\geq 2$, the equivalent $(M')^*$ will normally require such homotopies within its multicone structure.

We now specialize towards our goal of simplifying Khovanov complexes of tangles.  Let $\mathcal{C}^*$ represent any chain complex in Bar-Natan's category of planar diagrams and dotted cobordisms.

\begin{definition}\label{def:Multicone notation}
The notation
\begin{equation}\label{eqn:tangle Multicone notation}
\mathcal{C}^*=\mConeFull{\mathcal{Z}_i}{\mathcal{Z}_j}{\mathcal{C}^*}{\mathcal{Z}_i}{\phi_{i,j}}{\mathcal{Z}_j}
\end{equation}
indicates that the complex $\mathcal{C}^*$ is a multicone over various chain maps of complexes, where each individual complex is the Khovanov complex of a given tangle diagram $\mathcal{Z}_i$, and each $\phi_{i,j}$ is a dotted tangle cobordism (more generally, an $R$-linear combination of such cobordisms) from $\mathcal{Z}_i$ to $\mathcal{Z}_j$ that determines a chain map of dotted cobordisms $\phi_{i,j}^*:\KC^*(\mathcal{Z}_i)\rightarrow\KC^*(\mathcal{Z}_j)$.  Often the tangles $\mathcal{Z}_i$ will be genuine Temperley-Lieb diagrams $\delta_i$ with maps $\phi_{i,j}=0$ for $\h_{\C^*}(\delta_j)-\h_{\C^*}(\delta_i)\geq2$, so that we will see the simpler
\begin{equation}\label{eqn:Multicone notation}
\C^*=\mCone{\delta_i}{\delta_j}{\C^*}{\delta_i}{\phi_{i,j}}{\delta_j}.
\end{equation}
This will be the case when we are taking a known chain complex and choosing to view it as a multicone over its constituent Temperley-Lieb diagrams.
\end{definition}

If we are not concerned with the maps within such a multicone of cobordism maps, we have the following version of Proposition \ref{prop:gen multicone equiv}.
\begin{corollary}\label{cor:Cone of Concats}
Let $\mathcal{X}$ be some tangle.  Suppose $\C^*$ is a complex of dotted cobordisms $\phi_{i,j}$ between Temperley-Lieb diagrams $\delta_i,\delta_j$ where each $\delta_i$ can be concatenated with $\mathcal{X}$.  Suppose for each $i$ we have a tangle isotopy 
\[\rho_i:\braidAA[.15cm]{\mathcal{X}}{\delta_i} \xrightarrow{\cong} \braidA[.15cm]{\mathcal{X}'_i}.\]
Then there exist maps $(\phi_{i,j}')^*$ fitting into a multicone structure so that
\begin{equation}\label{eqn:Cone of Concats}
\begin{split}
M^*:=\mCone{\delta_i}{\delta_j}{\C^*} {\braidAA[.15cm]{\mathcal{X}}{\delta_i}} {\left(\braidAAmap[.15cm]{I_{\mathcal{X}}}{\phi_{i,j}}\right)} {\braidAA[.15cm]{\mathcal{X}}{\delta_j}} \\ 
 \simeq \mConeFullhqs{\delta_i}{\delta_j}{\C^*}{\mathcal{X}'_i}{(\phi'_{i,j})}{\mathcal{X}'_j}{a_i}{b_i}{a_j}{b_j}
\end{split}
\end{equation}
where the shifts $a_i,b_i,a_j,b_j$ are determined by applying Lemma \ref{lem:Reid grading shifts} to the isotopies $\rho_i$ and $\rho_j$.

Moreover, for the homological degree of any diagram $\epsilon$ coming from some $\KC^*(\mathcal{X}'_i)$ within the multicone $M^*$, we have
\begin{equation}\label{eqn:HomDeg in Multicone}
\h_{M^*}(\epsilon) = \h_{\KC^*(\mathcal{X}'_i)}(\epsilon) + \h_{\C^*}(\delta_i) + a_i
\end{equation}
and similarly for the $q$-degree.
\end{corollary}
\begin{proof}
If we use Equation \ref{eqn:Multicone notation} to write the known complex $\C^*$ as a multicone, this is Proposition \ref{prop:gen multicone equiv} applied to the current situation since the isotopies $\rho_i$ produce chain homotopy equivalences of shifted renormalized Khovanov complexes.  The diagram to have in mind is, for all $\delta_i,\delta_j\in\C$,

\ConeConcatCDprep
\[
\begin{tikzcd}
\usebox{\boxA}   \arrow[dd, "{\rho_i^*}"]  \arrow[rr, "{ \usebox{\boxAB} }"]  & &  \usebox{\boxB} \arrow[dd, "{\rho_j^*}"]\\
\\
\usebox{\boxC}   \arrow[rr, "{ (\phi_{i,j}')^* }"] & & \usebox{\boxD}
\end{tikzcd}.
\]
\end{proof}

On the other hand, if we need to know something about the maps $(\phi_{i,j}')^*$, we will need a version of Proposition \ref{prop:multicone VSDR equiv}, and so we will need our maps to be very strong deformation retracts.

\begin{lemma}\label{lem:R1 and R2 are VSDR}
Suppose $\rho:\mathcal{Z}\rightarrow\mathcal{Z}'$ is a Reidemeister 1 or 2 move that eliminates crossings.  Then the induced map $\rho^*$ on the Khovanov complexes is a very strong deformation retract.
\end{lemma}
\begin{proof}
This is easy to check using the definitions of the maps (given in Section 4 of \cite{BN}) and the fact that closed spheres evaluate to zero.  Indeed, in that paper the fact that Reidemeister 2 moves are strong deformation retracts is remarked upon and utilized to prove invariance under Reidemeister 3 moves.
\end{proof}

Although one could write down a general translation of Proposition \ref{prop:multicone VSDR equiv} for various situations, we will only need the following corollary that describes the main case of interest in this paper.

\begin{corollary}\label{cor:Consistent Cone of Concats}
Let $\mathcal{X}$ be some tangle.  Suppose $\C^*$ is a complex of dotted cobordisms $\phi_{i,j}$ between Temperley-Lieb diagrams $\delta_i,\delta_j$ where each $\delta_i$ can be concatenated with $\mathcal{X}$.  Suppose for each $\delta_i\in\C^*$ we have a tangle isotopy
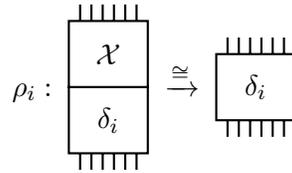
\[\rho_i:\braidAA[.15cm]{\mathcal{X}}{\delta_i} \xrightarrow{\cong} \braidA[.15cm]{\delta_i}\]
consisting entirely of Reidemeister 1 and 2 moves that remove crossings, satisfying the following conditions.
\begin{itemize}
\item The grading shifts $a_i,b_i$ for $\rho_i^*$ (via applying Lemma \ref{lem:Reid grading shifts}) are equivalent for all $i$ (that is, there exist constants $a,b$ independent of $i$ such that $a_i=a,b_i=b$ for all $i$).
\item For each dotted cobordism $\phi_{i,j}:\delta_i \rightarrow \delta_j$, the compositions $\rho_j\circ(I_\mathcal{X} \cdot \phi_{i,j})$ and $\phi_{i,j}\circ\rho_i$ are isotopic as dotted tangle cobordisms in $B^3\times[0,1]$:
\ConeConcatCobsCDprep
\[
\begin{tikzcd}
\usebox{\boxA}   \arrow[dd, "{\rho_i}"]  \arrow[rr, "{ \usebox{\boxAB} }"]  & &  \usebox{\boxB} \arrow[dd, "{\rho_j}"]\\
 & \cong & \\
\usebox{\boxC}   \arrow[rr, "{ \usebox{\boxCD} }"] & & \usebox{\boxD}
\end{tikzcd}.
\]
\end{itemize}
Then we have 
\[
\begin{split}
M^*:=\mCone{\delta_i}{\delta_j}{\C^*} {\braidAA[.15cm]{\mathcal{X}}{\delta_i}} {\left(\braidAAmap[.15cm]{I_{\mathcal{X}}}{\phi_{i,j}}\right)} {\braidAA[.15cm]{\mathcal{X}}{\delta_j}} \\ 
 \simeq \mConehqs{\delta_i}{\delta_j}{\C}{\delta_i}{\pm\phi_{i,j}}{\delta_j}{a}{b}{a}{b}.
\end{split}
\]
If in addition the Reidemeister maps $\rho_i^*$ can be given signs in a consistent manner to cancel with the $\pm$ signs above, we actually have
\[M^* \simeq \h^a\q^b \C^*.\]
\end{corollary}
\begin{proof}
Lemma \ref{lem:R1 and R2 are VSDR} tells us that all of our chain maps $\rho_i^*$ are very strong deformation retracts, so that Proposition \ref{prop:multicone VSDR equiv} ensures that the maps $(\phi_{i,j}')^*$ in the new multicone are originally defined as compositions
\[(\phi_{i,j}')^* := \rho_j^*( I_\mathcal{X} \cdot \phi_{i,j} )^*(\rho_i^{-1})^*.\]
Maps involving larger compositions do not exist because the complexes $\KChqs{a}{b}^*(\delta_i)$ are all \emph{single term} complexes, whereas the larger composition maps act as homotopies that need to reach lower homological gradings.  Then the assumptions imply we have an isotopy of dotted cobordisms
\[\rho_j\circ( I_\mathcal{X} \cdot \phi_{i,j} )\circ\rho_i^{-1} \cong \phi_{i,j}\]
at which point we invoke the functoriality of Bar-Natan's constructions to conclude that we have a homotopy
\[(\phi_{i,j}')^* \sim \pm\phi_{i,j}^*.\]
Once again, since we are dealing with single term complexes, the homotopy is irrelevant and these maps must in fact be equal.  Since $(-1)$ is also a very strong deformation retract, we have the option to include signs with our maps $\rho_i^*$ if necessary to cancel the $\pm$ signs, assuming a consistent choice of such signs can be found.  The proof is outlined by the commuting diagram of Figure \ref{fig:Consistent Cone of Concats proof}.
\end{proof}

In general, the majority of the theorems in this section are proved in a similar fashion.  We take a given tangle $\mathcal{Z}$ and view it as the concatenation of two tangles $\mathcal{Z}=\mathcal{X} \cdot \mathcal{Y}$ to analyze its renormalized Khovanov complex
\[\KC^*(\mathcal{Z})=\KC^*\left( \braidAA[.15cm]{\mathcal{X}}{\mathcal{Y}} \right) = \braidAA{\KC^*(\mathcal{X})}{\KC^*(\mathcal{Y})}.\]
We then look to apply either Corollary \ref{cor:Cone of Concats} or \ref{cor:Consistent Cone of Concats} by treating the complex $\KC^*(\mathcal{Y})$ (or perhaps some simplification of it) as the required $\C^*$ and finding tangle isotopies $\rho_i:\mathcal{X}\cdot\delta_i \xrightarrow{\cong} \mathcal{X}'_i$ for each $\delta_i\in\C^*$.  Under these conditions, Equation \ref{eqn:Cone of Concats} becomes
\begin{equation}\label{eqn:Cone of Concats std}
\KC^*(\mathcal{Z}) \simeq \mConeFullhqs{\delta_i}{\delta_j}{\C^*}{\mathcal{X}'_i}{(\phi'_{i,j})}{\mathcal{X}'_j}{a_i}{b_i}{a_j}{b_j}.
\end{equation}
In some cases we will only be interested in the homological degrees of various diagrams in $\KC^*(\mathcal{Z})$ so that the maps $(\phi'_{i,j})^*$ will be irrelevant and we can use the relatively simple Corollary \ref{cor:Cone of Concats}.  In other cases we will actually want to claim that $\KC^*(\mathcal{Z})$ is chain homotopy equivalent to a shifted copy of $\KC^*(\mathcal{Y})$, which will require checking all of the various requirements of Corollary \ref{cor:Consistent Cone of Concats}.

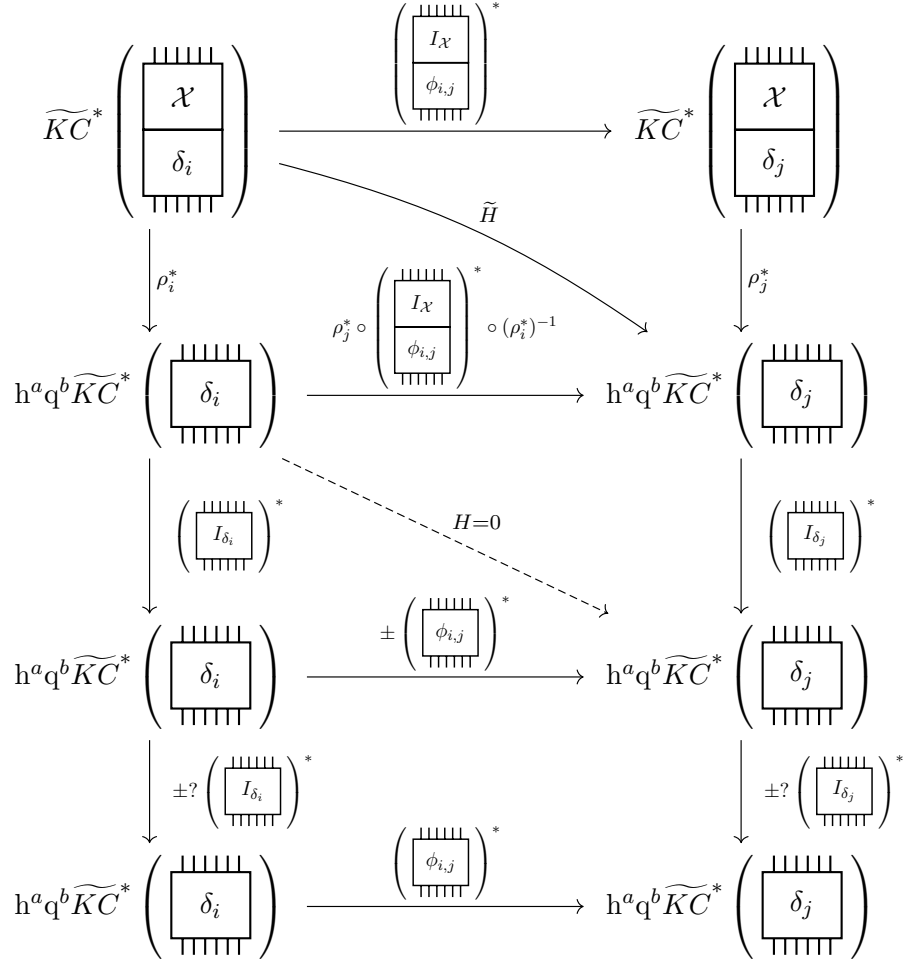
\begin{figure}

\ConsistentConeConcatCDprep
\centering
\begin{tikzcd}
\usebox{\boxA} \arrow[rrrr, "{\usebox{\boxAB}}"] \arrow[dd,"{\rho_i^*}"] \arrow[rrrrdd,bend left=10,"\widetilde{H}"] & & & & \usebox{\boxB} \arrow[dd,"{\rho_j^*}"] \\
\\
\usebox{\boxC} \arrow[rrrr, "{\usebox{\boxCD}}"] \arrow[ddd,"{\usebox{\boxCE}}"] \arrow[dddrrrr, dashed, "{H=0}"] & & & & \usebox{\boxD} \arrow[ddd,"{\usebox{\boxDF}}"] \\
\\
\\
\usebox{\boxE} \arrow[rrrr, "{\usebox{\boxEF}}"] \arrow[dd,"{\usebox{\boxEG}}"] & & & & \usebox{\boxF} \arrow[dd, "{\usebox{\boxFH}}"] \\
\\
\usebox{\boxG} \arrow[rrrr, "{\usebox{\boxGH}}"] & & & & \usebox{\boxH}
\end{tikzcd}
\caption{A diagram for Corollary \ref{cor:Consistent Cone of Concats}.  We envision a large system of these downward towers of maps, one for each pair of $\delta_i,\delta_j\in\C^*$ with $\h_{\C^*}(\delta_j)-\h_{\C^*}(\delta_i)=1$.  Because the terms $\KC^*(\delta_i)$ are one-term complexes, there is no need to account for higher degree maps.  Similarly, although Bar-Natan's functoriality only guarantees that the top two squares commute up to some homotopies $\widetilde{H}$ and $H$, the fact that our mid-layer complexes have only one term ensures that the central $H$ is zero.  The multicone of the top layer forms $\KC^*(\mathcal{Z})$, while the bottom layer forms a shifted copy of $\C^*$.  If the various $\rho$ maps involve only Reidemeister 1 and 2 simplifications (so the $\rho^*$ maps are very strong deformation retracts), we can build an equivalence from the first layer down to the second layer.  Since $H=0$, we can continue to the third layer; however a further argument is needed to find a consistent choice of signs to reach the fourth and final layer.}
\label{fig:Consistent Cone of Concats proof}
\end{figure}

\begin{remark}\label{rmk:VSDR required}
In \cite{Roz}, Rozansky does not make mention of the use of very strong deformation retracts, and instead collapses the tower of Figure \ref{fig:Consistent Cone of Concats proof} into two layers - the top and bottom only.  In doing so, it is true that the resulting square can be arranged to commute up to homotopy, but this is not enough to guarantee the existence of a map from the top multicone to the bottom multicone.  In addition, one needs these diagonal homotopies to commute with other `horizontal' maps within the multicones up to higher homotopies, which then need to commute with further horizontal maps up to even higher homotopies, and so forth.  Even then, it is not immediately clear what the homotopy inverse map should be.  The multi-layer approach of Figure \ref{fig:Consistent Cone of Concats proof} avoids this problem by using very strong deformation retracts to build the equivalence between the first two layers (put another way, the top square commutes up to a very simplistic homotopy $\widetilde{H}$ which we can incorporate into both the vertical map and its inverse), while the lower layers all have one-term complexes where commuting up to homotopy is the same is commuting on the nose and the equivalence is clear.
\end{remark}

\subsection{A remark on dotted cobordisms}\label{sec:Moving Dots}
In \cite{BN} Bar-Natan asserts that the arguments for functoriality concerning undotted tangle cobordisms continue to hold for dotted tangle cobordisms.  In rings where $2$ is invertible, a dotted cobordism can be converted to an undotted cobordism with a coefficient of $2^{-1}$ and so the claim is obvious.  Otherwise however, it is perhaps less clear to see precisely why this is the case.  Here we present a quick argument to show why two isotopic dotted tangle cobordisms do indeed induce homotopic maps on Khovanov complexes, up to a sign.  We begin with the following lemma which is well-known to experts.

\begin{lemma}\label{lem:Moving Dots}
The dotted identity cobordisms indicated below induce homotopic maps on Khovanov complexes up to a sign:
\begin{gather}\label{eq:Moving Dots}
\left(\crossingdotTL\right)^* \sim -\left(\crossingdotBR\right)^*\\
\left(\crossingdotTR\right)^* \sim -\left(\crossingdotBL\right)^*.
\end{gather}
\end{lemma}
\begin{proof}
The homotopy is illustrated in the following diagram.
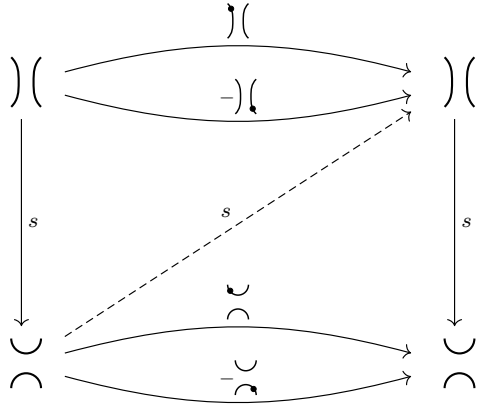
\MovingDotCDprep
\[
\begin{tikzcd}
\usebox{\boxA}  \arrow[rrrrr,bend left=15, "{ \usebox{\boxAB} }"]  \arrow[rrrrr, bend right=15, "{\usebox{\boxABa}}"]  \arrow[dddd, "s"]  & & & & &  \usebox{\boxB}  \arrow[dddd, "s"]\\
\\
\\
\\
\usebox{\boxC}  \arrow[rrrrr,bend left=15, "{\usebox{\boxCD}}"]  \arrow[rrrrr,bend right=15, "{\usebox{\boxCDa}}"]  \arrow[uuuurrrrr, dashed, "s"] & & & & &  \usebox{\boxD}
\end{tikzcd}
\]
The diagram shows the complex $\KC^*\left(\crossing\right)$ on the left and right, with the two chain maps $\left(\crossingdotTL\right)^*$ and $-\left(\crossingdotBR\right)^*$ drawn horizontally.  The homotopy is drawn as a dashed line.  The symbol $s$ stands for the obvious saddle map, so that composing $s\circ s$ via the homotopy creates a tube that can be cut using the neck-cutting relation of \cite{BN}.  The other homotopy is identical.
\end{proof}

\begin{proposition}\label{prop:dotted functoriality}
If $\phi_1$ and $\phi_2$ are two isotopic tangle cobordisms from tangle $Z_1$ to $Z_2$ in $B^3$, then their induced chain maps $KC^*(Z_1)\rightarrow KC^*(Z_2)$ are homotopic up to a sign.
\end{proposition}
\begin{proof}
An (un-dotted) tangle cobordism in $B^3\times [0,1]$ can be viewed as a `movie' in $B^3$, where $[0,1]$ is tracking the time coordinate of the movie.  Recall that any such movie can be decomposed into a sequence of elementary movies corresponding to Morse moves (saddle, birth, and death) and Reidemeister moves.  If such cobordisms are isotopic, then their movies are linked via some sequence of so-called `movie moves' (see Carter and Saito \cite{CaSaito}), all of which are shown to produce homotopies between the corresponding cobordism maps up to a sign in \cite{BN}.

For dotted tangle cobordisms, we must add one extra type of elementary movie: a single dotted identity cobordism.  Then if two dotted tangle cobordisms are isotopic, their movies are linked by a sequence of movie moves as in \cite{CaSaito} together with two new moves that correspond to sliding a dot along the cobordism in either the $B^3$ direction or the $[0,1]$ direction.  Sliding a dot in the $B^3$ direction clearly maintains the corresponding chain map, except possibly in the case where we move a dot past a crossing in our planar projection.  This is precisely the case covered by Lemma \ref{lem:Moving Dots}.

Meanwhile, sliding a dot in the $[0,1]$ direction gives rise to moves that swap the ordering of some elementary movie $m$ with a dotted identity $I_\bullet$.  With the help of Lemma \ref{lem:Moving Dots}, we can begin any such move by pushing the dot along the tangle in the $B^3$ direction until it is far from the portion of the tangle that is affected by $m$ (all elementary movies are local in nature).  Then the two orderings of movies $I_\bullet \circ m$ and $m \circ I_\bullet$ trivially must give the same cobordism map.

Finally, to ensure that these are all of the moves needed, we note that none of the movie moves of Cater-Saito involve a closed component of a cobordism, and thus there is never a need to consider a dot within one of their movie moves.  The dot can always be slid to the beginning frame or ending frame (depending on which boundary the dotted component reaches) before performing the movie move.
\end{proof}


Lemma \ref{lem:Moving Dots} also leads to chain homotopy equivalences for standard mapping cones:
\begin{gather}\label{eq:Moving Dots in Cones}
\cone\left( \KC^*\left(\crossing\right) \xrightarrow{\crossingdotTL} \KC^*\left( \crossing \right) \right)
\simeq
\cone\left( \KC^*\left(\crossing\right) \xrightarrow{- \crossingdotBR} \KC^*\left( \crossing \right) \right)\\
\cone\left( \KC^*\left(\crossing\right) \xrightarrow{\crossingdotTR} \KC^*\left( \crossing \right) \right)
\simeq
\cone\left( \KC^*\left(\crossing\right) \xrightarrow{- \crossingdotBL} \KC^*\left( \crossing \right) \right).
\end{gather}
We will further analyze how this affects certain specific multicones in Section \ref{sec:knotification Kh}.

With Corollaries \ref{cor:Cone of Concats} and \ref{cor:Consistent Cone of Concats} at the ready, and Proposition \ref{prop:dotted functoriality} allowing us to make use of Bar-Natan's functoriality for our constructions, we can now turn towards proving Theorem \ref{thm:inf twist complex} in earnest.  Although much of the work from this point is equivalent to the similar arguments of Rozansky in \cite{Roz}, we present the proofs for completeness.  The reader who is familiar with Rozansky's work may safely skim the remainder of this section.

\subsection{A simplified complex for $KC^*(\F_n)$}\label{sec:KC(F1)}
The saddle map of Equation \ref{eqn:KC main def} increases homological grading, and so it appears likely that, for a tangle like $\F_n$ made up entirely of right-handed crossings, higher homological gradings will correspond to diagrams with lower through-degrees.  Of course, for any diagram $\delta\in KC^*(\F_n)$, the parity of $\th(\delta)$ must match the parity of $n$.  Compare the following theorem with Theorem 8.3 in \cite{Roz}.

\begin{theorem}\label{thm:single ft cx}
For any $n\geq 1$, $\KC^*(\F_n)$ is chain homotopy equivalent to a complex $C^*(\F_n)$ where, for any $m\in \left\{0,1,\dots,\left\lfloor \frac{n}{2} \right\rfloor \right\}$, diagrams $\delta\in C^*(\F_n)$ satisfy the following properties:
\begin{enumerate}
\item Diagrams $\delta\in C^*(\F_n)$ having through-degree $\th(\delta)=n-2m$ appear only in homological degrees $\h(\delta)\leq 2m(n-m)$.
\item There exists at least one diagram $\delta\in C^*(\F_n)$ with through-degree $\th(\delta)=n-2m$ such that $\h(\delta)=2m(n-m)$.
\item No diagrams in $C^*(\F_n)$ contain disjoint circles.
\end{enumerate}
\end{theorem}
When $n$ is even, the complex $C^*(\F_n)$ will form the `base case' $\KCsimp(1)$ required by Theorem \ref{thm:inf twist complex} after some further shifts, but we refrain from that notation here since we will be inducting on $n$ (including odd $n$).  The symbol $m$ in Theorem \ref{thm:single ft cx} refers to the number of \emph{matchings} on either the top or bottom of the diagram $\delta$.  See Figure \ref{fig:thrudeg vs matchings}.

\begin{figure}
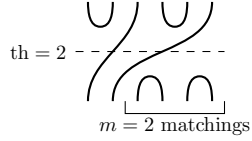

\centering
\thVSm
\caption{For this diagram $\delta$ with $n=6$, $m=2$ refers to the two matchings indicated (and of course, this implies that there are also $m=2$ matchings on the top half of $\delta$).  Thus $\th(\delta)=(6)-2(2) = 2$.}
\label{fig:thrudeg vs matchings}
\end{figure}

Although Theorem \ref{thm:single ft cx} can be strengthened to give relationships between $\th(\delta)$ and $\q(\delta)$ as well, we will not need those relationships here.

Before we can begin the proof, we need the following notations.

\begin{definition}\label{def:EM braid}
For any number of strands $n$, the symbol $\E_n$ will denote the right-handed Jucys-Murphy element of the braid group pictured in Figure \ref{fig:EM braid}.  In terms of the standard braid group generators, $\E_n:=\sigma_1\sigma_2\cdots\sigma_{n-1}\sigma_{n-1}\cdots\sigma_2\sigma_1$.
\end{definition}
Note that $\E_n$ has precisely $2(n-1)$ crossings.

\begin{figure}
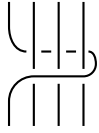

\centering
\Enpic
\caption{The right-handed Jucys-Murphy braid $\E_n$ for $n=4$.}
\label{fig:EM braid}
\end{figure}



\begin{definition}\label{def:delta top bot}
Let $\delta$ be a Temperley-Lieb diagram with through-degree $\th(\delta)=d$.  Any such diagram can be `pulled outward' until it can be written as the vertical concatenation of three diagrams, which we denote as follows:
\begin{equation}\label{eqn:delta top bot}
\delta = \topp{\delta} \cdot \iota_d \cdot \bott{\delta}.
\end{equation}
Here, $\iota_d$ is the vertical identity diagram on the $d$ strands that pass from top to bottom of $\delta$, contributing to the through-degree.  Then $\topp{\delta}$ and $\bott{\delta}$ refer to the upper and lower halves of $\delta$ that contain the upper and lower matchings.  As long as $\delta$ has no disjoint circles, this decomposition is unique.  See Figure \ref{fig:delta top bot} for clarification.
\end{definition}

\begin{figure}
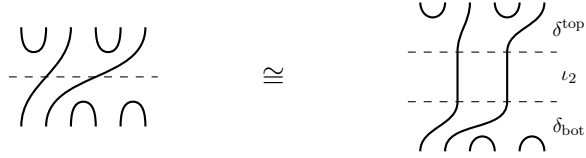

\centering
$\diagdecompA \hspace{.5in}\cong\hspace{.5in} \diagdecompB$
\caption{The diagram $\delta$ of Figure \ref{fig:thrudeg vs matchings} is decomposed as $\topp{\delta} \cdot \iota_d \cdot \bott{\delta}$.  Any Temperley-Lieb diagram can be decomposed in this way.}
\label{fig:delta top bot}
\end{figure}

The most important aspect of these notations is the isotopy described by the following lemma.
\begin{lemma}\label{lem:slide thru E}
For any $n$ and any Temperley-Lieb diagram $\delta$ with $\th(\delta)=d$, we have a braid isotopy
\begin{equation}\label{eqn:slide thru E}
\braidAB{\E_{n+1}}{\delta} \cong \dtEdb{\topp{\delta}}{\E_{d+1}}{\bott{\delta}}.
\end{equation}
that uses only Reidemeister 2 moves which remove crossings.
\end{lemma}
\begin{proof}
See Figure \ref{fig:slide thru E}, which illustrates the isotopy.  This is one version of `cup sliding' as it is referred to in \cite{Roz}.  If we always slide one `innermost' cup at a time, only Reidemeister 2 simplifications are used.
\end{proof}

\begin{figure}
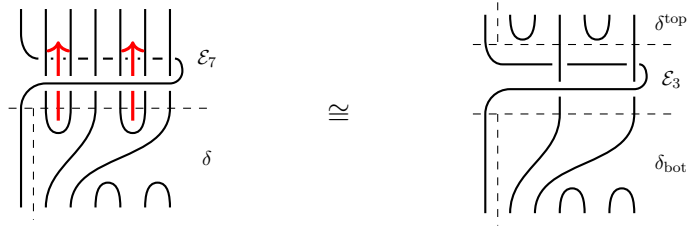

\centering
$\EncupslideA \hspace{.5in}\cong\hspace{.5in} \EncupslideB$
\caption{Illustration of the cup-sliding trick that performs the isotopy described by Equation \ref{eqn:slide thru E}.  The separation of the various braids is indicated by dashed lines.  In this example, $\delta$ is the diagram from Figure \ref{fig:thrudeg vs matchings} so $n=6$ with through-degree $d=2$.}
\label{fig:slide thru E}
\end{figure}

\begin{proof}[Proof of Theorem \ref{thm:single ft cx}]
We induct on $n$.  In the base case when $n=1$, the statement is obvious since the only choice for $m$ is zero, and $\KC^*(\F_1)$ is the one term complex consisting of only the identity diagram $\iota_1$ in homological degree 0.

Now we assume the theorem is true for $\F_n$, and seek to prove the statement for $\F_{n+1}$ using the well-known decomposition for $\F_{n+1}$ as 
\[
\braidA{\F_{n+1}} = \braidAB{\E_{n+1}}{\F_n}.
\]
We follow the strategy of Corollary \ref{cor:Cone of Concats} and Equation \ref{eqn:Cone of Concats std} with $\E_{n+1}$ playing the role of $\mathcal{X}$, while the role of $\C^*$ is played by $C^*(\F_n)$ as provided by the induction hypothesis, together with the single disjoint strand on the left.  Utilizing the notation of Equations \ref{eqn:Multicone notation} and \ref{eqn:Cone of Concats}, we can write this arrangement as
\[
\KC^*(\F_{n+1}) \simeq \mCone{\delta_i}{\delta_j}{C^*(\F_n)}{ \braidAB[.15cm]{\E_{n+1}}{\delta_i} } { \left( \braidABmap[.15cm]{I_{\E_{n+1}}}{\phi_{i,j}} \right) } { \braidAB[.15cm]{\E_{n+1}}{\delta_j} }.
\]

For each fixed $\delta_i\in C^*(\F_n)$ with $\th(\delta_i)=n-2m_i$, the isotopy $\rho_i$ is provided by Lemma \ref{lem:slide thru E}
\[
\rho_i: \braidAB{\E_{n+1}}{\delta_i} \xrightarrow{\cong} \dtEdb[.235cm]{\topp{\delta_i}}{\E_{n+1-2m_i}}{\bott{\delta^i}} =: \Edd{i}
\]
where we will use the symbol $\Edd{i}$ to denote the large diagram on the right (the symbol $\bott{\delta^i}$ is meant to indicate the bottom portion of $\delta_i$ as in Figures \ref{fig:delta top bot} and \ref{fig:slide thru E}).  Corollary \ref{cor:Cone of Concats} then gives us a multicone presentation
\[
\KC^*(\F_{n+1}) \simeq \mConeFullhqs{\delta_i}{\delta_j}{C^*(\F_n)} {\Edd{i}} {(\phi_{i,j}')} {\Edd{j}} {a_i}{b_i}{a_j}{b_j}.
\]
Here, if we want to estimate the homological degree of any diagram $\epsilon_i$ coming from some fixed $\KChqs{a_i}{b_i}(\Edd{i})$ within the multicone, we will need to consider the three quantities (see Equation \ref{eqn:HomDeg in Multicone})
\[a_i,\qquad \h_{C^*(\F_n)}(\delta_i),\qquad \h_{\KC^*\left( \Edd{i} \right)}(\epsilon_i) .\]

The homological shift $a_i$ can be computed using Lemma \ref{lem:Reid grading shifts} as applied to the isotopies $\rho_i$.  As noted in the proof of Lemma \ref{lem:slide thru E}, we use only Reidemeister 2 moves during the isotopy $\rho_i$, two for each matching.  Each such move eliminates a pair of crossings, one positive and one negative, regardless of orientations of strands.  Thus we eliminate precisely $2m_i$ negative crossings, and we have
\begin{equation}\label{eqn:Reid shift in Fn+1 cone}
a_i=2m_i.
\end{equation}

The induction hypothesis provides a homological bound for diagrams in $C^*(\F_n)$, giving
\begin{equation}\label{eqn:delta in Fn}
\h_{C^*(\F_n)}(\delta_i)\leq 2m_i(n-m_i).
\end{equation}

Finally, we consider diagrams in the complex $\KC^*\left( \Edd{i} \right)$.  The diagrams $\topp{\delta_i}$ and $\bott{\delta^i}$ have no crossings, while $\E_{n+1-2m_i}$ has precisely $2(n-2m_i)$ crossings.  The all-zero resolution of these crossings contributes a diagram $\epsilon_{i,0}$ with
\begin{equation}\label{eqn:all zero res in inner E}
\h_{\KC^*\left( \Edd{i} \right)}(\epsilon_{i,0}) = 0, \quad \th(\epsilon_{i,0})=n+1-2m_i
\end{equation}
while the other resolutions contribute diagrams $\epsilon_i$ satisfying
\begin{equation}\label{eqn:epsilon in inner E}
\h_{\KC^*\left( \Edd{i} \right)}(\epsilon_i) \leq 2(n-2m_i), \quad \th(\epsilon_i)=n+1-2m
\end{equation}
for some $m>m_i$.

Now to prove the theorem, we turn our arguments around and fix some $m\in\{0,1,\dots,\left\lfloor \frac{n+1}{2} \right\rfloor \}$.  Let $\hat{C}^*(\F_{n+1})$ denote the multicone
\[\hat{C}^*(\F_{n+1}):=\mConeFullhqs{\delta_i}{\delta_j}{C^*(\F_n)} {\Edd{i}} {(\phi_{i,j}')} {\Edd{j}} {a_i}{b_i}{a_j}{b_j}.\]
Suppose that $\epsilon\in\hat{C}^*(\F_{n+1})$ with $\th(\epsilon)=n+1-2m$.  By Equations \ref{eqn:all zero res in inner E} and \ref{eqn:epsilon in inner E}, we see that $\epsilon$ must come from some complex $\KC^*\left( \Edd{i} \right)$ with $\th(\delta_i)=m_i\leq m$.  In the case of equality $m_i=m$, Equations \ref{eqn:Reid shift in Fn+1 cone}, \ref{eqn:delta in Fn}, and \ref{eqn:all zero res in inner E} combine to give us the estimate
\begin{align*}
\h_{\hat{C}^*(\F_{n+1})}(\epsilon) &= a_i + \h_{C^*(\F_n)}(\delta_i) + \h_{\KC^*\left( \Edd{i} \right)}(\epsilon)\\
	&\leq 2m_i + 2m_i(n-m_i) + 0\\
	&= 2m(n+1-m).
\end{align*}
In the case $m_i<m$, we use Equation \ref{eqn:epsilon in inner E} in place of Equation \ref{eqn:all zero res in inner E}, but the rest remains the same:
\begin{align*}
\h_{\hat{C}^*(\F_{n+1})}(\epsilon) &= a_i + \h_{C^*(\F_n)}(\delta_i) + \h_{\KC^*\left( \Edd{i} \right)}(\epsilon)\\
	&\leq 2m_i + 2m_i(n-m_i) + 2(n-2m_i)\\
	&= 2(m_i+1)(n+1-(m_i+1))\\
	&\leq 2m(n+1-m)
\end{align*}
where the final inequality uses the fact that $2x(n+1-x)$ is an increasing function of $x$ for $x\in \left[ 0,\left\lfloor \frac{n}{2} \right\rfloor \right]$ together with the fact that $m_i<m$ (and so in particular $m_i+1\leq m$).  Together, these two bounds show that the first condition of Theorem \ref{thm:single ft cx} is satisfied by $\hat{C}^*(\F_{n+1})$.

In order to find some diagram $\epsilon$ that satisfies the second condition of the theorem, we split into two cases.  As long as we are not in the case of odd $n$ with $m = \frac{n+1}{2}$, we can use the inductive assumption to find a diagram $\delta_i\in C^*(\F_n)$ with $\th(\delta_i)=n-2m$ and $\h(\delta_i)=2m(n-m)$ for our desired $m$.  We then consider the diagram $\epsilon_{i,0}$ of Equation \ref{eqn:all zero res in inner E} for this choice of $\Edd{i}$ with
\[\th(\epsilon_{i,0}) = n+1-2m,\] 
\begin{align*}
\h_{\hat{C}^*(\F_{n+1})}(\epsilon_{i,0}) &= 2m+2m(n-m)+0\\
&= 2m(n+1-m)
\end{align*}
as desired.

Finally, we consider the second condition of the theorem where $n$ is odd and $m=\frac{n+1}{2}$.  That is, we would like to find a diagram $\epsilon\in C^*(\F_{n+1})$ with $\th(\epsilon)=0$ and $\h(\epsilon)=(n+1)\left(\frac{n+1}{2}\right)$.

By the inductive assumption there exists some $\delta_i\in C^*(\F_n)$ with $m_i=\frac{n-1}{2}$ matchings, having $\th(\delta_i)=1$ and $\h(\delta_i)=(n-1)\left(\frac{n+1}{2}\right)$ (ie a term in $C^*(\F_n)$ of maximal homological degree and minimal through-degree).  From this $\delta_i$ we arrive at a diagram $\Edd{i}$ of the form
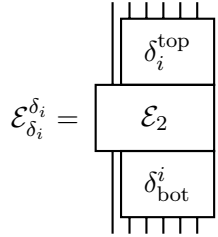
\[\Edd{i}=\dtEdb{\topp{\delta_i}}{\E_2}{\bott{\delta^i}}\]
and it is well-known that $\E_2=\F_2$ has a Khovanov complex having diagrams with through-degree zero in homological degrees one and two.  The homological degree 2 term here contributes the required diagram $\epsilon\in \hat{C}^*(\F_{n+1})$, since it would have
\[\h_{\hat{C}^*(\F_{n+1})}(\epsilon)=\left((n-1)\left(\frac{n+1}{2}\right)\right) + \left(2\left(\frac{n-1}{2}\right)\right) + \left(2\frac{}{}\right) = (n+1)\left(\frac{n+1}{2}\right)\]
as desired.

This indicates that our multicone $\hat{C}^*(\F_{n+1})$ has the first two desired properties, and is chain homotopy equivalent to $\KC^*(\F_{n+1})$.  Finally we can use Bar-Natan's local `delooping' relations \cite{BN2} to remove any disjoint circles from any of the diagrams in $\hat{C}^*(\F_{n+1})$ without altering homological gradings, arriving at our desired complex $C^*(\F_{n+1})$ and so we are done.
\end{proof}

\begin{corollary}\label{cor:F top degrees}
For even $n=2p$, $C^*(\F_n)$ has its two maximal non-empty homological gradings $2p^2-1$ and $2p^2$ containing only diagrams $\delta$ with $\th(\delta)=0$.
\end{corollary}
\begin{proof}
Note that the homological bound $2m(n-m)$ is an increasing function of $m\in\{0,1,\dots,p\}$ (or equivalently, a decreasing function of through-degree).
\begin{align*}
&\th(\delta)=0=n-2p &\Longrightarrow & &\h(\delta)\leq 2p(n-p)=2p^2\\
&\th(\delta)=2=n-2(p-1) &\Longrightarrow & &\h(\delta)\leq 2(p-1)(n-p+1)=2p^2-2\\
&\th(\delta)>2 &\Longrightarrow & &\h(\delta)<2p^2-2
\end{align*}
\end{proof}

\subsection{Equivalence of $KC^*(\F_n^k)$ and $KC^*(\F_n^{k+1})$ in high homological degree for even $n$}
\label{sec:equiv k to k+1}
Throughout this subsection (and indeed for the majority of the rest of this manuscript) we will focus on even $n=2p$.  In Section \ref{sec:KC(F1)} we defined a complex $C^*(\F_n)$ which will play the role of $\KCsimp(1)$ in Theorem \ref{thm:inf twist complex}.  Corollary \ref{cor:F top degrees} shows how the top two degrees of $C^*(\F_n)$, which will play the role of $\KCsimp_{>-2}(1)$, have special properties.  Our goal in this section is to define complexes $C^*(\F_n^k)$ inductively to play the role of $\KCsimp(k)$.  The inductive definition will be tailored to allow for the inclusion of truncations demanded by Theorem \ref{thm:inf twist complex}.

\begin{theorem}\label{thm:Fk to Fk+1}
Fix $n=2p$.  Then for all $k\in\N$ there exists a complex $C^*(\F_n^k)$ satisfying the following properties for each fixed $k$.
\begin{enumerate}[(i)]
\item \label{it:equiv to FTk} We have a chain homotopy equivalence
\[\C^*(\F_n^k)\simeq \KC^*(\F_n^k).\]
\item \label{it:th deg zero} The shifted, truncated complex $\h^{-2kp^2}\q^{-2kp(p+1)}C^*_{>-2k}(\F_n^k)$ is comprised of diagrams $\delta$ having no disjoint circles, and satisfying $\h(\delta)\leq 0$, $\th(\delta)=0$.
\item \label{it:k to k+1} We have an equality of shifted, truncated complexes
\begin{equation}\label{eqn:Fk to Fk+1}
\h^{-2kp^2}\q^{-2kp(p+1)}C^*_{>-2k}(\F_n^k) = \h^{-2(k+1)p^2}\q^{-2(k+1)p(p+1)}C^*_{>-2k}(\F_n^{k+1}).
\end{equation}
In other words, the top $2k$ degrees of $C^*(\F_n^k)$ are precisely the same as the top $2k$ degrees of $C^*(\F_n^{k+1})$ (with a $q$-degree shift).
\end{enumerate}
\end{theorem}
We will prove Theorem \ref{thm:Fk to Fk+1} using a combination of Corollary \ref{cor:Cone of Concats} and Corollary \ref{cor:Consistent Cone of Concats}.  We begin with a simple lemma for presenting the necessary Reidemeister maps $\rho$ related to concatenating with a full twist $\F_n$, together with their homological shifts.

\begin{lemma}\label{lem:ft Reid shifts}
If $\delta$ is a Temperley-Lieb diagram with $m$ matchings, so that $\th(\delta)=n-2m$ and $\delta=\topp{\delta}\cdot\iota_{n-2m}\cdot\bott{\delta}$ (as in Definition \ref{def:delta top bot}), we have
\begin{equation}\label{eqn:ft Reid shifts}
\KC^* \left( \braidAA[.15cm]{\F_n}{\delta} \right) \simeq \KChqs{2m(n-m)}{2m(n+1-m)}^* \left( \braidAAA{\topp{\delta}}{\F_{n-2m}}{\bott{\delta}} \right)
\end{equation}
via a tangle isotopy utilizing only Reidemeister 1 and 2 moves which remove crossings.
\end{lemma}
\begin{proof}
We pull each `innermost' matching of $\topp{\delta}$ up through the full twist as shown in Figure \ref{fig:pull thru ft} one at a time.  Each such pull removes 2 strands from $\F_n$ utilizing only Reidemeister 1 and 2 simplifications, so that we are left with $\F_{n-2m}$.  To compute the grading shifts, we count positive and negative crossings.  $\F_n$ had $n(n-1)$ crossings, while $\F_{n-2m}$ has $(n-2m)(n-2m-1)$ crossings, so we eliminated a total of $4m(n-m)-2m$ crossings.  Each matching pull removes precisely 2 crossings that are between the matched strands, which must be negative crossings (regardless of orientation).  This means $2m$ of the eliminated crossings were negative.  The other $4m(n-m-1)$ crossings occurred between unmatched strands, and thus must have occurred in positive and negative pairs (again regardless of orientation), and so we have removed $2m(n-m-1)$ positive crossings and $2m(n-m)$ negative crossings.  Lemma \ref{lem:Reid grading shifts} can then be used to complete the computation.

\begin{figure}
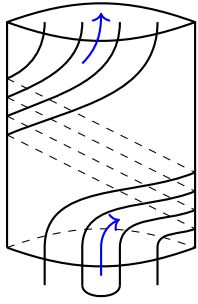

\centering
\FTpullthru
\caption{A matching can always be pulled through a full twist, as illustrated above.  Each matching in $\delta$ can be pulled through in this way, ending up in precisely the same spot above $\F_n$ as it was below it, but removing 2 strands from $\F_n$.}
\label{fig:pull thru ft}
\end{figure}

\end{proof}

Our next construction is designed to allow for the full use of Corollary \ref{cor:Consistent Cone of Concats} by assigning signs to Reidemeister maps $\rho_i$ coming from Lemma \ref{lem:ft Reid shifts} in the case where $\th(\delta_i)=0$.  This portion of the argument is based on Rozansky's discussion on `quasi-trivial' tangles (Section 7.3 in \cite{Roz}), where the full twist is considered quasi-trivial without proof (the same claim is made for the Jucys-Murphy braid - see Remark \ref{rmk:JM not quasi-trivial?}).  We will provide this proof for the full twist below (but will circumvent the corresponding claim for the Jucys-Murphy braid in Section \ref{sec:general Kh invariance}).

Following \cite{Roz}, we begin by fixing a `closure' diagram for our $n$ strands
\[\gamma:=
\begin{tikzpicture}[baseline={([yshift=-.7ex]current bounding box.center)},x=.25cm,y=.22cm]
\foreach \n in {0,2,5}
{
\draw[thick]
(\n,0) -- (\n,1) to[out=90,in=90] (\n+1, 1) -- (\n+1,0);
}
\draw
(4.1,0.7) node[scale=.7] {\dots};
\end{tikzpicture}.
\]
This allows us a consistent Reidemeister map
\[
\rho_\gamma: \toplessbraidAA[.15cm]{\gamma}{\F_n} \xrightarrow{\cong} \toplessbraidA[.15cm]{\gamma}
\]
via a reflection of the isotopy in Figure \ref{fig:pull thru ft}.  This in turn gives us, for any $\delta_i$ with $\th(\delta_i)=0$, a pair of Reidemeister maps

\rhogammarhoitangleCDprep
\begin{equation}\label{eq:rhogammarhoitangle}
\begin{tikzcd}
\usebox{\boxA} \arrow[rr,bend left,"{\usebox{\boxAB}}"] \arrow[rr,bend right,swap,"{\usebox{\boxABa}}"] & & \usebox{\boxB}
\end{tikzcd}.
\end{equation}

\begin{lemma}\label{lem:FT quasitrivial}
The two maps of Equation \ref{eq:rhogammarhoitangle} are isotopic as cobordisms embedded in $B^3\times[0,1]$ (ie the full twist is quasi-trivial in the language of \cite{Roz}).
\end{lemma}
\begin{proof}
We have the obvious isotopy

\rhogammarhoitangleCDprepTWO
\[
\begin{tikzcd}
\usebox{\boxA} \arrow[rr,bend left,"{\usebox{\boxAB}}"] \arrow[rr,bend right,swap,"{\usebox{\boxABa}}"] & \cong & \usebox{\boxB}
\end{tikzcd}
\]
and it is clear in $B^3\times[0,1]$ that
\[\toplessbraidAA[.25cm]{\rho_\gamma \circ \rho_\gamma^{-1}}{I_{\delta_i}} \circ \toplessbraidAA[.15cm]{I_\gamma}{\rho_i} \cong \toplessbraidAA[.15cm]{\rho_\gamma}{I_{\delta_i}} \circ \toplessbraidAA[.15cm]{\rho_\gamma^{-1}}{\rho_i}.\]
Let $\rho_{2\pi}$ denote the cobordism $\rho_\gamma^{-1}\cdot\rho_i$, which is isotopic to rotating the disjoint link $\gamma\cdot\F_n\cdot\topp{\delta_i}$ by a full $2\pi$ radians.  The key point is to notice that the link $\gamma\cdot\F_n\cdot\topp{\delta_i}$ is actually the \emph{unlink} and so bounds disjoint discs that are maintained throughout the rotation $\rho_{2\pi}$.  Thus $\rho_{2\pi}$ is isotopic to a cobordism that shrinks each component of $\gamma\cdot\F_n\cdot\topp{\delta_i}$ to a small circle (roughly a point in $B^3$) before rotating and then expanding back to $\gamma\cdot\F_n\cdot\topp{\delta_i}$, which is isotopic to identity since any two 1-dimensional cobordisms in $B^3\times[0,1]$ (with the same boundary and no closed components) are isotopic by dimension arguments.  Thus $\rho_{2\pi}=\rho_\gamma^{-1}\cdot\rho_i\cong I$ and we are done.
\end{proof}

With Lemma \ref{lem:FT quasitrivial} in hand, the functoriality of Bar-Natan's construction tells us that the link cobordisms in Equation \ref{eq:rhogammarhoitangle} must induce chain maps that are homotopic up to a sign which we will denote by $\sigma_i$.  The reader can then verify that the following diagram is entirely commutative up to homotopy, with compositions homotopic to identity:

\rhogammarhoichainCDprep
\begin{equation}\label{eqn:rho_gammarho_i}
\begin{tikzcd}
\usebox{\boxA} \arrow[rrrr,shift left=.5, bend left=05, "{\usebox{\boxAB}}", pos=.4] \arrow[rrrr,shift right=.5, bend right=05, swap, "{\usebox{\boxABa}}", pos=.4]
& & & & \usebox{\boxB} \arrow[llll, bend right=40, swap, "{\usebox{\boxBA}}", pos=.45] \arrow[llll, bend left=40, "{\usebox{\boxBAa}}", pos=.45]
\end{tikzcd}.
\end{equation}

\begin{lemma}\label{lem:ft cone preserved}
Fix $n=2p$.  If $\C^*$ is a chain complex made entirely of dotted cobordisms $\phi_{i,j}$ between diagrams $\delta_i,\delta_j$ all having through-degree zero \emph{and} having no disjoint circles, then
\begin{equation}\label{eqn:ft cone preserved}
\mCone{\delta_i}{\delta_j}{\C^*} {\braidAA[.15cm]{\F_n}{\delta_i}} {\left(\braidAAmap[.15cm]{I_{\F_n}}{\phi_{i,j}}\right)} {\braidAA[.15cm]{\F_n}{\delta_j}} \simeq \h^{2p^2}\q^{2p(p+1)}\C^*
\end{equation}
\end{lemma}
\begin{proof}
The braid isotopies of Lemma \ref{lem:ft Reid shifts} provide the maps $\rho_i: \mathcal{F}_n \cdot \delta_i \xrightarrow{\cong} \delta_i$ which utilize only Reidemeister 1 and 2 simplifications, as required by Corollary \ref{cor:Consistent Cone of Concats}.  Since the $\delta_i,\delta_j$ both have no disjoint circles, there is a consistent gap between their top and bottom matchings respected by the dotted cobordism $\phi_{i,j}$ allowing us to conclude that the tangle cobordisms $\rho_i,\rho_j,\phi_{i,j}$ commute up to isotopy, satisfying the second requirement of Corollary \ref{cor:Consistent Cone of Concats}.  The reader can quickly check that the grading shifts are precisely as indicated (note $\th(\delta_i)=0\Rightarrow m_i=p$).  We claim that the signs $\sigma_i$ defined with the help of $\gamma$ above provide the necessary consistency allowing us to collapse the bottom two layers of Figure \ref{fig:Consistent Cone of Concats proof} into the following single commuting diagram for all $\delta_i,\delta_j\in\C^*$:

\consistentFTconeCDprep
\begin{equation}\label{eqn:consistentFTcone}
\begin{tikzcd}
\usebox{\boxC} \arrow[rrrr, "{\usebox{\boxCD}}"] \arrow[ddd,"{\usebox{\boxCE}}"] \arrow[dddrrrr, dashed, "{H=0}"] & & & & \usebox{\boxD} \arrow[ddd,"{\usebox{\boxDF}}"] \\
\\
\\
\usebox{\boxE} \arrow[rrrr, "{\usebox{\boxEF}}"] & & & & \usebox{\boxF} 
\end{tikzcd}.
\end{equation}

To prove that the maps of Equation \ref{eqn:consistentFTcone} are indeed homotopic (and thus commute), we consider closing all of the diagrams present with our fixed closure $\gamma$ (now omitting the trivial homotopy $H$):

\RozQuasiIdeaCDprep
\begin{equation}\label{eqn:RozQuasiIdea}
\begin{tikzcd}
\usebox{\boxC} \arrow[rrrrrr, "{\usebox{\boxCD}}"] \arrow[ddd,"{\usebox{\boxCE}}"] & & & & & & \usebox{\boxD} \arrow[ddd,"{\usebox{\boxDF}}"] \\
\\
\\
\usebox{\boxE} \arrow[rrrrrr, "{\usebox{\boxEF}}"] & & & & & & \usebox{\boxF} 
\end{tikzcd}.
\end{equation}
We claim that the commutativity of the maps in Equation \ref{eqn:RozQuasiIdea} implies the same in Equation \ref{eqn:consistentFTcone}.  Indeed since any planar cobordism $\phi_{i,j}$ can be decomposed into a sequence of saddles and single-dotted identity cobordisms, it is enough to note that both of these types of maps remain non-zero after `closing' with $I_\gamma$.  Thus any sign discrepancy in the commutativity of Equation \ref{eqn:consistentFTcone} would persist in Equation \ref{eqn:RozQuasiIdea}, and so it is enough to show that the maps in Equation \ref{eqn:RozQuasiIdea} commute.

If we allow a slight abuse of notation whereby we omit the presence of concatenated identity cobordisms, we utilize the homotopies of Equation \ref{eqn:rho_gammarho_i} to see
\[
\sigma_j (\rho_j)^*(\phi_{i,j})^* ((\rho_i)^*)^{-1} \simeq \sigma_i(\rho_\gamma)^* (\phi_{i,j})^* ((\rho_\gamma)^*)^{-1}.
\] 
Now we use the fact that, since $\phi_{i,j}$ and $\rho_\gamma$ are tangle isotopies affecting disjoint parts of the planar diagrams involved ($\phi_{i,j}$ turns $\delta_i$ into $\delta_j$, while $\rho_\gamma$ untwists $\gamma\cdot\F_n$), $\phi_{i,j}^*$ and $\rho_\gamma^*$ commute with no sign issue and thus
\[
\sigma_i(\rho_\gamma)^* (\phi_{i,j})^* ((\rho_\gamma)^*)^{-1} \simeq \sigma_i(\phi_{i,j})^*(\rho_\gamma)^* ((\rho_\gamma)^*)^{-1} \simeq \sigma_i\phi_{i,j}^*.
\]
As in the proof of Corollary \ref{cor:Consistent Cone of Concats}, the fact that these are all one-term complexes means that all of these chain homotopies are actually equalities, and so the maps of Equation \ref{eqn:RozQuasiIdea} commute as desired.  All of this was true for any $\delta_i,\delta_j$ having $\th(\delta_i)=\th(\delta_j)=0$, and so we are done.
\end{proof}

Having Lemmas \ref{lem:ft Reid shifts} and \ref{lem:ft cone preserved} at our disposal, we can turn to the proof of Theorem \ref{thm:Fk to Fk+1}.  The strategy is to induct on $k$ as follows.  We view $\F_n^{k+1}$ as the concatenation $\F_n \cdot \F_n^k$ to view $\KC^*(\F_n^{k+1})$ as a multicone over the complex $C^*(\F_n^k)$.  We then use Corollary \ref{cor:Cone of Concats} to further simplify this multicone, resulting in our desired complex $C^*(\F_n^{k+1})$.  The homological degree formulas of Equation \ref{eqn:HomDeg in Multicone} will allow us to conclude that the top $2k$ degrees of $C^*(\F_n^{k+1})$ are determined by the top $2k$ degrees of $C^*(\F_n^k)$ alone, which are inductively guaranteed to fit the requirements of the complex $\C^*$ in Lemma \ref{lem:ft cone preserved}.

\begin{proof}[Proof of Theorem \ref{thm:Fk to Fk+1}]
We define the complex $C^*(\F_n^k)$ inductively.  The base case of $k=1$ is already defined in Theorem \ref{thm:single ft cx}.  Viewing $\F_n^{k+1}$ as the concatenation $\F_n \cdot \F_n^k$, we inductively simplify $\KC^*(\F_n^k)$ into $C^*(\F_n^k)$, which plays the role of $\C^*$ from the statement of Corollary \ref{cor:Cone of Concats}:
\begin{equation}\label{eqn:basic multicone for Fk+1}
\KC^*(\F_n^{k+1}) \simeq \mCone{\delta_i}{\delta_j}{C^*(\F_n^k)} {\braidAA[.15cm]{\F_n}{\delta_i}} {\left(\braidAAmap[.15cm]{I_{\F_n}}{\phi_{i,j}} \right)} {\braidAA[.15cm]{\F_n}{\delta_j}}.
\end{equation}
The isotopies $\rho_i:(\F_n \cdot \delta_i) \xrightarrow{\cong} (\mathcal{X}'_i)$ are provided by Lemma \ref{lem:ft Reid shifts} and so Corollary \ref{cor:Cone of Concats} gives
\begin{multline*}
\KC^*(\F_n^{k+1}) \simeq \\
\mConeFullhqs{\delta_i}{\delta_j}{C^*(\F_n^k)} {\braidAAA[.2cm]{\topp{\delta_i}}{\F_{n-2m_i}}{\bott{\delta^i}}} {(\phi_{i,j}')} {\braidAAA[.2cm]{\topp{\delta_j}}{\F_{n-2m_j}}{\bott{\delta^j}}} {a_i}{b_i}{a_j}{b_j}
\end{multline*}
where the shifts $a_i,b_i,a_j,b_j$ can be computed from Lemma \ref{lem:ft Reid shifts} depending on $m_i,m_j$.  We now utilize Theorem \ref{thm:single ft cx}, which tells us that $\KC^*(\F_{n-2m_i})\simeq C^*(\F_{n-2m_i})$, to simplify our multicone further.  The resulting complex is our definition for $C^*(\F_n^{k+1})$ which satisfies item \eqref{it:equiv to FTk} of Theorem \ref{thm:Fk to Fk+1} by construction.  If we allow a slight abuse of notation, the general statement of Proposition \ref{prop:gen multicone equiv} allows us to write $C^*(\F_n^{k+1})$ as
\begin{multline*}
\KC^*(\F_n^{k+1}) \simeq C^*(\F_n^{k+1}):= \\
\mConeFullNOKChqs{\delta_i}{\delta_j}{C^*(\F_n^k)} {\braidAAA[.33cm]{\topp{\delta_i}}{C^*(\F_{n-2m_i})}{\bott{\delta^i}}} {(\phi_{i,j}'')} {\braidAAA[.33cm]{\topp{\delta_j}}{C^*(\F_{n-2m_j})}{\bott{\delta^j}}} {a_i}{b_i}{a_j}{b_j}
\end{multline*}
where we have written $\phi_{i,j}''$ to indicate that our maps $\phi_{i,j}'$ may have changed yet again while utilizing the homotopy equivalences $\KC^*(\F_{n-2m_i})\simeq C^*(\F_{n-2m_i})$.

Now we investigate the homological grading in this multicone.  Since $\topp{\delta_i}$ and $\bott{\delta^i}$ have no crossings, any diagram $\epsilon\in C^*(\F_n^{k+1})$ coming from $\delta_i\in C^*(\F_n^k)$ with $\th(\delta_i)=m_i$ can actually be viewed as coming from the corresponding $C^*(\F_{n-2m_i})$ in the simplified multicone.  In this way the homological degree of such an $\epsilon$ can be computed (with the help of Lemma \ref{lem:ft Reid shifts}) as
\[\h_{C^*(\F_n^{k+1})}(\epsilon) = \h_{C^*(\F_{n-2m_i})}(\epsilon) + \h_{C^*(\F_n^k)}(\delta_i) + 2m_i(n-m_i).\]
Theorem \ref{thm:single ft cx} tells us that such an $\epsilon$ must be a diagram with $\th(\epsilon)=(n-2m_i)-2\ell$ for some number of `new' matchings $\ell$, and with homological degree
\[\h_{C^*(\F_{n-2m_i})}(\epsilon) \leq 2\ell((n-2m_i)-\ell)\]
which quickly yields the inequality
\begin{equation}\label{eqn:gen hom bound for FTk+1 in terms of FTk}
\h_{C^*(\F_n^{k+1})}(\epsilon) \leq \h_{C^*(\F_n^k)}(\delta_i) + 2(m_i+\ell)(2p-(m_i+\ell)).
\end{equation}

We first show that $C^*(\F_n^{k+1})$ satisfies item \eqref{it:th deg zero} of the statement of the theorem.  Taking the shifts into account, this is equivalent to showing that $\h_{C^*(\F_n^{k+1})}(\epsilon)\leq 2(k+1)p^2$ and that $\th(\epsilon)\geq2$ implies $\h_{C^*(\F_n^{k+1})}(\epsilon)\leq 2(k+1)(p^2-1)$.

The first of these two inequalities follows from Equation \ref{eqn:gen hom bound for FTk+1 in terms of FTk} almost immediately using the inductive hypothesis together with the fact that $x(2p-x)$ is maximized when $x=p$:
\begin{align*}
\h_{C^*(\F_n^{k+1})}(\epsilon) &\leq 2kp^2 + 2p^2\\
 &= 2(k+1)p^2.
\end{align*}
The second inequality follows by recognizing that $\th(\epsilon)\geq 2\Longrightarrow \th(\delta_i)\geq 2$, so that we may inductively conclude that $\h_{C^*(\F_n^k)}(\delta_i)\leq 2k(p^2-1)$ in this case.  From here, we note that $\th(\epsilon)\geq 2$ also indicates $x=m_i+\ell\leq p-1$ in our analysis of $x(2p-x)$, giving us
\begin{align*}
\th(\epsilon)\geq 2 \quad\Longrightarrow\quad \h_{C^*(\F_n^{k+1})}(\epsilon) &\leq 2k(p^2-1) + 2(p-1)(p+1)\\
&= 2(kp^2-k+p^2-1)\\
&= 2(k+1)(p^2-1).
\end{align*}

Finally, to verify item \eqref{it:k to k+1}, we begin by considering the case when $\delta_i$ was \emph{not} in the top $2k$ homological degrees of $C^*(\F_n^k)$, ie
\[\h_{C^*(\F_n^k)}(\delta_i)\leq 2kp^2-2k.\]
In this case, Equation \ref{eqn:gen hom bound for FTk+1 in terms of FTk} quickly yields
\begin{align}
\h_{C^*(\F_n^{k+1})}(\epsilon) &\leq 2kp^2 -2k + 2p^2 \notag \\
&= 2(k+1)p^2 - 2k
\label{eqn:hom bound for FTk=FTk+1 pf}
\end{align}
where we've again used the fact that $x(2p-x)$ is maximized when $x=p$.

The homological bound in Equation \ref{eqn:hom bound for FTk=FTk+1 pf} applies only to the terms in the multicone coming from $\delta\in C^*(\F_n^k)$ with $h_{C^*(\F_n^k)}(\delta)\leq 2kp^2-2k$.  As such, if we are interested in the truncated complex $C^*_{> 2(k+1)p^2-2k}(\F_n^{k+1})$, it is enough to consider the multicone of Equation \ref{eqn:basic multicone for Fk+1} only for $\delta_i,\delta_j\in C^*_{> 2kp^2-2k}(\F_n^k)$.  By inductive assumption, the complex $\bar{\C}^*:=C^*_{> 2kp^2-2k}(\F_n^k)$ consists of only diagrams $\delta$ with $\th(\delta)=0$ and no disjoint circles.  For this reason, Lemma \ref{lem:ft cone preserved} applies when we simplify the truncated Equation \ref{eqn:basic multicone for Fk+1}:
\begin{align*}
\KC^*_{> 2(k+1)p^2-2k}(\F_n^{k+1})&\simeq \mCone{\delta_i}{\delta_j}{\bar{\C}^*} {\braidAA[.15cm]{\F_n}{\delta_i}} {\left(\braidAAmap[.15cm]{I_{\F_n}}{\phi_{i,j}}\right)} {\braidAA[.15cm]{\F_n}{\delta_j}} \\
&\simeq \h^{2p^2}\q^{2p(p+1)}\bar{\C}^*
\end{align*}
implying that 
\[C^*_{> 2(k+1)p^2-2k}(\F_n^{k+1}) = \h^{2p^2}\q^{2p(p+1)} \bar{\C}^*.\]
We then apply an overall shift of $\h^{-2(k+1)p^2}\q^{-2(k+1)p(p+1)}$ to both sides of this equivalence, shifting the designated homological truncation points according to our notational convention (see Definition \ref{def:notation list}), to verify item \eqref{it:k to k+1} in the statement of the theorem.
\end{proof}

\subsection{Proof of Theorem \ref{thm:inf twist complex}}
\begin{proof}[Proof of Theorem \ref{thm:inf twist complex}]
We define the complexes $C^*(k)$ to be shifted copies of the complexes $C^*(\F_n^k)$ of Theorem \ref{thm:Fk to Fk+1}:
\[\KCsimp(k):= \h^{-2kp^2}\q^{-2kp(p+1)} C^*(\F_n^k),\]
which are then chain homotopy equivalent to further shifts of the Khovanov complexes $KC^*(\F_n^k)$ as required (see the renormalization of Definition \ref{def:n- convention}).  Item \eqref{it:k to k+1} of Theorem \ref{thm:Fk to Fk+1} then guarantees that our truncated complexes include into each other
\[\KCsimp_{> -2}(1) \hookrightarrow \KCsimp_{>-4}(2) \hookrightarrow \cdots \hookrightarrow \KCsimp_{>-2k}(k)\hookrightarrow \cdots\]
as required.  Each complex in this sequence of inclusions satisfies item \eqref{it:th deg zero} of Theorem \ref{thm:Fk to Fk+1}, and therefore the limiting complex does as well.  Because the maps are inclusions, any truncation of the limiting complex matches the corresponding truncation in the sequence of inclusions, and so can be computed by simplifying and truncating a Khovanov complex of finite twists.
\end{proof}

Theorem \ref{thm:inf twist complex} gives us access to a well-defined chain complex in Bar-Natan's categorification of the Temperley-Lieb algebra to assign to an infinite twist on $n=2p$ strands.  Any finite set of chain terms and differentials for this complex can be computed, up to a grading shift, as the upper-most homological degrees of (a chain homotopy equivalent simplification of) $KC^*(\F_n^k)$ for large enough $k$.  If we imagine that our infinite twist (or perhaps, many infinite twists on various sets of strands) are closed up in some way as in Figure \ref{fig:L in Mr}, then in fact we have Khovanov complexes of genuine links (with high but \emph{finite} amounts of twisting), and we can pass to actual Khovanov chain groups and homology groups purely from Bar-Natan's setting (ie, without recourse to Khovanov's invariant for tangles \cite{Khov2}).  This is the plan for the following sections.

\begin{remark}\label{rmk:cat JW proj}
As described in Section \ref{sec:WRT}, our complex $C^*(\F_n^\infty)$ has graded Euler characteristic recovering the zero-weight Jones-Wenzl projector $P_{n,0}$.  One can also show that $C^*(\F_n^\infty)$ is an idempotent complex by an argument similar to the proof of Theorem \ref{thm:Fk to Fk+1}: one copy of $C^*(\F_n^\infty)$ is comprised of through-degree zero terms, all of which unwind any finite approximation of the second copy of $C^*(\F_n^\infty)$.  However, following \cite{Roz} we hesitate to call this complex a categorified projector without a similar construction for the other $P_{n,m}$ and a categorical analogue of how the various projectors give a decomposition of identity.

On the other hand, Ben Elias and Matt Hogancamp have a separate construction in \cite{CatDiagFT} for categorifying all of the Jones-Wenzl projectors together via categorical diagonalization \cite{CatDiag} of the full twist operator in the category of Soergel bimodules, and it seems likely that our construction here matches theirs for a specific choice of eigenmap.  See also \cite{CoopHog}.
\end{remark}

\section{Defining Khovanov homology for links in $\#^r(S^2\times S^1)$}
\label{sec:Defining general Kh}
\subsection{Links in $\#^r(S^2\times S^1)$}
Let $M^r=\#^r(S^2\times S^1)$.  A link $\L$ in $M^r$ with $\ell$ components is a (smooth) embedding $\coprod_\ell S^1 \hookrightarrow M^r$.  As discussed in the intro, we build $M^r$ by performing $r$ $S^0$-surgeries on $S^3$.  We then draw $M^r$ as a copy of $S^3$ together with its embedded attaching spheres $(S^0\times S^2)_i$ for $i=1,\dots,r$.  Each single pair $(S^0\times S^2)_i$ is drawn as two oppositely oriented spheres with a dashed line (called a \emph{surgery line}, and denoted by $\sl_i$) drawn between them.  Rather than trying to envision the handle $(D^1\times S^2)_i$ that connects these two spheres in $M^r$, we imagine the spheres as teleportation points - a traveler in our diagram for $M^r$ who touches one sphere is immediately teleported to the `mirror image' point on the corresponding sphere (coordinates $(x,y,z)$ from the center of one sphere correspond to coordinates $(x,y,-z)$ from the center of the other).  In this way, an oriented link $\L\subset M^r$ can be drawn as a planar diagram $L$ with crossings as usual, except that the link is allowed to intersect any of the $(S^0\times S^2)_i$ in mirrored points.

To analyze such link diagrams in detail, we fix our diagram for $M^r$ by placing our attaching spheres $(S^0\times S^2)_i$ and surgery lines $\sl_i$ so that the planar projection $P'$ of $M^r$ is as shown in Figure \ref{fig:P' def}.

\begin{figure}
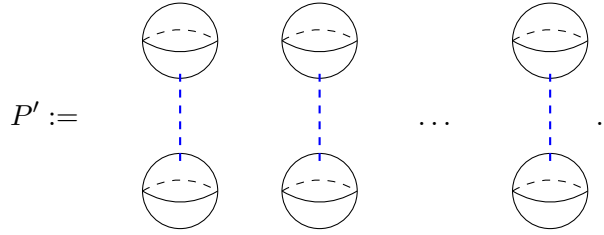

\[P' := \hspace{.25in} \SingleSurgSph \hspace{.25in} \SingleSurgSph \hspace{.25in} \dots \hspace{.25in} \SingleSurgSph.\]
\caption{$P'$ denotes the fixed diagram for the planar projection of $M^r$.  Our diagrams for links in $M^r$ will be drawn in $P'$.}
\label{fig:P' def}
\end{figure}

\begin{definition}\label{def:L standard position}
We say a link $\L\subset M^r$ is in \emph{standard position} if the following conditions are satisfied for each $i=1,\dots,r$:
\begin{itemize}
\item $\L\cap\sl_i=\emptyset$, that is, $\L$ does not intersect any of the surgery lines.
\item $\L\pitchfork\SSi$, that is, $\L$ intersects each attaching sphere $\SSi$ transversely.
\item The intersection points $\L\cap\SSi$ are arranged in a `straight' line on each sphere as indicated in the far right diagram of Figure \ref{fig:mirror points st pos}.  This arrangement shall also be referred to as \emph{standard position} for the intersection points.
\end{itemize}
Given such an $\L$, we let $n_i$ denote the geometric intersection number of $\L$ with the belt sphere of the handle $(D^1\times S^2)$, so that $\L\cap\SSi$ is the collection of $n_i$ pairs of mirrored points on the two spheres.
\end{definition}

The first two conditions of Definition \ref{def:L standard position} are clearly generic.  Furthermore, Figure \ref{fig:mirror points st pos} shows how we can further isotope any generic $\L$ into standard position by moving the intersection points of $\L\cap\SSi$ one at a time towards their chosen targets (generically missing the other intersection points).  Of course, this involves choices that we will consider further below.


\begin{figure}
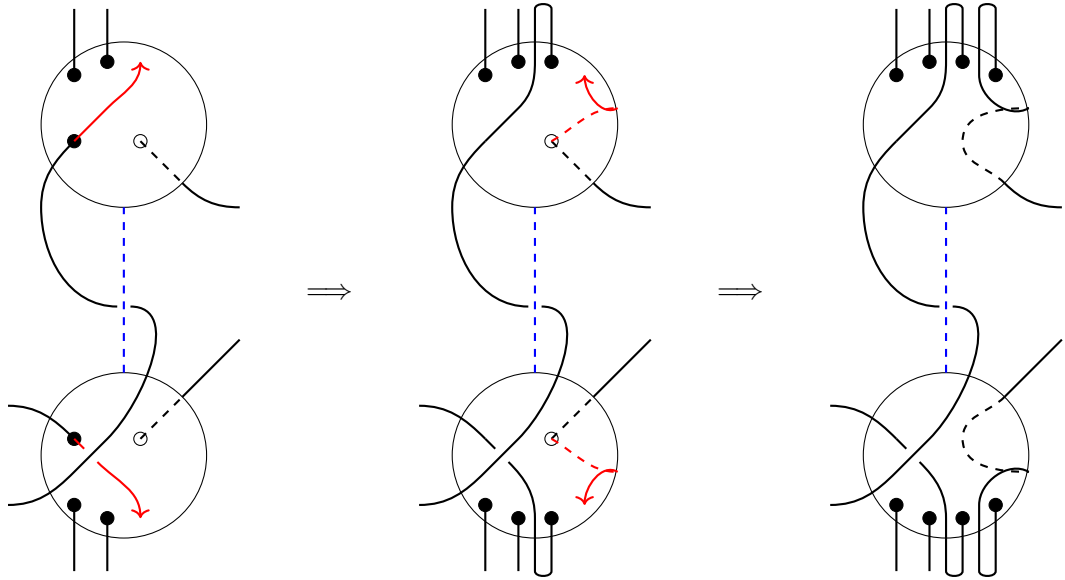

\[
\LintoStandardA \hspRarrow{.2in} \LintoStandardB \hspRarrow{.2in} \LintoStandardC
\]
\caption{We choose a path for each intersection point to follow along the sphere so that the mirrored points occur in a `straight' line, which we refer to as standard position for the points.  We then isotope $\L$ near each $(S^0\times S^2)_i$ by following along these paths, one at a time, and `tugging' the strand outward once the point has reached standard position, as illustrated in the sequence above.  The depths of the spheres are omitted for clarity.  The open circle indicates an intersection point on the `back' side of the sphere.  Note that there are many ordered sets of paths to arrange such points into standard position, but the resulting diagrams are all connected by isotopies of $\L$ in $M^r$ generated by the mapping class group of the $n_i$-punctured sphere.}
\label{fig:mirror points st pos}
\end{figure}

Given a link $\L\subset M^r$, we first isotope $\L$ so that it is in standard position, and then we project down to $P'$ and call the resulting diagram $L\subset P'$.  We then further isotope $\L$ as needed so that $L$ is also in general position.  That is, strands of $L$ miss the projections of the intersection points $\SSi\cap\L$ and $\SSi\cap\sl_i$, and the entire diagram contains no cusps or triple points of intersection involving any combination of strands of $L$, surgery lines $\sl_i$, and `borders' of attaching spheres $\SSi$.  We record over- and under-crossing data between strands of $L$ and/or surgery lines $\sl_i$ in the usual way; we also record whether or not strands of $L$ are `over' (respectively `under') any attaching sphere with solid (respectively dashed) lines.  Then $L$ is called a \emph{link diagram} for $\L$.


\begin{proposition}\label{prop:isotopic links}
Two link diagrams $L_1$ and $L_2$ drawn in $P'$ depict isotopic links in $M^r$ if and only if $L_1$ and $L_2$ are related by planar isotopies in $P'$, together with the following types of moves:
\begin{itemize}
\item \emph{Reidemeister moves} in $P'$, involving crossings of $L$ alone or crossings of $L$ with the fixed surgery lines and sphere borders.  See Figures \ref{fig:sphere Reid2 moves} and especially \ref{fig:sphere Reid3 moves}.
\item \emph{Mirror moves} - given any element $\beta$ in the braid group on $n_i$ strands, a copy of $\beta$ can be inserted adjacent to one sphere in the pair $(S^0\times S^2)_i$, together with a copy of $\beta^{-1}$ adjacent to the other sphere (see Figure \ref{fig:mirror move}).
\item \emph{Finger moves} - see the local picture of Figure \ref{fig:finger move}; finger moves at other points of the sphere can be moved to this point, up to possible braid group elements as described above.
\item \emph{Point-pass moves} - see the local picture of Figure \ref{fig:point-pass move}, where strands can pass either above or below any intersection point $\L\cap\SSi$ or $\sl_i\cap\SSi$ in accordance with whether the strand is above or below the sphere.
\item \emph{The surgery-wrap move} - we allow a disjoint nearby strand of $L$ to `wrap around' any dashed surgery-line, as illustrated in Figure \ref{fig:surgery-wrap move}.
\end{itemize}
\end{proposition}

\begin{figure}
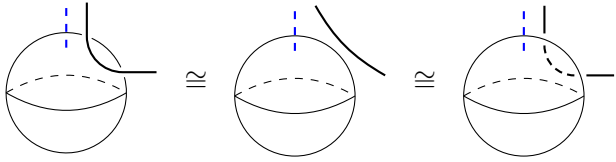

\[
\SphereReidTwoA \cong \SphereReidTwoB \cong \SphereReidTwoC
\]
\caption{Reidemeister 2 moves involving $L$ and sphere borders.  The intersection points $\L\cap\SSi$ are omitted for clarity.}
\label{fig:sphere Reid2 moves}
\end{figure}

\begin{figure}
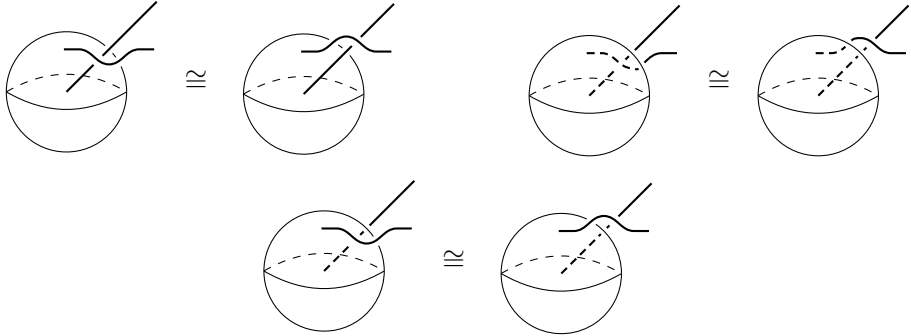

\[
\SphereReidThreeA \cong \SphereReidThreeB \hspace{.5in} \SphereReidThreeC \cong \SphereReidThreeD
\]
\[
\SphereReidThreeE \cong \SphereReidThreeF
\]
\caption{Some examples of Reidemeister 3 moves involving $L$ and sphere borders are shown here.  The intersection points $\L\cap\SSi$ and the surgery line $\sl_i$ are omitted.  Note that, since Reidemeister 3 moves always involve moving an `entire crossing' either above or below a third strand, we will always maintain which strands are over/under the sphere.}
\label{fig:sphere Reid3 moves}
\end{figure}

\begin{figure}
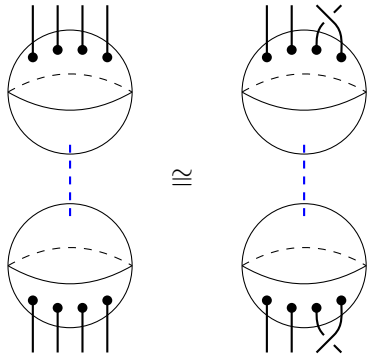

\centering
\mirrormoveA $\cong$ \mirrormoveB
\caption{An example of a mirror move, with $\beta=\sigma_3$.}
\label{fig:mirror move}
\end{figure}

\begin{figure}
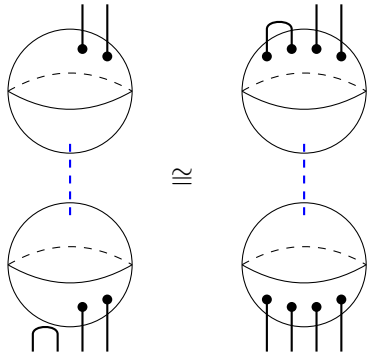

\centering
\fingermoveA $\cong$ \fingermoveB
\caption{The finger move, by definition.  We envision pushing the matching through the attaching sphere and out to the other side.  Naturally, there is a similar move given by reflecting this picture about the horizontal axis.}
\label{fig:finger move}
\end{figure}

\begin{figure}
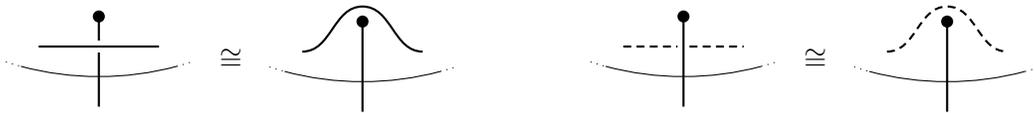

\[
\PointPassA \cong \PointPassB \hspace{.5in} \PointPassC \cong \PointPassD
\]
\caption{The point-pass moves for $\L\cap\SSi$, by definition.  A strand over (respectively under) the sphere can pass over (respectively under) any intersection point in $\L\cap\SSi$.  There are similar moves for $L$ passing over or under the points $\sl_i\cap\SSi$, as well as for the horizontal reflections of these pictures.}
\label{fig:point-pass move}
\end{figure}


\begin{figure}
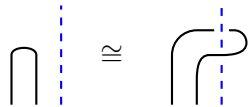

\centering
\disjointdash $\cong$ \dashwrap
\caption{The surgery-wrap move, by definition.}
\label{fig:surgery-wrap move}
\end{figure}

\begin{proof}
As usual, all of our choices in producing $L$ from $\L$ involved isotopies.  We begin with our choice of putting $\L$ into standard position, starting from the `linear' requirement for the intersection points.  This involves a choice of a path for each point of intersection towards a designated point on the sphere (missing the other intersection points).  Any two such choices are connected by an element of the mapping class group of the sphere.  Thus there are isotopies which miss the north pole which correspond to the braid group acting on the $n_i$ strands entering $\SSi$, and there are isotopies which `cross' the north pole and correspond to a wrapping-like move creating a Jucys-Murphy element $\E_{n_i}^{\pm 1}$.  Of course, such elements are themselves generated by the braid group, and so we need only consider braid-like isotopies of the $n_i$ strands.  These are generated by the mirror moves (since any isotopy on the surface of one sphere of the pair $\SSi$ must be mirrored on the other).

Tangencies of $\L$ with some $\SSi$ correspond to pushing a strand `through' the attaching sphere and out to the other side, which can be realized by a finger move up to a new choice of arranging the intersection points into standard position (so, up to further mirror moves).  Meanwhile, points of intersection of $\L\cap\sl_i$ correspond to passing a strand of $L$ `through' a surgery line $\sl_i$.  This move is generated by the surgery wrap move together with a Reidemeister 2 move as illustrated below:
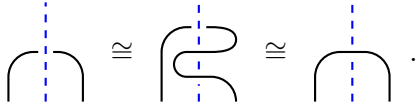
\[
\underdash \cong \underoverdashwrap \cong \overdash.
\]

We now analyze possible singularities in the planar diagram during an isotopy of $\L$.  Cusps, tangencies, and triple points of intersections of $L$ correspond to Reidemeister moves.  Having fixed our $\SSi$ and $\sl_i$, there are no cusps of sphere borders or surgery lines.  A triple point involving strand(s) of $L$, a sphere border, and/or a surgery line corresponds to a Reidemeister 3 move involving those objects.  A tangency between a strand of $L$ and a sphere border corresponds to a Reidemeister 2 move between them.  A tangency between a strand of $L$ and a single $\sl_i$ can correspond to either a Reidemeister 2 move, or a surgery-wrap move.  One might worry that there should be two versions of the surgery-wrap move corresponding to right-handed or left-handed twisting, but in fact in our unoriented diagrams these are equivalent via other Reidemeister moves:
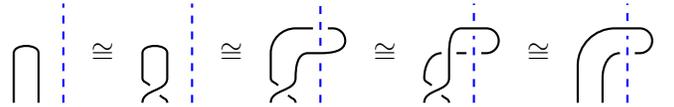
\[
\disjointdash \cong \negdashwrapB \cong \negdashwrapC \cong \negdashwrapD \cong \negdashwrapE .
\]
Of course this leads to two different versions once orientations are considered, which will lead to slightly different effects on the Khovanov homology further below.

Finally, the possibility of a strand of $L$ passing through an intersection point of $\L\cap\SSi$ or $\sl_i\cap\SSi$ corresponds to a point-pass move.  Notice that, since $\L$ may not pass through the sphere without appearing on the corresponding mirrored sphere, the move illustrated in Figure \ref{fig:non-point-pass} is disallowed (and similarly for the corresponding picture involving the point $\sl_i\cap\SSi$).

\end{proof}

\begin{figure}
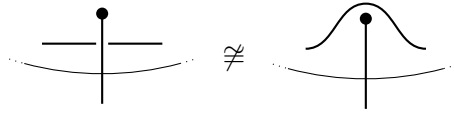

\[
\nonPointPass \not\cong \PointPassB
\]
\caption{This move, which may look similar to a point-pass move, is not allowed for link diagrams $L$ in $P'$.  It would require $\L$ to either pass through itself, or pass through the attaching sphere and appear on the corresponding mirrored sphere.  There is a similar disallowed move for a strand over the sphere passing `under' the point $\sl_i\cap\SSi$.}
\label{fig:non-point-pass}
\end{figure}


\begin{remark}\label{rmk:other wrapping moves}
Note that the moves listed in Proposition \ref{prop:isotopic links} allow for some perhaps more conventional moves that one might expect, such as various forms of `sphere-wrapping' that are combinations of point-passing and surgery-wrapping.  See Figure \ref{fig:other wrapping moves}.
\end{remark}

\begin{figure}
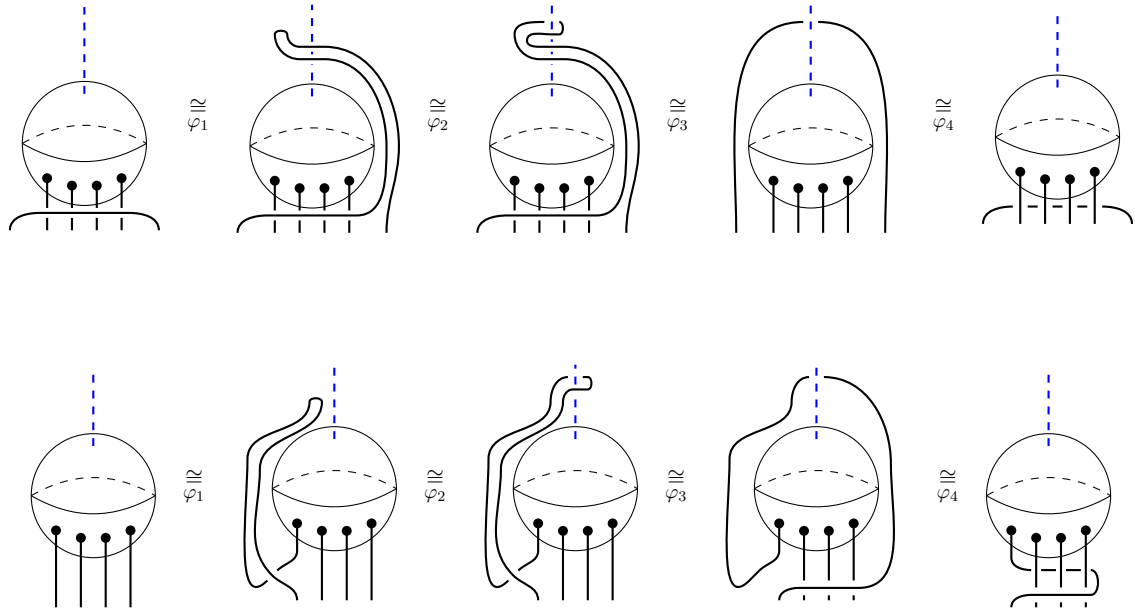

\[
\AltSpherePassA \stackrel{\cong}{\scalebox{.7}{$\varphi_1$}} \AltSpherePassB \stackrel{\cong}{\scalebox{.7}{$\varphi_2$}} \AltSpherePassC \stackrel{\cong}{\scalebox{.7}{$\varphi_3$}} \AltSpherePassD \stackrel{\cong}{\scalebox{.7}{$\varphi_4$}} \AltSpherePassE
\]
\vspace{.5in}
\[
\AltWrapA \stackrel{\cong}{\scalebox{.7}{$\varphi_1$}} \AltWrapB \stackrel{\cong}{\scalebox{.7}{$\varphi_2$}} \AltWrapC \stackrel{\cong}{\scalebox{.7}{$\varphi_3$}} \AltWrapD \stackrel{\cong}{\scalebox{.7}{$\varphi_4$}} \AltWrapE
\]
\caption{Some other `wrapping-style' moves that are allowable using combinations of the moves listed in Proposition \ref{prop:isotopic links}.  In both sequences above, the isotopy $\varphi_1$ is accomplished with basic Reidemeister moves, while $\varphi_2$ is a surgery-wrap move.  $\varphi_3$ uses further Reidemeister moves and point-pass moves for passing over the sphere, and finally $\varphi_4$ uses Reidemeister moves and point-pass moves for passing under the sphere.  Note the ability to create a single, non-mirrored Jucys-Murphy element $\E_{n_i}$, which is the version Rozansky considers in \cite{Roz}.}
\label{fig:other wrapping moves}
\end{figure}



\subsection{Defining the Khovanov homology $Kh(L)$ of a link diagram $L\subset P'$}
\label{sec:def KC(L)}
Given an oriented link $\L\subset M^r$, we can derive an oriented link diagram $L\subset P'$ as described in the previous section.  In order to ease the notation of some grading shifts later, we (arbitrarily) orient our surgery lines.  When reading over the following notations, it is useful to remember that our goal is to convert our diagram $L$ into the diagram $\Lk$ containing many full twists in place of the surgery lines (see Figure \ref{fig:L in Mr}).

\begin{definition}\label{def:notations for surgery strands}
We collect the following notations for quantities involving the various attaching spheres and surgery lines here as a single definition.  All of the following definitions are taken over $i=1,\dots,r$, and assume that we are given a specific link diagram $L\subset P'$ representing $\L\subset M^r$.
\begin{itemize}
\item $\sl_i$ will denote the $i^\text{th}$ surgery line.
\item $n_i=2p_i$ is the geometric intersection number between $L$ and the belt sphere in the $i^\text{th}$ handle $(D^1\times S^2)_i$.
\item $n_i^+$ and $n_i^-$ denote the number of positive (respectively negative) crossings in a \emph{single copy of the full twist} on $n_i$ strands, assuming they are oriented in the same way as the strands of $L$ entering the attaching spheres $(S^0\times S^2)_i$.  In the notation of Section \ref{sec:Infinite twist} (see Definition \ref{def:notation list}), $n_i^-:=n_{\F_{n_i}}^-$.
\item $N_i:= 2n_i^- - n_i^+$.
\item $\eta_i$ will denote the difference between the number of strands of $L$ that enter $(S^0\times S^2)_i$ with the same orientation as that of $\sl_i$, and the number of strands that enter with the opposite orientation as $\sl_i$.  Thus $\eta_i$ is the algebraic intersection number of $L$ with the $i^{\text{th}}$ belt sphere, oriented according to the orientation of $\sl_i$.
\end{itemize}
See Figure \ref{fig:oriented L in Mr} for clarification.  Note that swapping the orientation of $\sl_i$ also swaps the sign of $\eta_i$.
\end{definition}

\begin{figure}
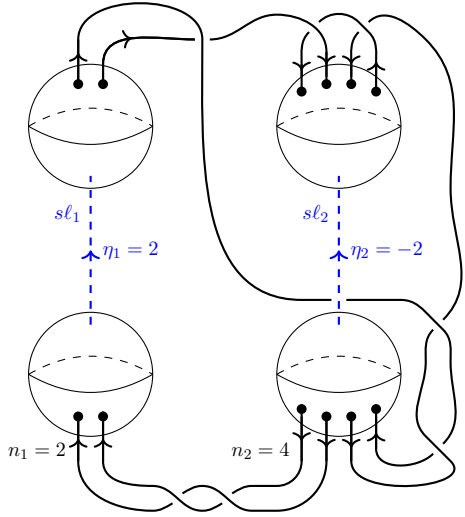

\centering
\LinMrNotations
\caption{The example of Figure \ref{fig:L in Mr} has been oriented, and the basic notations of Definition \ref{def:notations for surgery strands} are included.  The orientations of the $\sl_i$ are arbitrary; swapping such an orientation would swap the sign of the corresponding $\eta_i$.  The reader can consider oriented full twists in place of the surgery lines to verify that we have $n_1^+=2$ and $n_1^-=0$ so that $N_1=-2$, while $n_2^+=6$ and $n_2^-=6$ so that $N_2=6$.}
\label{fig:oriented L in Mr}
\end{figure}

\begin{definition}\label{def:Lk}
Given a link diagram $L\subset P'$ and a vector $\vec{k}=(k_1,\dots,k_r)\in \Z^r$, the symbol $\Lk$ denotes the oriented planar link diagram obtained from $L$ by performing the following steps:
\begin{enumerate}
\item Replace each $\sl_i$ with $n_i$ parallel copies of $\sl_i$ attached to $L$ at the intersection points of $L$ with $(S^0\times S^2)_i$.  These new strands are oriented to match the orientation of $L$.
\item Over/under crossing data for the new strands is inherited from the corresponding data from the $\sl_i$ and $\SSi$.  That is to say, if a strand of $L$ intersects a new strand at a point formerly along $\sl_i$, then $L$ is drawn over (respectively under) the new strand if $L$ was originally drawn over (respectively under) that point along $\sl_i$.  Similarly, if $L$ intersects a new strand at a point within the projection of $\SSi$, than $L$ is drawn over (respectively under) the new strand if $L$ was originally drawn over (respectively under) that sphere.
\item For each $i=1,\dots,r$ choose a point along $\sl_i$ that is not a crossing point, and insert the braid $\F_{n_i}^{k_i}$ into the set of $n_i$ parallel strands at that point.  Note that, in the case of $\vec{k}=\vec{0}$, this step does nothing.
\item Replace $P'$ with the standard plane $P$ (ie, erase or ignore the surgery spheres in the diagram).
\end{enumerate}
See Figure \ref{fig:L in Mr} for clarification.  We view $\Lk\subset P$ as a planar diagram for a link in $S^3$.
\end{definition}


\begin{definition}\label{def:KC(Linf)}
Given a link diagram $L\subset P'$, we construct the \emph{Khovanov complex of $L$}, denoted $KC^*(L)$, using the following steps:
\begin{enumerate}
\item Replace $L$ by the planar diagram $\Lzero$, viewed as a link in $S^3$.
\item For each $i=1,\dots,r$ choose a point along $\sl_i$ that is not a crossing point, and insert the complex $C^*(\F_{n_i}^\infty)$ of Theorem \ref{thm:inf twist complex} at that point.
\item Take the planar algebraic tensor product of these complexes, together with the Khovanov complexes coming from the crossings already present in the diagram $\Lzero$, in the sense of the tangle canopolies of \cite{BN}.
\end{enumerate}
See Figure \ref{fig:KC(Linf)} for clarification.  In a slight abuse of notation via Definition \ref{def:Lk}, we can think of this complex as $\KCsimp(\Linf)$.  Analogously, we define the \emph{finite approximation complex} $\KCsimp(\Lk)$ to be the (finite) complex obtained by following the same procedure, but placing the finite complex $\KCsimp(k_i)$ at each insertion point instead of $C^*(\F_{n_i}^\infty)$.

From here we define the \emph{Khovanov homology of $L$} $Kh(L)$ from $KC^*(L)$ by applying Khovanov's functor to complexes of modules over the ground ring and taking homology as usual.
\end{definition}

\begin{proposition}\label{prop:fin approx}
Given a link diagram $L\subset P'$ and an arbitrary homological lower bound $a$, there exists some finite $\vec{k}$ such that the truncated Khovanov complex $KC^*_{\geq a}(L)=\KCsimp_{\geq a}(\Linf)$ is equal to the truncation of the finite approximation complex $\KCsimp_{\geq a}(\Lk)$.  In other words, $KC^*(L)$ can be approximated by the truncation of a finite complex in any given homological range.  See Figure \ref{fig:KC(Linf)} for clarification.
\end{proposition}

\begin{figure}
\begin{align*}
KC_{\geq a}^*\left( \scalebox{.8}{\orientedLinMr} \right) := KC_{\geq a}^* \left( \scalebox{.8}{ \orientedLkinMr{C^*(\F_2^\infty)}{C^*(\F_4^\infty)} }
\right)&\\
\vspace{.2in}
\\
= KC_{\geq a}^* \left( \scalebox{.8}{ \orientedLkinMr{\KCsimp(k_1)}{\KCsimp(k_2)} } \right)&\\
\\
\vspace{.2in}
\\
\simeq \KChq{-2k_1-2k_2}{-6k_1-6k_2}_{\geq b}^* \left( \scalebox{.8}{ \orientedLkinMr{C^*(\F_2^{k_1})}{C^*(\F_4^{k_2})} } \right)&
\end{align*}
\caption{A diagrammatic example illustrating Definition \ref{def:KC(Linf)} (the first line) and Proposition \ref{prop:fin approx} (the second line is the truncated finite approximation complex $\KCsimp_{\geq a}$ by definition) for the oriented link diagram $L$ of Figure \ref{fig:oriented L in Mr}.  The third line unpacks the homological shifts inherent in the definitions of the $\KCsimp(k_i)$ for this figure; note that such shifts also affect the truncation point.
}
\label{fig:KC(Linf)}
\end{figure}

Recall that $\KCsimp(k_i)$ is chain homotopy equivalent to the shifted Khovanov complex of $k_i$ full twists.  In particular, we can use Proposition \ref{prop:fin approx} to prove:
\begin{corollary}\label{cor:KC(Linf)=KC(Lk)}
Given a link diagram $L\subset P'$, the Khovanov homology $Kh^*(L)$ can be approximated in any finite range of homological degrees by the shifted Khovanov homology of a genuine link diagram $\Lk$ for some finite $\vec{k}$.  More precisely, given any homological bound $a$, there exists some $\vec{k}$ that depends on $a$ such that
\begin{equation}\label{eq:KC(Linf)=KC(Lk)}
Kh^*(L) \cong H^* \left( \KChq{\sum_i k_i(n_i^- - 2p_i^2)}{\sum_i k_i( N_i - 2p_i(p_i+1) ) } (\Lk) \right)
\end{equation}
in all homological gradings $*\geq a$.
\end{corollary}
\begin{proof}[Proof of both Proposition \ref{prop:fin approx} and Corollary \ref{cor:KC(Linf)=KC(Lk)}]
By definition, $KC^*(L)$ is a planar tensor product of homotopy colimits of inclusions of truncated complexes, which can be viewed as a large stable homotopy colimit of truncations of the tensor products.  That is to say, there must exist some starting value $\vec{k_0}$ and truncation points $c_i\rightarrow-\infty$ allowing us to build $KC^*(L)$ via a limit of inclusions
\[\KCsimp_{>c_0}(L(\vec{k_0})) \hookrightarrow \KCsimp_{>c_1}(L(\vec{k_0}+\vec{1})) \hookrightarrow \cdots\]
where we are choosing to take the colimit `diagonally' increasing each entry of $\vec{k}$ by one at each step (this is only for ease of notation).  Thus, given the finite homological bound $a$, we can find $c_i<a$ allowing $\vec{k_o}+\vec{i}$ to satisfy the requirements of Proposition \ref{prop:fin approx}.

In this way we have $Kh^*(L) = H^*(\KCsimp(\Lk))$ for $*>a$.  Since each $\KCsimp(k_i)$ within $\KCsimp(\Lk)$ is chain homotopy equivalent to $\KChq{k_i(n_i^- - 2p^2)}{k_i(N_i - 2p_i(p_i+1))}^*(\F_{n_i}^{k_i})$ (Theorem \ref{thm:inf twist complex}), we have for the full tensor product
\[\KCsimp(\Lk) \simeq \KChq{\sum_i k_i(n_i^- - 2p^2)}{\sum_i k_i(N_i - 2p_i(p_i+1))}^*(\Lk).\]
Although these complexes may no longer be chain homotopic after truncation, their homology groups beyond the truncation point must remain isomorphic, which proves Corollary \ref{cor:KC(Linf)=KC(Lk)}.
\end{proof}

\subsection{Invariance of $Kh(\L)$ for links in $M^r$}\label{sec:general Kh invariance}
The goal of this section is to prove the following theorem.
\begin{theorem}\label{thm:Kh invariance}
Suppose $L_1$ and $L_2$ are two link diagrams in $P'$ representing isotopic links in $M^r$.  Then, up to grading shifts that depend only on the quantities $\eta_i$ for each surgery line, $Kh(L_1)\cong Kh(L_2)$.
\end{theorem}

According to Proposition \ref{prop:isotopic links}, it is enough to show the isomorphism whenever $L_1$ and $L_2$ are related by one of the move types listed there (together with showing invariance of our choice of insertion points along the $\sl_i$).  According to Corollary \ref{cor:KC(Linf)=KC(Lk)}, the isomorphism can be checked in any finite homological range for shifted Khovanov complexes of genuine links $\Lk[1]$ and $\Lk[2]$ in $S^3$.  Note that, although any given required $k_i$ may be different for the diagrams $L_1$ and $L_2$, we can always take the maximum and thus assume that the vector $\vec{k}$ is the same for $\Lk[1]$ and $\Lk[2]$.  Thanks to this viewpoint, many of the move types of Proposition \ref{prop:isotopic links} are nearly immediate consequences of isotopies of links in $S^3$.  The only type of move that requires closer examination is the surgery-wrap move.  Perhaps unsurprisingly, it is this type of move that causes the grading shifts as we will see below.  The strategy for proving invariance under this move is very similar to the strategy for proving Theorem \ref{thm:Fk to Fk+1}.  Compare the following lemma to Lemma \ref{lem:ft Reid shifts} from Section \ref{sec:equiv k to k+1}.

\begin{lemma}\label{lem:wrap Reid shifts}
If $\delta$ is a Temperley-Lieb diagram with $\th(\delta)=0$, we have
\begin{equation}\label{eqn:wrap Reid shifts}
\KC^* \left( \dashwrapA{$\delta$} \right) \simeq \KChqs{n_i}{n_i}^* \left( \disjointdashA{$\delta$} \right)
\end{equation}
regardless of the orientation of any of the strands.  The tangle isotopy that accomplishes this utilizes only Reidemeister 2 moves that remove crossings.
\end{lemma}
\begin{proof}
Since $\th(\delta)=0$, there is a clear innermost cup sliding isotopy of Reidemeister 2 simplifications as in Lemma \ref{lem:slide thru E} and the blue strands coming down from $\delta$ must come in oppositely oriented pairs.  From here, we count crossings as usual.  The details are left to the reader.
\end{proof}

In Section \ref{sec:equiv k to k+1}, Lemma \ref{lem:FT quasitrivial} provided an isotopy between two link cobordisms, implying that the chain maps induced by those two cobordisms were homotopic up to a sign.  This time around our two link cobordisms will not necessarily be isotopic as far as we can tell, but their `difference' up to isotopy will be a special sort of link cobordism whose induced chain map will be ignorable via the following lemma.

\begin{lemma}\label{lem:circle wrap is identity}
Let $\rho_\circ$ denote the tangle cobordism (embedded in $B^3\times[0,1]$) illustrated by the following movie of Reidemeister moves:
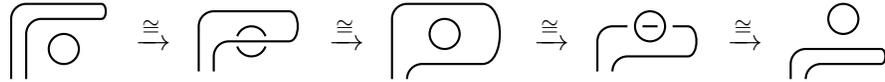
\[ \circlewrapIsoA \xrightarrow{\cong} \circlewrapIsoB \xrightarrow{\cong} \circlewrapIsoC \xrightarrow{\cong} \circlewrapIsoD \xrightarrow{\cong} \circlewrapIsoE\]
Then $\rho_\circ^*=I^*$, the identity map (ie the product cobordism in Bar-Natan's category of dotted cobordisms between planar diagrams) sending the disjoint circle to itself and the `wrapping strand' to itself.
\end{lemma}

\begin{figure}
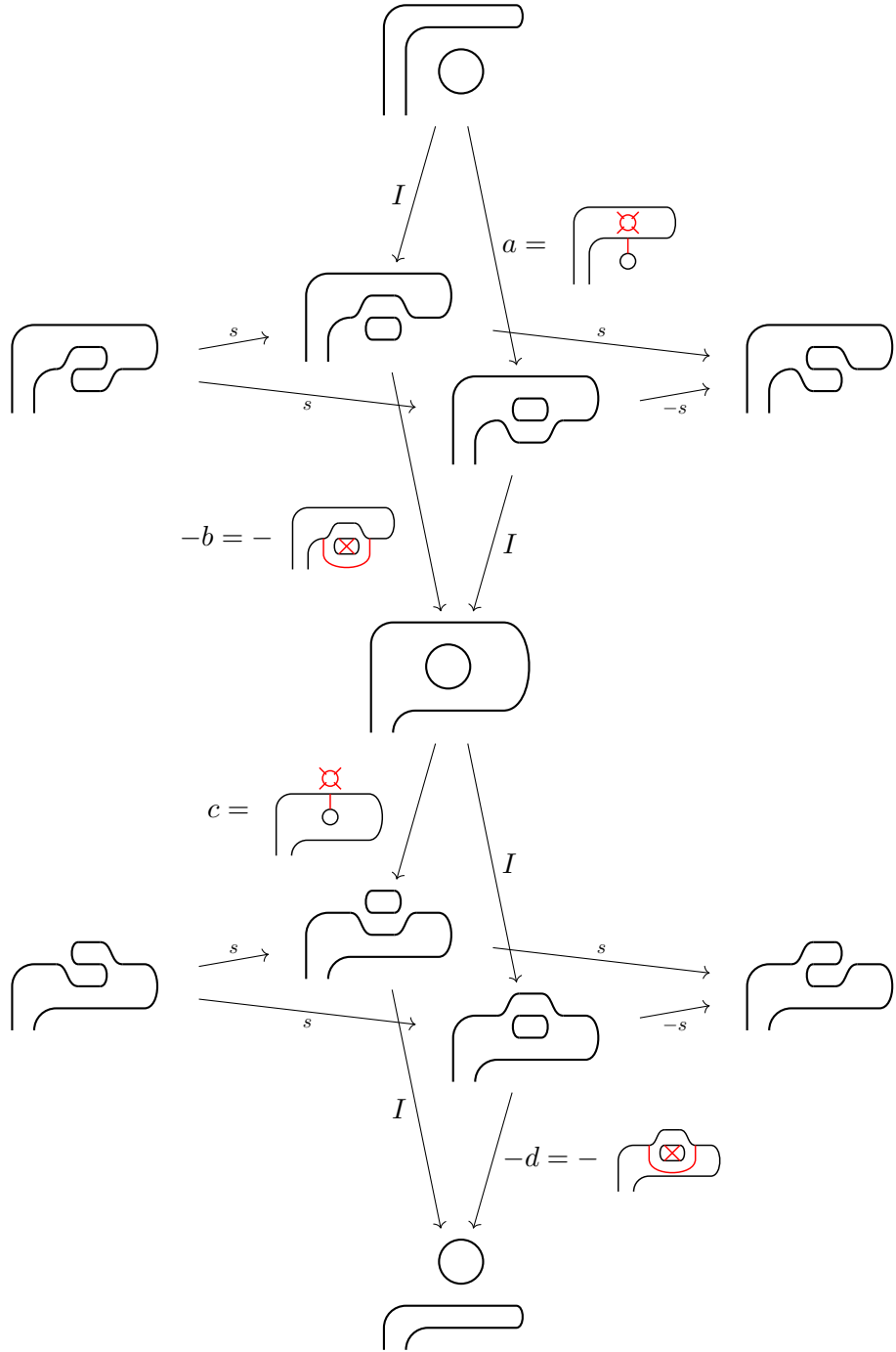

\centering
\circlewrapChainmaps
\caption{The chain maps corresponding to the Reidemeister movie $\rho_\circ$ of Lemma \ref{lem:circle wrap is identity}.  Identity maps are labelled $I$.  Saddle cobordisms are marked with red lines.  The maps $a$ and $c$ have circle births, while the maps $b$ and $d$ have circle deaths.}
\label{fig:circlewrapChainmaps}
\end{figure}

\begin{proof}
We write out the chain maps corresponding to the sequence of Reidemeister 2 moves indicated by $\rho_\circ$ in Figure \ref{fig:circlewrapChainmaps} (see \cite{BN} where these maps are defined and shown to be chain homotopy equivalences).  Utilizing the notation in that figure, together with the $\cdot$ notation for vertical concatenation as usual, the overall map $\rho_\circ^*$ can be computed as
\begin{align*}\rho_\circ^*&=(a-b)\cdot(c-d)\\
&= a\cdot c - a\cdot d - b\cdot c +b\cdot d.\end{align*}
The birth and death in $a \cdot d$ create a disjoint sphere, so that $a \cdot d=0$.  The other compositions are described below and expanded utilizing Bar-Natan's neck-cutting relation.
\begin{gather*}a\cdot c \quad = \quad \circlewrapACcobmap \quad = \quad \circlewrapDOTcobmap{1} + \circlewrapDOTcobmap{2}\\
-b\cdot c \quad = \quad -\circlewrapBCcobmap \quad = \quad -2 \circlewrapDOTcobmap{2}\\
b\cdot d \quad = \quad \circlewrapBDcobmap \quad = \quad \circlewrapDOTcobmap{2} + \circlewrapDOTcobmap{3}
\end{gather*}
After combining, the terms with a dot on the `folded curtain' all cancel, leaving a sum of two terms that is clearly equivalent (again via a neck-cutting relation) to an identity morphism taking the disjoint circle around to the other side of the curtain.
\end{proof}

\begin{lemma}\label{lem:wrap cone preserved}
If $\C^*$ is a chain complex made entirely of dotted cobordisms $\phi_{x,y}$ between diagrams $\delta_x,\delta_y$ all having through-degree zero \emph{and} having no disjoint circles, then abusing notation slightly we have
\begin{equation}\label{eqn:wrap cone preserved}
\mCone{\delta_x}{\delta_y}{\C^*} {\dashwrapA[.3cm]{$\delta_x$}} {\left(\dashwrapAmap{$\phi_{x,y}$}\right)} {\dashwrapA[.3cm]{$\delta_y$}} \simeq \h^{n_i}\q^{n_i}\left( \disjointdashA[.3cm]{$\C^*$} \right)
\end{equation}
regardless of the orientation of any of the strands.
\end{lemma}


\begin{proof}
The proof is identical in form to the proof of Lemma \ref{lem:ft cone preserved}, so we will be brief.  Because each of the $\delta_x$ have $\th(\delta_x)=0$ and no disjoint circles, Lemma \ref{lem:wrap Reid shifts} gives tangle isotopies
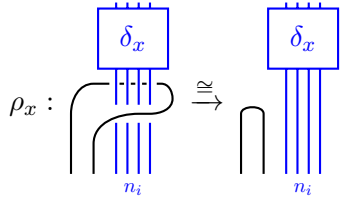
\[\rho_x: \dashwrapA[.3cm]{$\delta_x$} \xrightarrow{\cong} \disjointdashA[.3cm]{$\delta_x$}\]
which satisfy the requirements of Corollary \ref{cor:Consistent Cone of Concats}.  Meanwhile, after choosing a closure
\[\gamma:=
\begin{tikzpicture}[baseline={([yshift=-.7ex]current bounding box.center)},x=.25cm,y=.22cm]

\foreach \n in {3,5.5}
{
\draw[thick,blue]
(\n,0) -- (\n,-1) to[out=-90,in=-90] (\n+.5, -1) -- (\n+.5,0);
}
\draw[blue]
(4.6,-0.7) node[scale=.7] {\dots}
(4.5,.7) node[scale=.7] {$n_i$}
;
\end{tikzpicture}
\]
to concatenate on the \emph{bottom} of the strands coming from the $\sl_i$, we also have a consistent tangle isotopy
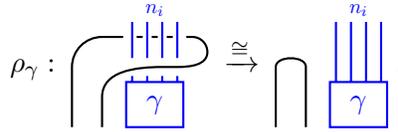
\[\rho_\gamma: \bottomlessdashwrapB{$\gamma$} \xrightarrow{\cong} \bottomlessdisjointdashB{$\gamma$}.\]
Just as in the proof of Lemma \ref{lem:FT quasitrivial}, there is an isotopy of cobordisms
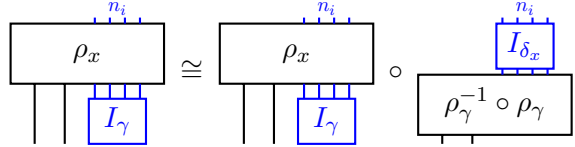
\[\bottomlesswrapABmap{$\rho_x$}{$I_\gamma$} \cong \bottomlesswrapABmap{$\rho_x$}{$I_\gamma$} \circ \bottomlesswrapBAmap{$I_{\delta_x}$}{$\rho_\gamma^{-1} \circ \rho_\gamma$}\]
where we consider the composition $(\rho_x\cdot I_\gamma) \circ (I_{\delta_x} \cdot \rho_\gamma^{-1})$ separately.  The reader can verify that, since this cobordism is taking place after $(I_{\delta_x} \cdot \rho_\gamma)$ which effectively removes any crossings from the local picture, our composition $(\rho_x\cdot I_\gamma) \circ (I_{\delta_x} \cdot \rho_\gamma^{-1})$ is isotopic to a sequence of tangle cobordisms of the form $\rho_\circ$ as in Lemma \ref{lem:circle wrap is identity}, all of which induce identity maps on the chain level.  This allows us to use functoriality to select a consistent choice of signs $\sigma_x$ such that 
\[\sigma_x\left(\bottomlesswrapABmap{$\rho_x$}{$I_\gamma$}\right)^* \simeq \left(\bottomlesswrapBAmap{$I_{\delta_x}$}{$\rho_\gamma$}\right)^*\]
leading to a corresponding version of Equation \ref{eqn:rho_gammarho_i} for use in the bottom two layers of Figure \ref{fig:Consistent Cone of Concats proof}.  Just as in the proof of Lemma \ref{lem:ft cone preserved}, we have $\rho_\gamma$ and $\phi_{x,y}$ commuting because they affect disjoint parts of the tangle diagram, and the fact that we are dealing with only one-term complexes ensures that the homotopies between maps are trivial.
\end{proof}

Note that we have adjusted the subscript notation of the diagrams to involve $x$ and $y$ instead of $i$ and $j$ as in Section \ref{sec:Infinite twist} because $i$ is now being used to track which surgery line $\sl_i$ we are considering.

\begin{remark}\label{rmk:JM not quasi-trivial?}
In \cite{Roz}, Rozansky's version of this proof using quasi-triviality boils down to the assertion that the two cobordisms $(\rho_x\cdot I_\gamma)$ and $(I_{\delta_x} \cdot \rho_\gamma)$ are themselves isotopic.  It seems unclear whether or not this is the case for these cobordisms in $B^3\times [0,1]$, which is where functoriality for Bar-Natan's construction has been established (even if we close the other strands and view our cobordisms in $S^3\times[0,1]$, this isotopy seems unclear to the author; the 2-sphere swept out by the `wrapping strand' would split $S^3$ into two 3-balls, but both 3-balls contain other strands of the link).  The computation of Lemma \ref{lem:circle wrap is identity} here provides the work-around: although $\rho_\circ$ may not be isotopic to identity as an embedded cobordism in $B^3\times[0,1]$, it still induces an identity map on the chain level.
\end{remark}

\begin{definition}\label{def:pos/neg wrapping move}
We call a surgery-wrap \emph{positive} (respectively \emph{negative}) if the crossings involved between $L$ and the surgery-line are positive
(respectively negative).
\end{definition}

\begin{theorem}\label{thm:wrapping move shift}
Suppose a link diagram $L_2$ is obtained from another link diagram $L_1$ in $P'$ by introducing a positive surgery wrap along $\sl_i$.  Then 
\begin{equation}\label{eq:pos wrap shift}
Kh^*(L_2)\cong \Khhq{\eta_i}{3\eta_i}^*(L_1).
\end{equation}
Similarly, if $L_3$ is obtained from $L_1$ by introducing a negative surgery wrap along $\sl_i$, we have
\begin{equation}\label{eq:neg wrap shift}
Kh^*(L_3)\cong \Khhq{-\eta_i}{-3\eta_i}^*(L_1).
\end{equation}
\end{theorem}
\begin{proof}
$L_1$ and $L_2$ are the same diagram everywhere except near a specific surgery line $\sl_i$, where we have the following local pictures:
\[L_2=\posdashwrap, \hspace{.5in} L_1 = \posdisjointdash.\]
According to Proposition \ref{prop:fin approx}, it is enough to produce a chain homotopy equivalence between all of the truncated complexes which approximate $KC^*(L_1)$ and $KC^*(L_2)$.  Allowing a slight abuse of notation once again, we must locally show that
\[ KC^* \left( \posdashwrapB{$\KCsimp_{>-2k_i}(k_i)$} \right)
\stackrel{?}{\simeq}
\h^{\eta_i} \q^{3\eta_i} \left( \posdisjointdashB{$\KCsimp_{>-2k_i}(k_i)$} \right)
\]
for any $k_i>0$.  We know from Theorem \ref{thm:inf twist complex} that the complexes $\KCsimp_{>-2k_i}(k_i)$ satisfy the conditions of Lemma \ref{lem:wrap cone preserved}, allowing us to write
\[
\KC^* \left( \posdashwrapB{$\KCsimp_{>-2k_i}(k_i)$} \right) \simeq \h^{n_i} \q^{n_i} \left( \posdisjointdashB{$\KCsimp_{>-2k_i}(k_i)$} \right).
\]
Recalling Definition \ref{def:n- convention}, this is equivalent to writing
\[
KC^* \left( \posdashwrapB{$\KCsimp_{>-2k_i}(k_i)$} \right) \simeq \h^{n_i-n^-} \q^{n_i-N} \left( \posdisjointdashB{$\KCsimp_{>-2k_i}(k_i)$} \right)
\]
where $n^-$ and $N$ count positive and negative crossings in the diagram on the left.  Notice that any blue strand that agrees with the orientation assigned to $\sl_i$ will contribute two positive crossings to this diagram by definition, while the disagreeing strands will contribute two negative crossings.  From here it is a simple exercise to verify that $n_i-n^- = \eta_i$ and $n_i-N=3\eta_i$.  The proof for the negative wrapping move is similar, except the roles of agreement and disagreement are reversed and so we negate $\eta_i$.

Notice that if we had chosen the opposite orientation for $\sl_i$, our positive and negative surgery wrap moves would swap, but we would also negate our $\eta_i$.  Thus our shifts are well-defined.
\end{proof}

\begin{proof}[Proof of Theorem \ref{thm:Kh invariance}]
We need to check invariance with respect to the move types listed in Proposition \ref{prop:isotopic links}, as well as changing the insertion point along any $\sl_i$.  Surgery-wrap moves have already been considered in Theorem \ref{thm:wrapping move shift}, causing chain homotopy equivalences on truncations up to grading shifts depending on the $\eta_i$.  For all of the other moves, we utilize Corollary \ref{cor:KC(Linf)=KC(Lk)} and check for chain homotopy equivalences between complexes of genuine links.  That is to say, if $L_2$ is obtained from $L_1$ by performing either a Reidemeister move, a mirror move, a finger move, or a point-pass move, then we wish to show that
\begin{equation}\label{eq:general invariance check}
\begin{multlined}
\KChq{\sum_i k_i(n_i^- - 2p_i^2)}{\sum_i k_i( N_i - 2p_i(p_i+1) ) } ^* (\Lk[1]) 
\\
\stackrel{?}{\simeq} \KChq{\sum_i k_i(n_i^- - 2p_i^2)}{\sum_i k_i(N_i - 2p_i(p_i+1) ) } ^*(\Lk[2])
\end{multlined}
\end{equation}
for any $\vec{k}$.  Notice that in Equation \ref{eq:general invariance check} we do not truncate the complexes.  Isotopies $\Lk[1]\cong\Lk[2]$ in $S^3$ will allow us to prove Equation \ref{eq:general invariance check} which implies isomorphisms of homology groups even after truncation, despite the fact that the truncated complexes may no longer be chain homotopy equivalent.  However, the shifts $p_i,n_i^-,N_i$ are somewhat deceptive here since each is taken to apply to its respective diagram (ie, in the first line $n_i^-=n_i^-(\Lk[1])$, while in the second line $n_i^-=n_i^-(\Lk[2])$).  The basic idea of the proof is to ensure that the overall shifts remain consistent during any necessary isotopy in $S^3$ of $\Lk[1]$ into $\Lk[2]$.  It is also necessary to ensure that the various $\eta_i$ remain constant as well, since the grading shift from surgery-wrapping uses this.  We will collectively refer to the set $\{p_i,n^-_i,N_i,\eta_i\}$ as the \emph{shifting data} for a given $\sl_i$ in a diagram.

We first note that, since we are now viewing $\Lk[1]$ and $\Lk[2]$ in $S^3$ and are ignoring the spheres $\SSi$, Reidemeister moves involving sphere borders are completely ignorable and correspond to simple planar isotopies.  In the case of Reidemeister moves involving crossings of $L_1$, we immediately get corresponding Reidemeister moves for $\Lk[1]$ in $S^3$ that do not affect any of the shifting data for any of the $\sl_i$.  Similarly, a Reidemeister move involving crossings of $L_1$ with some $\sl_i$ leads to corresponding `parallel' Reidemeister moves for $\Lk[1]$ in $S^3$ which again do not affect any of the shifting data.  Such a Reidemeister 2 move is illustrated below:
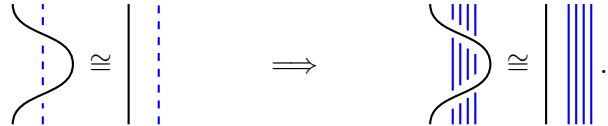
\[  \ReidExa \cong \ReidExb  \hspace{.5in} \Longrightarrow \hspace{.5in} \ReidExc \cong \ReidExd .\]
Moving an insertion point $\mathrm{pt}_i$ along $\sl_i$ leads to obvious isotopies in $S^3$ that `slide' the full twisting along the parallel strands, possibly over or under some other strands via many Reidemeister 3 moves:
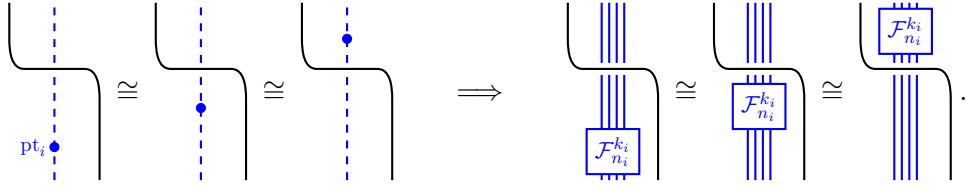
\[ \InsertionExa \cong \InsertionExb \cong \InsertionExc \hspace{.25in} \Longrightarrow \hspace{.25in} \InsertionExd \cong \InsertionExe \cong \InsertionExf .\]
Again it is clear that the shifting data remain constant between such diagrams.  A point-pass move likewise corresponds to a simple isotopy that maintains all shifting data where we slide a strand over or under the `gluing point' of the black and blue strands according to whether the strand was over or under the sphere.  Note that since we are ignoring the spheres in our diagrams for $\Lk[1]$ and $\Lk[2]$, we no longer draw the strands under the sphere as dashed lines.
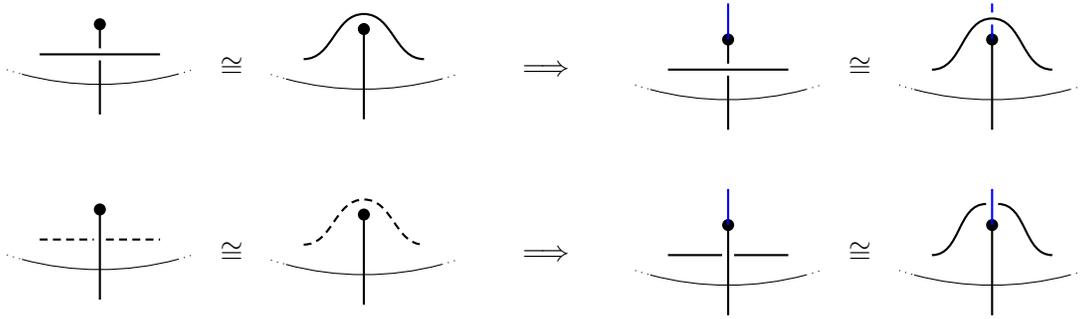
\begin{gather*}
\PointPassA \cong \PointPassB \hspRarrow{.2in} \PointPassIsoA \cong \PointPassIsoB\\
\\
\PointPassC \cong \PointPassD \hspRarrow{.2in} \PointPassIsoC \cong \PointPassIsoD
\end{gather*}

Mirror moves utilize the fact that the full twist is a central element of the braid group.  Thus any braid $\beta$ inserted adjacent to a surgery sphere can be `slid' upwards along the surgery line $\sl_i$, above or below any other crossings with $L$, passing through the full twists $\F_{n_i}^{k_i}$, until it meets its inverse braid $\beta^{-1}$ that was also inserted adjacent to the other sphere.  Once again, the shifting data clearly remains constant.  We illustrate the case below with one strand of $L$ crossing over the given $\sl_i$.
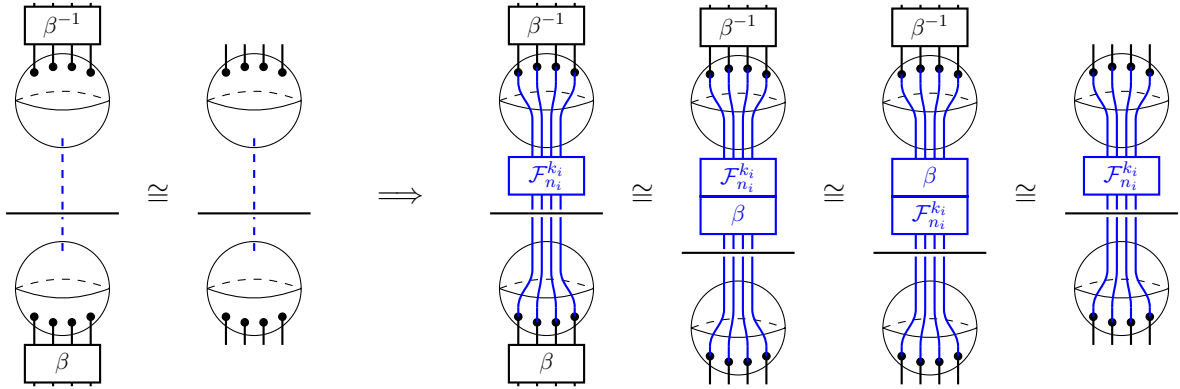
\[ \MirrorMoveIsoX \cong \MirrorMoveIsoY \hspRarrow{.2in}
\MirrorMoveIsoA \cong \MirrorMoveIsoB{\F_{n_i}^{k_i}}{\beta} \cong \MirrorMoveIsoB{\beta}{\F_{n_i}^{k_i}} \cong \MirrorMoveIsoC \]

The remaining move, which is slightly more interesting than the previous moves, is the finger move where some of the shifting data can change.  We present the invariance under this move in slightly more detail than the others.  First, it is clear from Figure \ref{fig:finger move} that, although the number of strands $n_i$ changes by two during a finger move, the two new strands must be oppositely oriented so that $\eta_i$ remains unchanged.  Next, we demonstrate the $S^3$ isotopy that performs the finger move on $\Lk$.  Let $\Lk[1]$ and $\Lk[2]$ be as indicated by the local pictures below, corresponding to the left and right diagrams from Figure \ref{fig:finger move} after inserting our full twists into a `topmost' point on the surgery line $\sl_i$:
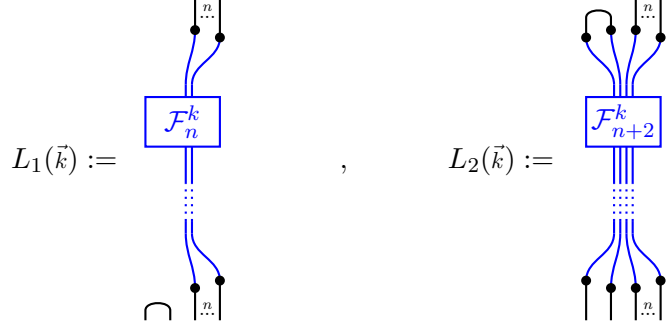
\[ \Lk[1] := \fingermoveLka \hspace{.5in} , \hspace{.5in} \Lk[2] := \fingermoveLkb
\]
where we have fixed $k:=k_i$ and $n:=n_i(\Lk[1])$, so that $n_i(\Lk[2]) = n+2$.  The symbols $\stackrel{n}{\mathrm{...}}$ are to indicate that there can be any (even) number of strands $n$ to start with.  The dotted lines in place of the blue strands are there to indicate that the link may have strands crossing over or under the parallel blue strands in that area.

There is a clear isotopy in $S^3$ from $\Lk[2]$ to $\Lk[1]$ by first pulling the matching through the full twist as in Figure \ref{fig:pull thru ft}, and then `sliding' the matching further down along the path of the surgery line (perhaps passing over or under other stands of $L$, similarly to the braid $\beta$ in a mirror move) until we arrive at $\Lk[1]$.  To find relationships in the shifting data before and after the isotopy, we count crossings in the same way as we did during the proof of Lemma \ref{lem:ft Reid shifts}.  Letting $n^-_{i,j}:=n^-_i(\Lk[j])$, and similarly for the other shifting data, we find the following conversions (recall from Definition \ref{def:notations for surgery strands} that $n^-_{i,j}$ denotes the number of negative crossings that occur in a \emph{single} copy of $\F_{n_i}$ in $\Lk[j]$, and similarly for $N_{i,j}$):
\begin{gather*}
p_{i,2} = p_{i,1}+1\\
n_{i,2}^- = n_{i,1}^- + 2n + 2\\
N_{i,2} = N_{i,1} + 2n + 4\\
\eta_{i,2} = \eta_{i,1}.
\end{gather*}
Compare this to the shifts of $2m(n-m)$ and $2m(n-m+1)$ of Lemma \ref{lem:ft Reid shifts} in the case of $m=1$ matching, but with $n+1$ used in place of $n$.  The reader can check that these shifts result in \emph{no} overall shift in either grading when checking Equation \ref{eq:general invariance check}, ie
\begin{gather*}
n_{i,2}^- - 2p_{i,2}^2 = n_{i,1}^- - 2p_{i,1}^2\\
N_{i,2} - 2p_{i,2}(p_{i,2}+1) = N_{i,1} - 2p_{i,1}(p_{i,1}+1)
\end{gather*}
and so the isotopy described above gives the required chain homotopy equivalence between complexes.

This finishes all of the moves required by Proposition \ref{prop:isotopic links}, and so we are done.
\end{proof}

\begin{proof}[Proof of Theorem \ref{thm:Main Thm}]
For a link $\L\subset M^r$ with diagram $L$ in $P'$, we have Khovanov homology groups $Kh(L)$ (Definition \ref{def:KC(Linf)}) which are well-defined up to grading shifts by Theorem \ref{thm:Kh invariance} allowing us to use the notation $Kh(\L)$.  In the case when $r=0$ and there are no attaching spheres or surgery lines to replace with infinite twists, our definition is clearly equivalent to that of the usual Khovanov homology of $\L\subset S^3$.
\end{proof}

\begin{remark}\label{rmk:fixed projection}
Our constructions and proofs have used a fixed projection $M^r\rightarrow P'$, with isotopies that fix the ambient $M^r$ while isotoping the link.  This is enough to construct a well-defined $Kh(\L)$ for links in a fixed copy of $M^r$, proving Theorem \ref{thm:Main Thm}.  It is also possible to show that the construction is invariant under moves that alter the projection, corresponding to moving the attaching spheres and surgery lines in the planar diagram $P'$.  The proofs are similar in spirit to the proof of Theorem \ref{thm:Kh invariance}.  We will investigate this further in Section \ref{sec:knotification} where knotifications for $r$-component links in $S^3$ naturally lead to a different planar projection of $M^r$.
\end{remark}

\section{Examples in $M^1=S^2\times S^1$ and $M^2$}
\label{sec:Examples}
In this section we present some computations of $Kh^*(L)$ over $\Z$ for some simple diagrams $L$ for links in $M^r$.  Throughout this section we will make use of the well-known complex for the infinite full twist on two strands:
\begin{equation}\label{eq:two strand cx}
C^*(\F_2^\infty) \simeq
\left(\cdots\rightarrow \q^{-5} \hres \xrightarrow{\hresdotT + \hresdotB} \q^{-3} \hres \xrightarrow{\hresdotT - \hresdotB} \underline{\q^{-1} \hres}
\right)
\end{equation}
where the pair of maps $(\hresdotT+\hresdotB)$ and $(\hresdotT-\hresdotB)$ repeat ad infinitum to the left and we have underlined the term in homological degree zero.  (Compare to the 2-strand projector in \cite{CK}, which has a fixed identity term on the far left; here, this term is pushed out to $-\infty$ since we are fixing the terms in maximal homological degree instead.)

\subsection{Two unlinked longitudes in $S^2\times S^1$}
\label{sec:two unlinked longitudes}
\begin{figure}
\[
L_1:=\LtwoLongitudes
\qquad \qquad \qquad
\begin{tikzpicture}[x=.8cm,y=-.4cm, baseline={([yshift=-.7ex]current bounding box.center)}]
\foreach \i/\j/\g in {0/0/\Z, 0/-2/\Z, -1/-2/\Z, -1/-4/{\Z_2}, -2/-6/\Z, -3/-6/\Z, -3/-8/{\Z_2}, -4/-10/\ddots, -4/2/\cdots, 1/-10/\vdots}
{
\node at (\i,\j){$\g$};
}

\foreach \i in {0,..., -3}
{
\draw (\i-.5,-10-.5) -- (\i-.5,3);
\node at (\i,2) {$\i$};
}
\foreach \j in {0,-2,-4,-6,-8}
{
\draw (-4.25,\j-1) -- (1.5,\j-1);
\node at (1,\j) {$\j$};
}
\draw[thick]
(.5,3) -- (.5,-10.5)
(-4.25,1) -- (1.5,1);

\draw[thick] (.5,1) -- (1.5,3);
\node[scale=.9] at (1-.2,2+.4) {$\h$};
\node[scale=.9] at (1+.2,2-.4) {$\q$};

\end{tikzpicture}
\]
\caption{The diagram $L_1$ along with the Khovanov homology groups $Kh(L_1)$ presented in a table with horizontal axis denoting homological grading and vertical axis denoting $q$-grading.  Because we are forming complexes that are `infinite on the left', the axes are drawn `backwards'.  The pattern of groups is 2-periodic for $\h$-grading below $0$ and for $\q$-grading `below' $0$.}
\label{fig:LtwoLongitudes}
\end{figure}
Let $L_1$ denote the diagram of Figure \ref{fig:LtwoLongitudes} for the link in $M^1=S^2\times S^1$ formed by taking two longitudes in $D^2\times S^1$.  Inserting the complex $\eqref{eq:two strand cx}$ in place of the surgery line results in a simple complex of the form
\[ 
\begin{tikzcd}
\cdots \arrow[r] & \q^{-9}\TLcircle \arrow[r,"{ 2\circledot }"] & \q^{-7}\TLcircle \arrow[r,"0"] & \q^{-5}\TLcircle \arrow[r,"{ 2\circledot }"] & \q^{-3}\TLcircle \arrow[r,"0"] & \underline{\q^{-1}\TLcircle}.
\end{tikzcd}
\]
Bar-Natan's delooping isomorphism converts this to
\[
\begin{tikzcd}
 & \q^{-10}\Z & \q^{-8}\Z & \q^{-6}\Z & \q^{-4}\Z & \underline{\q^{-2}\Z}\\
\cdots & \oplus & \oplus & \oplus & \oplus & \oplus\\
 & \q^{-8}\Z \arrow[ruu,"2"] & \q^{-6}\Z & \q^{-4}\Z \arrow[ruu,"2"] & \q^{-2}\Z & \underline{\q^{0}\Z}
\end{tikzcd}
\]
from which we can quickly compute the homology to be:

\[Kh^{i,j}(L_1)\cong \begin{cases}
\Z & \text{if $(i,j)=(0,0)$}\\
\Z & \text{if $(i,j)=(-2k+\epsilon,-2-4k)$ for $k\geq 0$, $\epsilon\in\{0,1\}$}\\
\Z_2 & \text{if $(i,j)=(-1-2k,-4-4k)$ for $k\geq 0$}
\end{cases}.
\]
The table is presented in Figure \ref{fig:LtwoLongitudes}.  Note that this is precisely the stable homology of the torus links $T(2,k)$ as $k\rightarrow\infty$.



\subsection{The Whitehead knot in $S^2\times S^1$}
\label{sec:Whitehead}
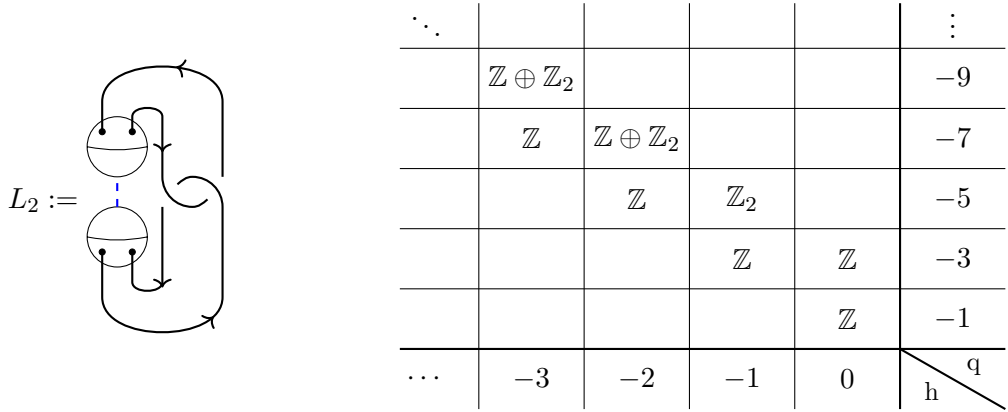
\begin{figure}
\[L_2:= \LWhitehead
\qquad \qquad \qquad
\begin{tikzpicture}[x=1.4cm,y=-.4cm, baseline={([yshift=-.7ex]current bounding box.center)}]
\foreach \i/\j/\g in {-3/-9/{\Z\oplus\Z_2}, -3/-7/\Z, -2/-7/{\Z\oplus\Z_2}, -2/-5/\Z, -1/-5/{\Z_2}, -1/-3/\Z, 0/-3/\Z, 0/-1/\Z, -4/-11/\ddots, 1/-11/\vdots, -4/1/\cdots}
{
\node at (\i,\j){$\g$};
}

\foreach \i in {0,..., -3}
{
\draw (\i-.5,-11-.5) -- (\i-.5,2);
\node at (\i,1) {$\i$};
}
\foreach \j in {-1,-3,-5,-7,-9}
{
\draw (-4.25,\j-1) -- (1.5,\j-1);
\node at (1,\j) {$\j$};
}
\draw[thick]
(.5,2) -- (.5,-11.5)
(-4.25,0) -- (1.5,0);

\draw[thick] (.5,0) -- (1.5,2);
\node[scale=.9] at (1-.2,1+.4) {$\h$};
\node[scale=.9] at (1+.2,1-.4) {$\q$};

\end{tikzpicture}
\]
\caption{The diagram $L_2$ along with the Khovanov homology groups $Kh(L_2)$ presented in a table with horizontal axis denoting homological grading and vertical axis denoting $q$-grading.  The pattern of groups is 1-periodic for $\h$-grading below $-1$ and for $\q$-grading `below' $-5$.}
\label{fig:LWhitehead}
\end{figure}

Let $L_2$ denote the diagram of Figure \ref{fig:LWhitehead} for the knot in $M^1=S^2\times S^1$ resulting from performing 0-surgery on one component of the Whitehead link in $S^3$ ($L_2$ can also be viewed as the knotification of the properly oriented Hopf link; see Section \ref{sec:knotification}).  We tensor the complex \eqref{eq:two strand cx} along the surgery line together with the complex
\[\h^{-2}\q^{-4}\left(\underline{\hres} \xrightarrow{\quad s \quad} \q\vres \xrightarrow{\vresdotL-\vresdotR} \q^3\vres \right)\]
(here $s$ denotes the usual saddle cobordism) coming from the two crossings already present in the diagram to arrive at the following complex
\[
\begin{tikzcd}
\cdots \arrow[r] & \q^{-13} \TLtwocircle \arrow[r,"w^+"] \arrow[d,"s"] & \q^{-11} \TLtwocircle \arrow[d,"s"] \arrow[r,"w^-"] & \q^{-9} \TLtwocircle \arrow[r,"w^+"] \arrow[d,"s"] & \q^{-7} \TLtwocircle \arrow[d,"s"] \arrow[r,"w^-"] & \q^{-5} \TLtwocircle \arrow[d,"s"]
\\
\cdots \arrow[r] & \q^{-12} \TLcircle \arrow[d,"0"] \arrow[r,"-2\circledot"] & \q^{-10}\TLcircle \arrow[d,"0"] \arrow[r,"0"] & \q^{-8}\TLcircle \arrow[d,"0"] \arrow[r,"-2\circledot"] & \q^{-6}\TLcircle \arrow[d,"0"] \arrow[r,"0"] & \q^{-4}\TLcircle \arrow[d,"0"]
\\
\cdots \arrow[r] & \q^{-10}\TLcircle \arrow[r,"2\circledot"] & \q^{-8}\TLcircle \arrow[r,"0"] & \q^{-6}\TLcircle \arrow[r,"2\circledot"] & \q^{-4}\TLcircle \arrow[r,"0"] & \underline{\q^{-2}\TLcircle}
\end{tikzcd}
\]
where $w^{\pm}$ denotes the map
\[w^\pm:= \twocircledotT \pm \twocircledotB.\]

If we use Bar-Natan's delooping isomorphism and split this complex in terms of $\q$-degree, we see the following collection of complexes:
\begin{itemize}
\item In $q$-degree $-1$ there is a single-term complex $\Z$ in homological grading zero whose homology is the same.
\item In $q$-degree $-3$ there is a complex of the form (still arranged as a tensor product of two complexes, with homological gradings constant along diagonals in the usual way)
\[
\left(
\begin{tikzcd}
 & \Z \arrow[d,"1"]\\
 & \Z\\
\Z & \underline{\Z}
\end{tikzcd}
\right).
\]
The homology is then easily computed to be
\[Kh^{i,-3}(L_2) \cong \begin{cases}
\Z & \text{if $i=-1,0$}\\
0 & \text{else}
\end{cases}.
\]
\item In $q$-degree $-5$ we have a complex of the form
\[
\left(
\begin{tikzcd}
 & \Z \arrow[d,"1"] \arrow[r,"\begin{pmatrix} 1 \\ -1 \end{pmatrix}"] & \Z\oplus\Z \arrow[d,"\big( 1 \quad 1 \big)"]
\\
 & \Z & \Z
\\
\Z \arrow[r,"2"] & \Z & 
\end{tikzcd}
\right)
\]
where the lower right (empty) corner has homological degree zero.  The homology can be computed after a Gaussian elimination of the vertical identity map to be
\[Kh^{i,-5}(L_2) \cong \begin{cases}
\Z & \text{if $i=-2$}\\
\Z_2 & \text{if $i=-1$}\\
0 & \text{else}
\end{cases}.
\]

\item In $q$-degrees $-4k-3$ for $k\geq 1$, we have complexes of the form
\[
\left(
\begin{tikzcd}[column sep=large]
 & \Z \arrow[d,"1"] \arrow[r,"\begin{pmatrix} 1 \\ 1 \end{pmatrix}"] & \Z\oplus\Z \arrow[d,"\big( 1 \quad 1 \big)"] \arrow[r,"\big( -1 \quad 1 \big)"] & \Z
\\
 & \Z \arrow[r,"-2"] & \Z &
\\
\Z & \Z & &
\end{tikzcd}
\right)
\]
where the lower right (empty) corner has homological degree $-2(k-1)$.  The homology can be computed after Gaussian elimination of the vertical identity and one of the right-most horizontal identities to be
\[Kh^{i,-4k-3}(L_2) \cong \begin{cases}
\Z \oplus \Z_2 & \text{if $i=-2k$}\\
\Z & \text{if $i=-2k-1$}\\
0 & \text{else}
\end{cases}.
\]

\item In $q$-degrees $-4k-1$ for $k\geq 2$, we have complexes of the form
\[
\left(
\begin{tikzcd}[column sep=large]
 & \Z \arrow[d,"1"] \arrow[r,"\begin{pmatrix} 1 \\ -1 \end{pmatrix}"] & \Z\oplus\Z \arrow[d,"\big( 1 \quad 1 \big)"] \arrow[r,"\big( 1 \quad 1 \big)"] & \Z
\\
 & \Z & \Z &
\\
\Z \arrow[r,"2"] & \Z & &
\end{tikzcd}
\right)
\]
where the lower right (empty) corner has homological degree $-2k+3$.  The homology can be computed after similar Gaussian eliminations to the previous case, giving the same homology
\[
Kh^{i,-4k-1}(L_2) \cong \begin{cases}
\Z \oplus \Z_2 & \text{if $i=-2k+1$}\\
\Z & \text{if $i=-2k$}\\
0 & \text{else}
\end{cases}.
\]
\end{itemize}
All of these homology groups are summarized neatly in the table of Figure \ref{fig:LWhitehead}.

\subsection{A crossingless knot in $M^2$}
\label{sec:crossingless}
\begin{figure}
\[L_3:= \Lcrossingless
\qquad \qquad
\begin{tikzpicture}[x=1.4cm,y=-.4cm, baseline={([yshift=-.7ex]current bounding box.center)}]
\foreach \i/\j/\g in {0/-1/\Z, 0/-3/\Z, -1/-3/\Z^2, -1/-5/\Z_2^2, -2/-5/\Z, -2/-7/\Z^2\oplus\Z_2, -3/-7/\Z^3, -3/-9/\Z\oplus\Z_2^3, -4/-9/\Z^2, -4/-11/\Z^3\oplus\Z_2^2, -5/-11/\Z^4, -5/-13/\ddots, -6/1/\cdots, 1/-13/\vdots}
{
\node at (\i,\j){$\g$};
}

\foreach \i in {0,..., -5}
{
\draw (\i-.5,-13.5) -- (\i-.5,2);
\node at (\i,1) {$\i$};
}
\foreach \j in {-1,-3,-5,-7,-9,-11}
{
\draw (-6.25,\j-1) -- (1.5,\j-1);
\node at (1,\j) {$\j$};
}
\draw[thick]
(.5,2) -- (.5,-13.5)
(-6.25,0) -- (1.5,0);

\draw[thick] (.5,0) -- (1.5,2);
\node[scale=.9] at (1-.2,1+.4) {$\h$};
\node[scale=.9] at (1+.2,1-.4) {$\q$};

\end{tikzpicture}
\]
\\
\[Kh^{i,j}(L_3) \cong \begin{cases}
\Z & \text{if $(i,j)=(0,-1)$}\\
\Z^{k+2} & \text{if $(i,j)=(-1-2k,-3-4k)$ for $k\geq 0$}\\
\Z^k & \text{if $(i,j)=(-2k,-1-4k)$ for $k\geq 1$}\\
\Z^k\oplus\Z_2^{k+2} & \text{if $(i,j)=(-1-2k,-5-4k)$ for $k\geq 0$}\\
\Z^{k+1}\oplus\Z_2^k & \text{if $(i,j)=(-2k,-3-4k)$ for $k\geq 0$}\\
0 & \text{else}
\end{cases}
\]

\caption{The diagram $L_3$ along with the Khovanov homology groups $Kh(L_3)$ presented in a table with horizontal axis denoting homological grading and vertical axis denoting $q$-grading.  The pattern is harder to see in this table, so the concrete formulas have been provided.  The second and third cases describe the diagonal $2i-j=1$, while the third and fourth cases describe the diagonal $2i-j=3$.}
\label{fig:Lcrossingless}
\end{figure}

Let $L_3$ denote the diagram of Figure \ref{fig:Lcrossingless} for the corresponding knot in $M^2$ We tensor together two copies of the complex \eqref{eq:two strand cx} along the two surgery lines together, noting that all resolutions available create a single circle, to build an infinite complex of the form
\[
\begin{tikzcd}
 & & & & & \vdots \arrow[d]\\
 & & & \ddots & & \q^{-10}\TLcircle \arrow[d,"2\circledot"]\\
 & & \ddots & & \q^{-10}\TLcircle \arrow[d,"0"] \arrow[r,"0"] & \q^{-8}\TLcircle \arrow[d,"0"]\\
 & \ddots & & \q^{-10}\TLcircle \arrow[d,"2\circledot"] \arrow[r,"2\circledot"] & \q^{-8}\TLcircle \arrow[d,"2\circledot"] \arrow[r,"0"] & \q^{-6}\TLcircle \arrow[d,"2\circledot"]\\
 & & \q^{-10}\TLcircle \arrow[d,"0"] \arrow[r,"0"] & \q^{-8}\TLcircle \arrow[d,"0"] \arrow[r,"-2\circledot"] & \q^{-6}\TLcircle \arrow[d,"0"] \arrow[r,"0"] & \q^{-4} \TLcircle \arrow[d,"0"]\\
\cdots \arrow[r] & \q^{-10}\TLcircle \arrow[r,"2\circledot"] & \q^{-8}\TLcircle \arrow[r,"0"] & \q^{-6}\TLcircle \arrow[r,"2\circledot"] & \q^{-4} \TLcircle \arrow[r,"0"] & \underline{\q^{-2} \TLcircle}
\end{tikzcd}.
\]
The reader may verify that, after delooping, there are no non-zero compositions of maps, so that the signs are irrelevant.  If we split the complex according to $\q$-degree as in the example with $L_2$, we find that the complex in any given $\q$-degree (other than the easy case of $\q$-degree $-1$) is supported in precisely two homological degrees, with a simple format that depends on the congruence class of the $\q$-degree modulo 4:
\begin{itemize}
\item For $\q$-degree $-1-4k$ with $k\geq 1$, $KC^*(L_3)$ is equivalent to a complex $\Z^{2k+1}\xrightarrow{d_{-1}}\Z^{2k}$ where $d_{-1}$ is a matrix of the form
\[d_{-1}=\begin{pmatrix}
2 & 0 & 0 & 0 & 0 & \cdots & 0\\
0 & \vdots & 2 & \vdots & 0 & & \vdots\\
\vdots & & 2 & & 0 & & \\
 & & 0 & & 2 & & \\
 & & \vdots & & 2 & & \\
 & & & & 0 & & \vdots \\
\vdots & & & & \vdots & \ddots & 0 \\
0 & \cdots & & & & & 2
\end{pmatrix}
\]
where each odd column (other than the first and last) has a $2$ on and directly above the `diagonal' as shown, and the rest of the entries are $0$.

\item For $\q$-degree $-3-4k$ with $k\geq 0$, $KC^*(L_3)$ is equivalent to a complex $\Z^{2(k+1)}\xrightarrow{d_{-3}}\Z^{2k+1}$ where $d_{-3}$ is a matrix of the form
\[
d_{-3} = \begin{pmatrix}
0 & \cdots & & & & & \cdots & 0\\
0 & 2 & 2 & 0 & \cdots & & & \vdots\\
0 & \cdots & & & & & & \\
0 & 0 & 0 & 2 & 2 & 0 & \cdots & \\
\vdots & & & & & & \ddots & \\
0 & \cdots & & & & & \cdots & 0
\end{pmatrix}
\]
where each even row has a $2$ on and directly to the right of the `diagonal' as shown, and the rest of the entries are $0$.

\end{itemize}
From here, it is a simple computation to show that the homology of such complexes fits into the formula and table of Figure \ref{fig:Lcrossingless}.

\section{Decategorification and odd intersection numbers with belt spheres} \label{sec:decat}
\subsection{The Kauffman bracket skein module}
Given a 3-manifold $M$, Hoste and Przytycki (see the survey \cite{HP} on this topic) study the skein module $\sk(M)$ generated over $R[q,q^{-1}]$ by isotopy classes of framed links in $M$, subject to local relations based on the Kauffman bracket $\langle\cdot\rangle$ (up to a renormalization depending on the orientation of the link):
\begin{equation}\label{eq:Kauffman relations}
\langle \crossing \rangle = \langle \vres \rangle - q \langle \hres \rangle,\qquad \langle L\sqcup U \rangle = (q+q^{-1})\langle L \rangle.
\end{equation}
In the case of $M=M^r$, our definition for Khovanov homology requires the use of semi-infinite complexes.  As such, in order to relate this to the skein module $\sk(M^r)$, we must allow for power series in our ground ring.  We let $\sk_\ell(M^r)$ denote the skein module taken over the ring $R[q,q^{-1}]]$ of formal Laurent series in $q$, with corresponding Kauffman bracket $\langle \cdot \rangle_\ell$.

In the following theorem, we call an oriented link $\L\subset M^r$ \emph{null-homologous} if $[\L]=0$ in $H_1(M^r)$.  Note that this is equivalent to all of the algebraic intersection numbers $\eta_i$ being zero, so our Khovanov homology groups are well-defined with no shifts incurred for performing surgery-wrap moves.

\begin{theorem}\label{thm:general skein module cat}
Given a framed, null-homologous link $\L\subset M^r$, with Khovanov homology groups $Kh^*(\L)$ as defined in Section \ref{sec:Defining general Kh}, the graded Euler characteristic $\chi_q( Kh^*(\L) )$ recovers the Kauffman bracket skein module element $\langle\L\rangle_\ell \in \sk_\ell(M^r)$ over the formal Laurent series ring $R[q,q^{-1}]]$ up to a renormalization depending on the orientation of $\L$.
\end{theorem}
\begin{proof}
Let $L$ be a diagram for the link $\L$.  Our definition for the Khovanov complex of a link in $M^r$ is based on the usual Khovanov complex away from the inserted full twists, while allowing for complexes that are infinite `to the left' whose graded Euler characteristic can give power series in $q^{-1}$.  The Kauffman bracket relations of Equation \ref{eq:Kauffman relations} are precisely the relations categorified by the usual Khovanov complex via the exact triangle implied by Equation \ref{eqn:KC cone def} and Bar-Natan's delooping isomorphism \cite{BN}.  It follows that, after adding kinks to $L$ as necessary to equate the given framing for $\L$ with the blackboard framing, $\chi_q(KC^*(L))$ satisfies all of the exact same relations as the Kauffman bracket and thus must match $\langle L \rangle_\ell$.
\end{proof}

Since our categorified construction passes through complexes of diagrams in $S^3$, which may be delooped into diagrams between empty links, we may rephrase Theorem \ref{thm:general skein module cat} as follows.
\begin{corollary}
The skein module $\sk_\ell^0(M^r)$ of isotopy classes of null-homologous framed links in $M^r$ (over $R[q,q^{-1}]]$) is isomorphic to a single free summand $R[q,q^{-1}]]$ generated by the empty link.
\end{corollary}

In the case of $M^1=S^2\times S^1$, Hoste and Przytycki show \cite{S1S2skein} that the $\sk(M^1)$ over $R[q,q^{-1}]$ is also isomorphic to a single free summand $R[q,q^{-1}]$ generated by the empty link, together with infinitely many torsion summands.  For a link $\L$, let $\langle \L \rangle_{free}\in R[q,q^{-1}]$ denote the free part of $\langle \L \rangle \in \sk(M^1)$.

\begin{corollary}\label{cor:polys in M1}
For a null-homologous $\L\subset M^1$, the graded Euler characteristic $\chi_q( Kh^*(\L) )$ recovers the polynomial $\langle \L \rangle_{free} \in R[q,q^{-1}]$.
\end{corollary}
\begin{proof}
The arguments in \cite{S1S2skein} show that, over $R[q,q^{-1}]$,
\[\langle \L \rangle = \langle \L \rangle_{free} + \mathrm{torsion}\]
where all of the torsion terms are annihilated by quantities of the form $(1-q^{2a})$.  The exact same arguments can be run in $\sk_\ell(M^1)$ over $R[q,q^{-1}]]$ but now the annihilators are all units, so that $\langle \L \rangle_\ell = \langle \L \rangle_{free}$.
\end{proof}
The computations in Sections \ref{sec:two unlinked longitudes} and \ref{sec:Whitehead} provide examples of this behavior (note that in Section \ref{sec:two unlinked longitudes}, we would get the same answer if the strands were oppositely oriented so that $\L$ were null-homologous).

\begin{remark}
The computation in Section \ref{sec:crossingless} provides a link $\L\subset M^2$ with $\chi_q(Kh^*(\L)) \notin R[q,q^{-1}]$, showing that Corollary \ref{cor:polys in M1} does not hold there.  This indicates that some aspect of Hoste's and Przytycki's argument must be altered in $\sk(M^r)$ over $R[q,q^{-1}]$: either there is more than one free summand (corresponding to some non-empty, null-homologous links, perhaps similar to the one in Figure \ref{fig:Lcrossingless}), or the torsion elements are annihilated by quantities that do not become units in $R[q,q^{-1}]$ (most likely corresponding to performing combinations of surgery wrap moves on differing surgery lines).
\end{remark}

For links $\L\subset M^r$ with non-zero algebraic intersection numbers $\eta_i$, our Khovanov homology groups are only well-defined up to grading shifts depending on the $\eta_i$.  From one point of view this would make the corresponding version of Theorem \ref{thm:general skein module cat} unsatisfactory, since $\chi_q(Kh^*(\L))$ would only recover $\langle \L \rangle_\ell$ up to multiples of $q^{3g}$, where $g=\gcd(\eta_1,\dots,\eta_r)$.  On the other hand, one could speculate that this ambiguity represents a naive categorical analogue of some aspect of the torsion in $\sk(M^r)$, whereby we lose a notion of absolute gradings on our homology groups, but do retain some notion of relative gradings which are wholly invisible on the decategorified level.

\begin{remark}\label{rmk:skein module odd int}
In \cite{S1S2skein}, Hoste and Przytycki provide computations in $\sk(S^2\times S^1)$ that show that any $\L$ having odd intersection number with the belt sphere represents a trivial element in the skein module.  Their methods generalize immediately to $M^r$ for any $r$, so we are tempted to define
\[Kh^*(\L):= 0\]
for all $\L$ having belt sphere intersection number $n_i$ odd for some $i$.  
We will see another justification for this extension in what follows.
\end{remark}

\subsection{The WRT invariant}\label{sec:WRT}
In this section we describe how, for a link $\L\subset M^r$, the graded Euler characteristic $\chi_q(Kh^*(\L))$ (evaluated at certain roots of unity) recovers the corresponding WRT invariant of $\L$ in $M^r$.  All of this material is a straightforward generalization of the arguments in the appendix of \cite{Roz}.

We begin by recalling some of the basic definitions (see \cite{KL} for a detailed account of these topics).  Fixing an integer $n$, there exist in the Temperley-Lieb algebra $TL_n$ certain idempotent elements $P_{n,d}$ for each $d$ such that $n-d\in2\Z_{\geq 0}$.  These are the Jones-Wenzl projectors \cite{Wenzl}, and they involve \emph{rational} functions of $q$ as their coefficients in $TL_n$.  These projectors satisfy the following properties:
\begin{itemize}
\item $P_n:=P_{n,n}$ is the usual Jones-Wenzl projector satisfying, for any Temperley-Lieb diagram $\delta$,
\[P_n\cdot \langle \delta \rangle = \langle \delta \rangle \cdot P_n = \begin{cases}
P_n & \text{if $\delta=\iota_n$, the identity diagram}\\
0 & \text{otherwise}
\end{cases};\]

\item $P_{n,d}\cdot \langle\delta\rangle = 0 $ if $\th(\delta) < d$;
\item $P_{n,d}$ is a linear combination of terms of the form $\langle \topp{\delta}\rangle \cdot P_d \cdot \langle \bott{\delta} \rangle$ for $\th(\delta)=d$ (see Definition \ref{def:delta top bot});
\item $\sum_d P_{n,d} = \langle \iota_d \rangle$, the identity element in $TL_n$.
\end{itemize}
Here we are using the notation $\langle \delta \rangle$ to indicate the element in $TL_n$ corresponding to the isotopy class of $\delta$.

With the help of the various $P_{n,d}$ we can define the WRT invariants \cite{RT,Wit} of a link $\L$ in a 3-manifold $M$.  Suppose $M$ can be obtained by performing surgery on a fixed (framed) link $\mathcal{L}_M\subset S^3$ (with diagram $L_M$).  Let $\L'$ be the link in $S^3$ (with diagram $L'$), disjoint from $\mathcal{L}_M$, that gives rise to $\L\subset M$ after the surgery is performed.  Then the $\rho^{th}$ WRT invariant of $\L$ in $M$, denoted by $Z_\rho(\L,M)$, can be computed as a sum of terms involving the Kauffman brackets of diagrams formed from $L'\cup L_M\subset S^3$ by cabling the surgery link $L_M$ and inserting Jones-Wenzl projectors into each component.  The sum is taken over various cablings from a 0-cabling (deleting the link $L_M$) through to a $(\rho-2)$-cabling.  After summing together these various Kauffman brackets, the expression is evaluated at $q=e^{i\pi/\rho}$ to arrive at the complex number $Z_\rho(\L,M)$.

For the case of $\L\subset M=M^r$ as illustrated in the previous sections, the links $\L'$ and $\mathcal{L}_M$ are clear.  Given $\L\subset M^r$, the preimage link diagram $L'$ is precisely what we have called $\Lzero$ in Section \ref{sec:def KC(L)}, and the surgery link diagram $L_{M^r}$ is just the unlink formed by placing 0-framed meridians encircling each surgery line.  In the case of such meridian surgeries, it can be shown (see the appendix in \cite{Roz}) that, so long as each intersection number $n_i$ of $\L$ with each belt sphere is even and satisfies $n_i\leq \rho-2$, the majority of the terms in the sum of Kauffman brackets for $Z_\rho(\L,M^r)$ disappear and the result is equivalent to \emph{deleting} the meridians and instead placing the projector $P_{n_i,0}$ along the cabling of what we have called the surgery lines $\sl_i$.  See Figure \ref{fig:computing WRT}.  Furthermore, if any of the intersection numbers $n_i$ is odd, the resulting WRT invariant is zero.

This gives a further justification for simply defining $Kh^*(\L):=0$ for links having odd intersection number with any of the belt spheres.  Then in order to confirm that our homology $Kh^*(\L)$ `categorifies' the WRT invariants (in the sense that evaluating the graded Euler characteristic of $Kh^*(\L)$ at $e^{i\pi/\rho}$ recovers $Z_\rho$), it is enough to show that the stable limiting complex associated to the infinite full twist as in Section \ref{sec:Infinite twist} categorifies $P_{n_i,0}$ in this same sense.  This is shown in the appendix to \cite{Roz}, but we repeat the argument here in our normalization.

\begin{figure}
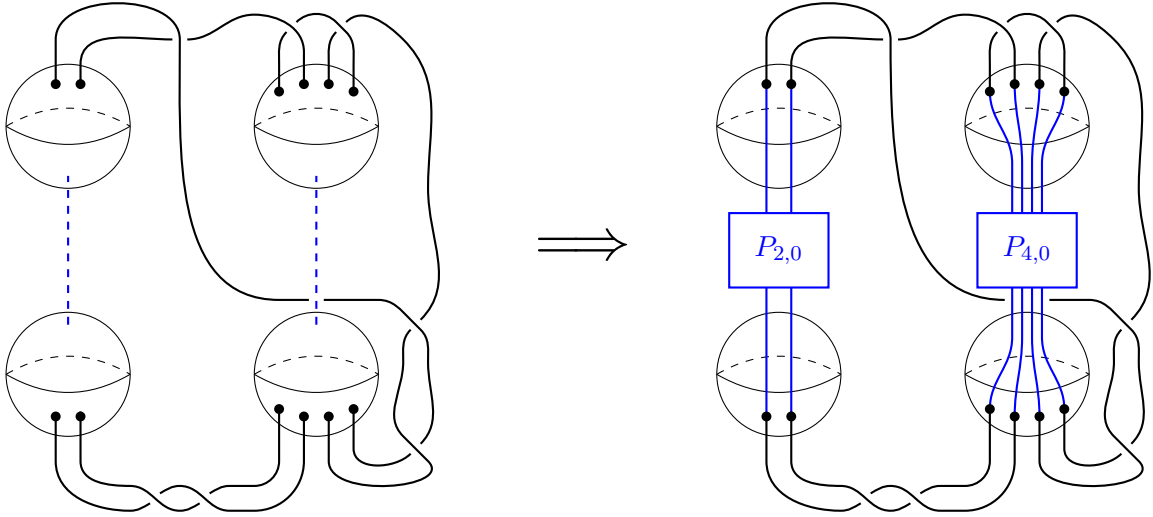

\centering
\LinMr
\hspace{.3in}
\resizebox{.5in}{!}{$\Longrightarrow$}
\hspace{.3in}
\LkinMr{P_{2,0}}{P_{4,0}}
\caption{To compute the WRT invariant $Z_\rho(\L,M^r)$, it is enough to use the diagram $\Lzero$ and insert $P_{n_i,0}$ instead of full twists into the cabled surgery lines, before taking the Kauffman bracket and evaluating at $q=e^{i\pi/\rho}$ (so long as $n_i\leq \rho-2$ for each $i$).}
\label{fig:computing WRT}
\end{figure}

We continue to use the notation $\langle \cdot \rangle$ to denote the unnormalized Kauffman bracket, which in $S^3$ is the same as the graded Euler characteristic of $\KC^*(\cdot)$ and satisfies equivalent formulas for shifts due to Reidemeister moves.  In particular, for any term $\langle \topp{\delta} \cdot P_d \cdot \bott{\delta} \rangle$ (allowing a slight abuse in the use of the notation $\langle\cdot\rangle$) in the linear combination for $P_{n,d}$, we must have $\topp{\delta}$ having $\frac{n-d}{2}$ matchings.  If we concatenate with a full twist $\F_n$, these matchings can be pulled through the twist giving (see Lemma \ref{lem:ft Reid shifts}; all homological shifts are even, and so have no effect on the Euler characteristic):
\begin{equation}\label{eq:Kauffman FT shift}
\langle \F_n \cdot \topp{\delta} \cdot P_d \cdot \bott{\delta} \rangle = q^{(n-d)\left(n+1-\frac{n-d}{2}\right)} \langle \topp{\delta} \cdot \F_d \cdot P_d \cdot \bott{\delta} \rangle.
\end{equation}
Now we consider the graded Euler characteristic of our limiting complex of Theorem \ref{thm:inf twist complex} for the infinite twist on $n=2p$ strands.
\begin{align*}
\lim_{k\rightarrow\infty} \chi_q(C^*(k)) &= \lim_{k\rightarrow\infty} \chi_q \left( \h^{-2kp^2}\q^{-2kp(p+1)} \KC_{\geq -(2k-1)}^*(\F_n^k) \right)\\
&= \lim_{k\rightarrow\infty} q^{-2kp(p+1)} \langle \F_n^k \rangle
\end{align*}
where we have discarded the truncation $(*\geq -(2k-1))$ since we are taking the stable limit of each coefficient in these (growing) polynomials, and any error terms get eliminated in this process.  We now multiply by the identity written as a sum of projectors, and then write each projector as a linear combination $P_{2p,2m}=\sum_{\delta} a_\delta \langle \topp{\delta} \cdot P_{2m} \cdot \bott{\delta} \rangle$ before using Equation \ref{eq:Kauffman FT shift}:
\begin{align*}
\lim_{k\rightarrow\infty} q^{-2kp(p+1)}\langle \F_n^k \rangle &= \lim_{k\rightarrow\infty} q^{-2kp(p+1)}\langle \F_n^k \rangle \cdot \left( \sum_{m=0}^p P_{2p,2m}\right)\\
&= \lim_{k\rightarrow\infty} q^{-2kp(p+1)} \sum_{m=0}^p \langle \F_n^k \rangle \cdot P_{2p,2m}\\
&= \lim_{k\rightarrow\infty} q^{-2kp(p+1)} \sum_m \sum_{\delta} a_\delta \langle \F_n^k \cdot \topp{\delta} \cdot P_{2m} \cdot \bott{\delta} \rangle\\
&= \lim_{k\rightarrow\infty} q^{-2kp(p+1)} \sum_m \sum_{\delta} a_\delta q^{2k(p-m)(p+m+1)}\langle \topp{\delta} \cdot \F_{2m}^k \cdot P_{2m} \cdot \bott{\delta} \rangle\\
&= \lim_{k\rightarrow\infty} \sum_m \sum_{\delta} a_\delta q^{-2k(m^2+m)}\langle \topp{\delta} \cdot \F_{2m}^k \cdot P_{2m} \cdot \bott{\delta} \rangle\\
&=\lim_{k\rightarrow\infty} \sum_m q^{-2k(m^2+m)} P_{2p,2m}
\end{align*}
where in the last line we have used the fact that $\langle \F_{2m} \rangle \cdot P_{2m} = P_{2m}$, which can be derived from the condition that $P_{2m}$ kills all turnbacks.  It is clear that, as $k\rightarrow\infty$, all terms for which $m\neq 0$ cannot contribute to a stable limit for any coefficient of the series, and so only the term $P_{2p,0}$ remains.

\section{Khovanov homology for the knotification of a link in $S^3$}\label{sec:knotification Kh}
We now change our viewpoint slightly.  Given a link $\L\subset S^3$ with $r+1$ components, there is a well-defined procedure to construct a \emph{knot} $\K_\L\subset M^r$ which will be called the \emph{knotification} of $\L$.  This can be used, for instance, to define knot Floer invariants for $\L$ \cite{OS}.  Although the construction of Section \ref{sec:Defining general Kh} provides a corresponding Khovanov homology theory for $\K_\L$, we wish to augment the construction to allow for direct computation from the diagram $L$ for $\L\subset S^3$, rather than first rearranging $M^r$ (and $\K_\L$ within it) into the standard position of Section \ref{sec:Defining general Kh}.

\subsection{The knotification of a link in $S^3$}\label{sec:knotification}
Given an oriented link diagram $L$ for a link $\L\subset S^3$ with $r+1$ components, we describe the process that produces a diagram $K_L$ for the knotification $\K_\L\subset M^r$ (see Section 2.1 in \cite{OS} for further discussion).

First we focus on constructing $\K_\L$.  We choose $r$ pairs of points $(P_1,Q_1),(P_2,Q_2),\dots,(P_r,Q_r)$ along $\L$ in such a way that the graph formed by identifying the points $P_i \cong Q_i$ on $\L$ is connected.  We view any pair $(P_i,Q_i)$ as an embedding of $\SSi$ into $S^3$ at the points $P_i$ and $Q_i$, on which we perform $S^0$ surgery.  The two strands of $\L$ that intersect these spheres are connected in an orientation preserving manner by a band running along the handle $(D^1\times S^2)_i$ (with an arbitrary amount of twisting), so that $\L$ has geometric intersection number $n_i=2$ with the belt sphere of the handle.  This gives us the knot $\K_\L$.

To construct the planar diagram, we begin by drawing straight surgery lines in $S^3$ connecting each $P_i$ to $Q_i$.  After a small isotopy as necessary, we can assume our link, our choices of points, and our surgery lines are in general position so that the surgery lines did not intersect, and after projecting down to the plane, we have no points of triple intersection (this precludes the placement of a point $P_i/Q_i$ at a crossing), no tangencies, and no cusps.  As usual, we keep track of over and under crossing data in this process, and we draw our surgery lines $\sl_i$ as dashed lines in blue.  We leave the $\sl_i$ unoriented, since $L$ intersects each attaching sphere in two oppositely oriented points, ensuring that $\eta_i=0$ for all $i$ (see Definition \ref{def:notations for surgery strands}, which can easily be adapted to our new situation).  The resulting diagram is denoted $K_L$, representing the knot $\K_\L\subset M^r$.  As before, we do not attempt to draw the handle.  See Figure \ref{fig:knotification ex} for an example.

It is proven in \cite{OS} that this process gives a well-defined function taking a link $\L\subset S^3$ to a knot $\K_\L\subset M^r$.  In particular, isotopies of $\L$ and alternative choices of the data result in isotopic knots in $M^r$.  Of particular interest here is the possibility of a handle-slide type of move which takes a chosen attaching point on one component of $\L$ and passes it through the handle attached to some other point on that same component of $\L$, as illustrated for a diagram $K_L$ in Figure \ref{fig:handleslide}.

\begin{figure}
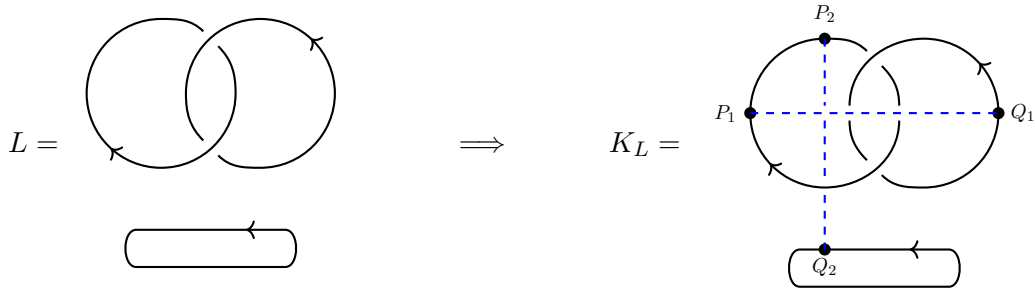

\[L=\knotificationExL \hspace{.5in} \Longrightarrow \hspace{.5in} K_L=\knotificationExKL\]
\caption{The oriented link diagram $L$ for $\L\subset S^3$ has three components.  We choose two pairs of points $(P_i,Q_i)$ to act as attaching spheres and draw surgery lines between each pair.  The resulting diagram $K_L$ represents a well-defined knot $\K_\L\subset M^2$.}
\label{fig:knotification ex}
\end{figure}

\begin{figure}
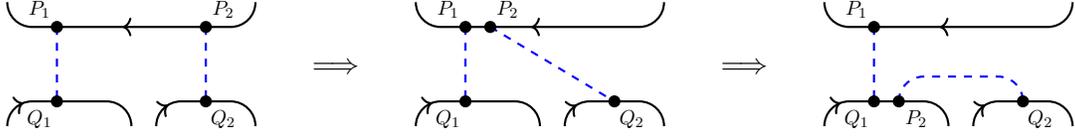

\[\hslideexA \hspace{.15in} \Longrightarrow \hspace{.15in} \hslideexB \hspace{.15in} \Longrightarrow \hspace{.15in} \hslideexC\]
\caption{A handle-slide move that lets one point, representing an attaching sphere, pass through the surgered handle corresponding to another point.  Note the orientations - with the ability to perform Reidemeister 1 moves on $L$ before a handle-slide, all such moves are diagrammatically equivalent to the one illustrated here.}
\label{fig:handleslide}
\end{figure}

\subsection{The Khovanov homology of the knotification}
We construct a Khovanov complex for our diagram $K_L$ in a manner analogous to the construction of Section \ref{sec:Defining general Kh}.  First, we delete $P_i$ and $Q_i$ from our link diagram $L$.  Then we replace each surgery line $\sl_i$ in our planar diagram $K_L$ with two parallel copies of the line (maintaining crossing data with other strands), and then orient the strands to match with the strands of $L$ at $P_i$.  If necessary, we add a half twist (ie a positive crossing) near $Q_i$ to allow these orientated strands to attach to $L$ at $Q_i$.  The resulting diagram could be denoted $\KLzero$.  We then choose a point on each $\sl_i$ (that is not a crossing point) at which to insert the complex $C^*(\F_2^\infty)$, and tensor this in the planar algebraic sense with the Khovanov complex for the rest of the diagram just as in Section \ref{sec:Defining general Kh}, and define the result as the \emph{Khovanov complex} of the diagram $K_L$.  Again we allow ourselves the abuse of notation to think of this as defining
\begin{equation}\label{eq:KC(KL) def}
KC^*(K_L):=KC^*(\KLinf).
\end{equation}
The \emph{Khovanov homology} $Kh^*(K_L)$ of the diagram $K_L$ is then the homology of this complex $KC^*(K_L)$.

Meanwhile for any vector $\vec{k}=(k_1,\dots,k_r)$ we use the notation $\KLk$ to denote the genuine knot diagram formed by inserting $\F_2^{k_i}$ at each insertion point along $\sl_i$.  See Figure \ref{fig:KLk example} for clarification.  We also use the notation $\KCsimp(\KLk)$ to denote the simplification of $\KC^*(\KLk)$ where the complexes $\KC^*(\F_2^{k_i})$ have been simplified to approximate Equation \ref{eq:two strand cx}.  This allows us to state the following versions of Proposition \ref{prop:fin approx} and Corollary \ref{cor:KC(Linf)=KC(Lk)} whose proofs are identical to the earlier versions and are omitted here.

\begin{figure}
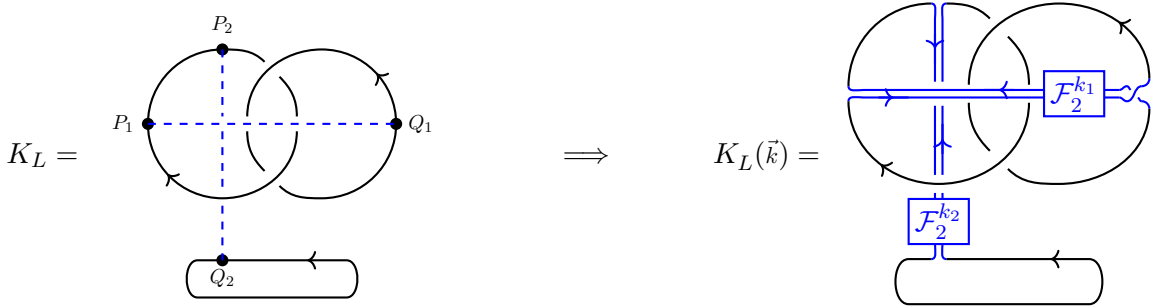

\[K_L=\knotificationExKL \hspace{.5in} \Longrightarrow \hspace{.5in} \KLk=\knotificationExKLk{$\F_2^{k_1}$}{$\F_2^{k_2}$}\]
\caption{Continuing the example of Figure \ref{fig:knotification ex}, we build a genuine knot diagram $\KLk$ from the oriented diagram $K_L$ (which itself came from the diagram $L$ for $\L\subset S^3$).  $\KLinf$ can be pictured similarly, with the complexes $C^*(\F_2^\infty)$ in place of the twists $\F_2^{k_i}$.}
\label{fig:KLk example}
\end{figure}

\begin{proposition}\label{prop:KL fin approx}
Given an $r+1$ component link diagram $L$ for $\L\subset S^3$ and an arbitrary homological lower bound $a$, there exists some finite $\vec{k}$ such that the truncated Khovanov complex $KC^*_{\geq a}(K_L)=\KCsimp_{\geq a}(\KLinf)$ is equal to the truncation of the finite approximation complex $\KCsimp_{\geq a}(\KLk)$.  In other words, $KC^*(K_L)$ can be approximated by a finite complex in any given homological range.
\end{proposition}

\begin{corollary}\label{cor:KC(KLinf)=KC(KLk)}
Given an $r+1$ component link diagram $L$ for $\L\subset S^3$, the Khovanov homology $Kh^*(K_L)$ can be approximated in any finite range of homological degrees by the Khovanov homology of a genuine link diagram $\KLk$ for finite $\vec{k}$.  More precisely, given any homological bound $a$, there exists some $\vec{k}$ that depends on $a$ such that
\begin{equation}\label{eq:KC(KLinf)=KC(KLk)}
Kh^*(K_L) \cong Kh^*(\KLk)
\end{equation}
in all homological gradings $*\geq a$.
\end{corollary}
\begin{proof}
Notice that in the case where all of the $\eta_i=0$ (so all of the crossings in $\F_2$ are negative), all of the grading shifts cancel.
\end{proof}

\subsection{Invariance of $Kh^*(K_L)$}
The goal of this section is two-fold.  On the one hand, we wish to show that our definition for $Kh^*(K_L)$ is indeed an invariant of $\K_L\subset M^r$.  Secondly, we expect our new construction to give the same homology as we would get by applying the results of Section \ref{sec:Defining general Kh} to our knot $\K_L$.  Fortunately, these two goals are related since our new construction can be related to our old one by planar isotopies corresponding to changing the projection $M^r\rightarrow \R^2$ (see Remark \ref{rmk:fixed projection}).

\begin{theorem}\label{thm:old style isotopies}
$Kh^*(K_L)$ is invariant under any of the relevant moves from Proposition \ref{prop:isotopic links} and Theorem \ref{thm:Kh invariance}.  Specifically, any planar isotopies, Reidemeister moves, point-pass moves (better thought of now as `$P_i$-pass' or `$Q_i$-pass' moves), and surgery-wrap moves which fix the surgery data $P_i,Q_i,\sl_i$ induce an isomorphism on $Kh^*(K_L)$ with no grading shifts.  $Kh^*(K_L)$ is also invariant under choice of insertion points for the $C^*(\F_2^\infty)$ along each $\sl_i$.
\end{theorem}
\begin{proof}
The proof works exactly the same way as in Theorem \ref{thm:wrapping move shift} (utilizing the through-degree zero property of the terms in $C^*(\F_2^\infty)$) and Theorem \ref{thm:Kh invariance} (utilizing isotopies of genuine knots in $S^3$ now with the help of Corollary \ref{cor:KC(KLinf)=KC(KLk)}).  The grading shifts all depended on $\eta_i$, which are all zero in the case of $K_L$.  If we truly blow up the points $P_i,Q_i$ into genuine spheres and arrange the intersections of them with the link properly, we can also arrange for mirror moves and finger moves, but these are unnecessary for our arguments in this section.
\end{proof}

\begin{theorem}\label{thm:isotope surgery data}
$Kh^*(K_L)$ is invariant under any planar isotopy, Reidemeister move, $P_i/Q_i$-pass move, and surgery-wrap move that affects the surgery data $P_i,Q_i,\sl_i$.  See the various figures in the proof below for clarification.
\end{theorem}
\begin{proof}
The proofs for planar isotopies and Reidemeister 2 and 3 moves work precisely the same as before.  We use Corollary \ref{cor:KC(KLinf)=KC(KLk)} and seek simple isotopies of genuine knots.  A version of Reidemeister 2 is illustrated below:
\[\ReidExaOLD \cong \ReidExbOLD \hspace{.5in} \Longrightarrow \hspace{.5in} \ReidExcOLD \cong \ReidExdOLD .\]
We can use the same approach for $P_i/Q_i$-pass moves and any isotopies involving them, as illustrated below:
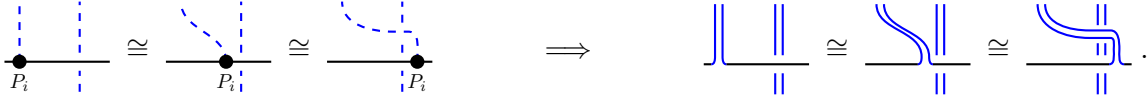
\[\PiPassExa \cong \PiPassExb \cong \PiPassExc \hspace{.5in} \Longrightarrow \hspace{.5in} \PiPassExd \cong \PiPassExe \cong \PiPassExf .\]
Note that we can always move insertion points so that the twists $\F_2^{k_i}$ do not need to be included in these local pictures (although clearly they would not cause problems if they were present).

Similarly, surgery-wrap moves involving a second surgery line (which in turn allow surgery lines to pass through each other or through themselves) are handled in the exact same way as they were in Theorem \ref{thm:wrapping move shift}, where here it is important that we begin by moving the insertion points for the complexes $C^*(\F_2^\infty)$ so that only one is visible in the diagram.  Rather than going through the details again, we simply show the diagrammatic version and refer the reader back to the proof of Theorem \ref{thm:wrapping move shift}.
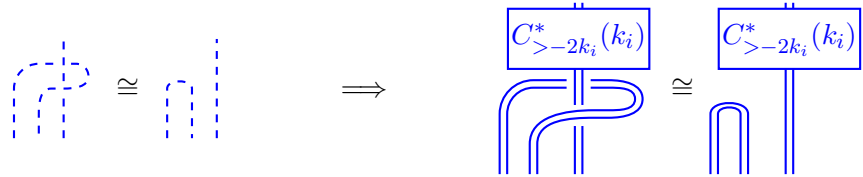
\[
\dashdashwrap \cong \dashdisjointdash \hspace{.5in} \Longrightarrow \hspace{.5in} \dashdashwrapB{$\KCsimp_{>-2k_i}(k_i)$} \cong \dashdisjointdashB{$\KCsimp_{>-2k_i}(k_i)$}
\]
The essential idea is that, since $\KCsimp_{>-2k_i}(k_i)$ is a complex made up of nothing but through-degree zero diagrams $\delta$, we have cup-sliding isotopies giving very strong deformation retracts throughout the multicone expansion for $\KCsimp_{>-2k_i}(k_i)$ that result in a multicone of single-term complexes.  An additional argument is needed to ensure that we can give consistent signs to the resulting maps in order to keep our multicone truly unaltered - this is handled by choosing a consistent bottom closure $\gamma$ for the diagram at hand, and noting that the isotopies could involve sliding the matchings of $\gamma$ instead.  These two choices are isotopic as tangle cobordisms up to a sequence of cobordisms of the form $\rho_\circ$ which induce identity maps on the complexes (see Lemma \ref{lem:circle wrap is identity}).  The projective functoriality of Bar-Natan's constructions then tell us that the two choices for each isotopy induce homotopic maps up to a sign, producing our set of consistent signs for our maps.  The grading shifts work exactly the same as before.

Finally, we note that a Reidemeister 1 move on a surgery line creates a framing twist, which is equivalent to adding a full twist along blue strands in $K_L$.  However, in the case of having only two parallel strands, a full twist is equivalent to a Jucys-Murphy element which can be added by a surgery wrap with one of the incoming strands as in Figure \ref{fig:other wrapping moves}, and we have already checked general surgery wrapping moves.  So we are done.
\end{proof}

\begin{corollary}\label{cor:Khs are the same}
Given the knotification $\K_L\subset M^r$ of a link $\L\subset S^3$, the homology groups $Kh^*(K_L)$ constructed in this section are isomorphic to the homology groups $Kh^*(\K_\L)$ constructed in Section \ref{sec:Defining general Kh}.
\end{corollary}
\begin{proof}
With the ability to isotope the link and the surgery data $P_i,Q_i,\sl_i$, it is not hard to see that any diagram $K_L$ can be isotoped into a diagram which matches the format of the diagrams of Section \ref{sec:Defining general Kh}, at which point the constructions are precisely the same.  See Figure \ref{fig:Khs are the same} for an illustration using our previous example diagram.
\end{proof}

\begin{figure}
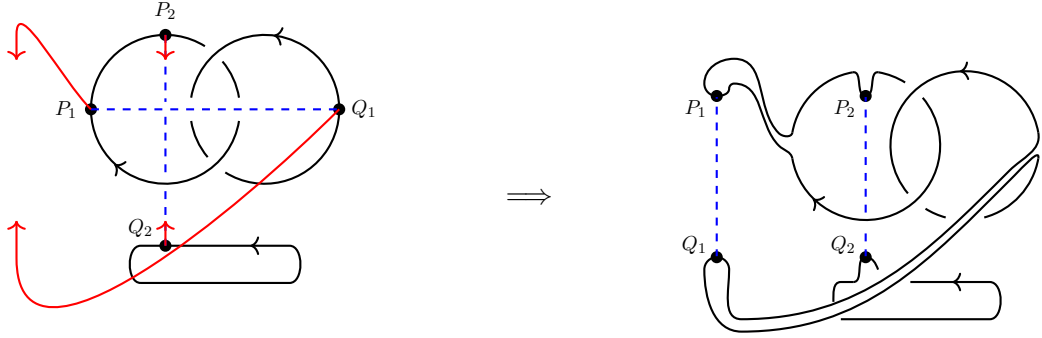

\[
\KLtoStandardA \hspRarrow{.5in} \KLtoStandardB
\]
\caption{The isotopy moving $K_L$ from Figure \ref{fig:knotification ex} into the standard position of Section \ref{sec:Defining general Kh}, from which point the two constructions are equivalent (after replacing the points $P_i,Q_i$ with small attaching spheres).  Recall that surgery lines can pass through any strands of the link and each other, so there is no trouble in isotoping them into the desired position.}
\label{fig:Khs are the same}
\end{figure}

Finally, we turn to the question of invariance under the genuinely new diagrammatic move available in this section, the handle-slide.  This move requires an extension of Lemma \ref{lem:Moving Dots} from Section \ref{sec:Moving Dots} for moving dots past crossings in a diagram.  That lemma focused on changing a single dotted cobordism near a crossing.  The following lemma concerns the case where we have many copies of the same dotted cobordisms, and we would like to change them all.

\begin{lemma}\label{lem:Moving Dots in Cone}
Consider a multicone of the form
\[\C^*:= \mCone{\mathcal{Z}_x}{\mathcal{Z}_y}{\C^*} {\mathcal{Z}_x} {\phi_{x,y}} {\mathcal{Z}_y}\]
where we have:
\begin{itemize}
\item All of the diagrams are the same tangle: $\mathcal{Z}_x=\mathcal{Z}_y=\mathcal{Z}$ for all $x$ and $y$.
\item There is a single crossing in $\mathcal{Z}$ such that, for all $x,y$, $\phi_{x,y}=\crossingdotTL + (-1)^{\h_{\C^*}(\mathcal{Z}_x)} \tilde{\phi}$ for some fixed $\tilde{\phi}$ that is a sum of dotted cobordisms that are identity near the crossing.  In other words, every map $\phi_{x,y}$ includes the same dotted identity cobordism near a crossing as one of its summands, while the other summands alternate as we traverse the multicone.
\end{itemize}
Then we have
\[\C^*\simeq \mCone{\mathcal{Z}_x}{\mathcal{Z}_y}{\C^*} {\mathcal{Z}_x} {(\phi_{x,y}')} {\mathcal{Z}_y}\]
where
\[\phi_{x,y}' = -\crossingdotBR + (-1)^{\h_{\C^*}(\mathcal{Z}_x)} \tilde{\phi}.\]
In other words, we are free to apply Lemma \ref{lem:Moving Dots} to the entire multicone at once.
\end{lemma}
Of course, there is a similar statement for passing a dot along an over-crossing throughout a multicone as well.
\begin{proof}
The simplistic form of the diagrams and maps involved means that the homotopy of Lemma \ref{lem:Moving Dots} can be repeated throughout the mapping cone, and the resulting diagram will commute as necessary giving a chain homotopy equivalence between the two multicones.  See Figure \ref{fig:Moving Dots in Cone}.  Note that the alternating signs of the `extra' summands $\tilde{\phi}$ allow the diagonal homotopies to commute with the other maps as needed (ie there are no higher homotopies required).  The details are left to the reader.

\MovingDotsConeCDprep
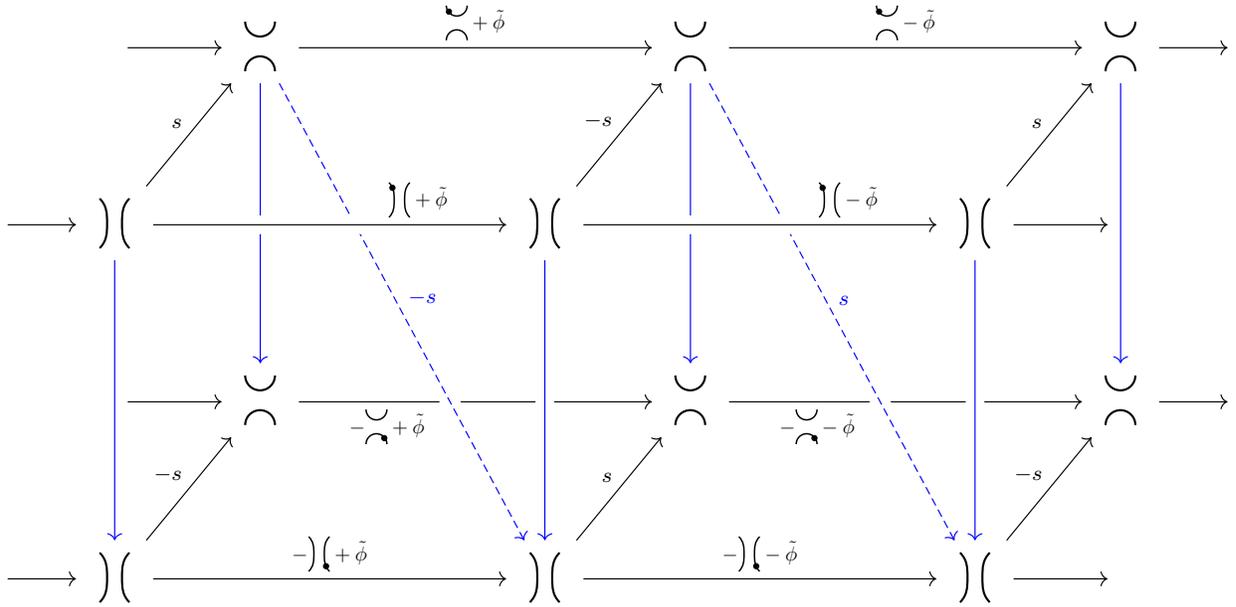
\begin{figure}
\[
\begin{tikzcd}
& \arrow[r] & \usebox{\boxB} \arrow[rrrr, "{ \usebox{\boxBD} }"] \arrow[dddd, blue]  & & & & \usebox{\boxB} \arrow[rrrr, "{ \usebox{\boxBDa} }"] \arrow[dddd, blue] & & & & \usebox{\boxB} \arrow[dddd, blue] \arrow[r] & {}\\
\\
\arrow[r] & \usebox{\boxA} \arrow[ruu, "s"] \arrow[dddd, blue] & & & & \usebox{\boxA} \arrow[ruu,  "-s"] & & & & \usebox{\boxA} \arrow[ruu,  "s"] \arrow[r] & {} & \\
\\
& \arrow[r] & \usebox{\boxB} \arrow[rrrr, "{ \usebox{\boxBDb} }", swap, near start]  & & & & \usebox{\boxB} \arrow[rrrr, "{ \usebox{\boxBDc} }", swap, near start] & & & & \usebox{\boxB} \arrow[r] & {}\\
\\
\arrow[r] & \usebox{\boxA} \arrow[rrrr, "{ \usebox{\boxACb} }"] \arrow[ruu, "-s"]  & & & & \usebox{\boxA} \arrow[rrrr, "{ \usebox{\boxACc} }"] \arrow[ruu, "s"] & & & & \usebox{\boxA} \arrow[ruu, "-s"] \arrow[r] & {} & 
\arrow[from=1-3, to=7-6, blue, dashed, crossing over, "-s"]
\arrow[from=1-7, to=7-10, blue, dashed, crossing over, "s"]
\arrow[from=3-6, to=7-6, blue, crossing over]
\arrow[from=3-10, to=7-10, blue, crossing over]
\arrow[from=3-2, to=3-6, crossing over, "{ \usebox{\boxAC} }", near end]
\arrow[from=3-6, to=3-10, crossing over, "{ \usebox{\boxACa} }", near end]
\end{tikzcd}
\]
\caption{The diagram illustrating Lemma \ref{lem:Moving Dots in Cone}.  The top faces, which we imagine continuing on in either direction (but not indefinitely), represent the multicone expansion of $\C^*$, while the bottom faces are the same but replacing $\phi$ with $\phi'$ by changing the sign of the specified dotted identity cobordism.  The blue maps combine to provide the chain homotopy equivalence between these two multicones.  The unmarked maps are identity maps, while the $s$ stands for the obvious saddle cobordism.}
\label{fig:Moving Dots in Cone}
\end{figure}

\end{proof}

\begin{theorem}\label{thm:Kh preserved under handle slides}
Given a link diagram $L$ for $\L\subset S^3$, let $K_L^1$ and $K_L^2$ be two knotification diagrams coming from $L$ such that one is obtained from the other by a handle-slide move as in Figure \ref{fig:handleslide}.  Then we have an isomorphism
\[Kh^*(K_L^1) \cong Kh^*(K_L^2).\]
\end{theorem}
\begin{proof}
Call the two surgery lines involved $\sl_1$ (running between $P_1$ and $Q_1$) and $\sl_2$ (running between $P_2$ and $Q_2$), where $P_2$ is the point that is slid through the handle of $\sl_1$ from $P_1$ to $Q_1$ (this is the case illustrated in Figure \ref{fig:handleslide}). Thus $\sl_1$ is unchanged by the handle-slide but $\sl_2$ is changed.

Having the ability to perform isotopies with the surgery lines, we begin by isotoping the local picture for $K_{L_1}$ so that $P_2$ is very close to $P_1$ (this alone gives the second diagram in Figure \ref{fig:handleslide}).  Next we isotope $\sl_2$ further so that, starting from $P_2$, it runs parallel along $\sl_1$ until it reaches a point near $Q_1$ before turning away and going towards $Q_2$.  The ability to have $\sl_2$ pass through any other strands of $L$ and/or other surgery lines ensures that we can accomplish this in such a way that the over/under crossing data of $\sl_2$ matches that of $\sl_1$ while running alongside it.  Call this new diagram $K_L^{1.5}$ (See Figure \ref{fig:KL1.5} for clarification).  Clearly we have $Kh^*(K_L^1)\cong Kh^*(K_L^{1.5})$.  We wish to show that $Kh^*(K_L^{1.5})\cong Kh^*(K_L^2)$.

\begin{figure}
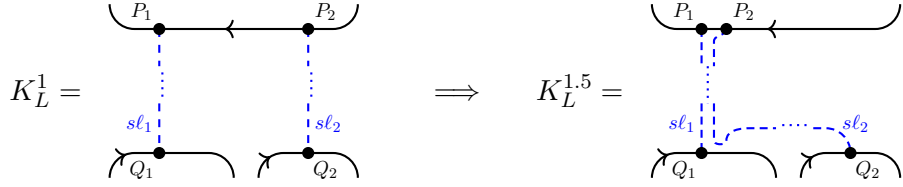

\[K_L^1 = \hslideEqvA \hspace{.25in} \Longrightarrow \hspace{.25in} K_L^{1.5}= \hslideEqvB \]
\caption{The local diagram for $K_L^1$ can be isotoped to the new local diagram $K_L^{1.5}$ where $\sl_2$ runs parallel along $\sl_1$ before reaching $Q_2$.  The dotted lines indicate places where there could be more crossings with other strands and surgery lines, but this crossing data is the same for $\sl_1$ and $\sl_2$ while they are running parallel.}
\label{fig:KL1.5}
\end{figure}

In order to do this, we choose insertion points near the point $Q_1$ and insert a genuine twist $\F_2^{k_1}$ along $\sl_1$, but the truncated complex $\KCsimp_{>a}(k_2)$ along $\sl_2$ (here we set $a:=-2k_2$ to avoid clutter in our diagrams).  For all surgery lines away from this local picture, we choose arbitrary insertion points and insert $\F_2^{k_i}$ as usual.  Allowing the usual abuse of notation, we will consider this as a diagram denoted $\KLk[1.5]$.  See Figure \ref{fig:KL1.5k}.

\begin{figure}
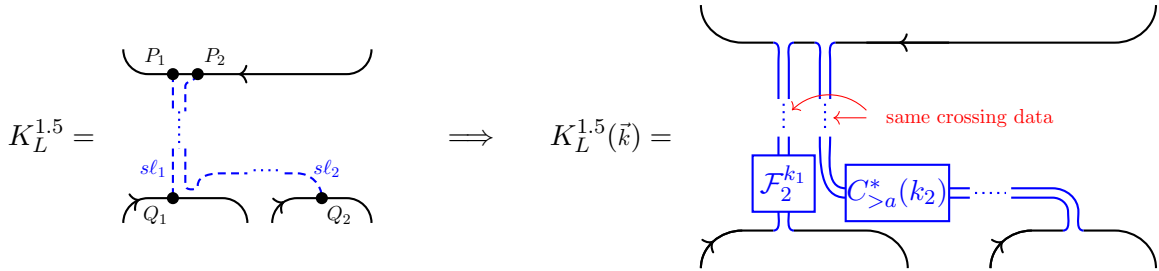

\[K_L^{1.5}= \hslideEqvB \hspace{.25in} \Longrightarrow \hspace{.25in} \KLk[1.5] = \hslideEqvC{$\F_2^{k_1}$}{$\KCsimp_{>a}(k_2)$} \]
\caption{The diagram $K_L^{1.5}$ leads to the diagram $\KLk[1.5]$ built by choosing insertion points near the point $Q_1$.  Again, the dotted lines indicate places where there could be more crossings with other strands and surgery lines, but this crossing data is the same for the strands coming from $\sl_1$ and $\sl_2$ while they are running parallel.}
\label{fig:KL1.5k}
\end{figure}

Now we use the complex for $\KCsimp_{>a}(k_2)$ described in Section \ref{sec:Examples}, notated in such a way as to fit into the diagram $\KLk[1.5]$ in Figure \ref{fig:KL1.5k}:
\begin{equation}\label{eq:two strand trunc cx}
\twostrandA{$\KCsimp_{>a}(k_2)$} \simeq
\left(
\twostrandcups \xrightarrow{\twostranddotL + \twostranddotR} \twostrandcups
\xrightarrow{\twostranddotL - \twostranddotR} \twostrandcups
\rightarrow \dots \rightarrow \twostrandcups
\right).
\end{equation}
In Equation \ref{eq:two strand trunc cx}, the pair of maps $(\twostranddotL+\twostranddotR)$ and $(\twostranddotL-\twostranddotR)$ repeat in sequence for a total of $k_2$ such pairs.  The $q$-degree shifts in this complex are omitted for simplicity.

If we imagine inserting Equation \ref{eq:two strand trunc cx} into the diagram for $\KLk[1.5]$ in Figure \ref{fig:KL1.5k}, we can see a complex of dotted cobordisms which satisfies the conditions of Lemma \ref{lem:Moving Dots in Cone} where we choose to move the dot only on the `left-dotted cobordisms' $\twostranddotL$, viewing the `right-dotted cobordisms' as the extra maps $\tilde{\phi}$ in that lemma's notation (note that these maps do indeed alternate in sign as required).  In this way, we can move the left dot first downwards then along $\sl_2$ parallel to $\sl_1$, then back down along $\sl_1$, through the full twists $\F_2^{k_1}$, and finally settling at a point just to the right of $Q_1$.  See Figure \ref{fig:hslide dotslide}.  As can be seen there, the dot will slide through various crossing data along $\sl_2$ as indicated by the dotted line there, but will then slide through the exact same crossing data along $\sl_1$.  Finally, it slides through $\F_2^{k_1}$ which of course contains $2k_1$ crossings.  In this way, we can guarantee that the dot passes an even number of crossings, so that the sign changes of applying Lemma \ref{lem:Moving Dots in Cone} for each passed crossing all cancel.

We can now isotope the resulting diagram so that it clearly matches the diagram we would draw for $\KLk[2]$ using an insertion point near $P_2$ (see Figure \ref{fig:handleslide} for reference).  Although the chain homotopy equivalence of Lemma \ref{lem:Moving Dots in Cone} uses the full complex $\KC^*(\F_2^{k_1})$ on the left of our diagrams (as opposed to the truncated $\KCsimp_{>-2k_1}(k_1)$ used to build $KC^*(K_L^{1.5})$), we still get isomorphisms on homology groups beyond the truncation point, allowing us to conclude in the limit that $Kh^*(K_L^{1.5})\cong Kh^*(K_L^2)$ and so we are done.

\begin{figure}
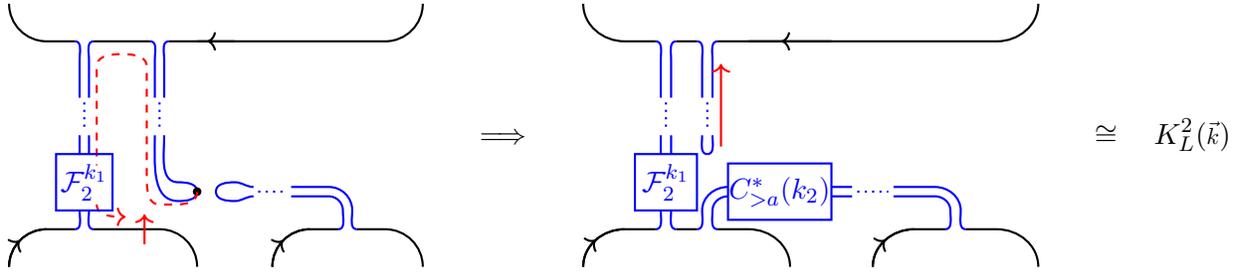

\[\hslideEqvD{$\F_2^{k_1}$} \hspace{.15in}\Longrightarrow\hspace{.15in} \hslideEqvE{$\F_2^{k_1}$}{$\KCsimp_{>a}(k_2)$} \hspace{.15in}\cong\hspace{.15in} \KLk[2]
\]
\caption{The left dot of the dotted cobordisms in Equation \ref{eq:two strand trunc cx} can be slid along $\KLk[1.5]$ as shown by the dashed red arrow.  Along the way it will pass over and under an even number of crossings, so that the overall complex remains unchanged.  The solid red arrows indicate simple planar isotopies that make the resulting complex equivalent to $\KLk[2]$; see Figure \ref{fig:handleslide}.}
\label{fig:hslide dotslide}
\end{figure}

\begin{remark}
We note here that, in this simple knotification scenario where all of our $n_i=2$, we can also establish the invariance of our complex under the surgery wrap move utilizing Lemma \ref{lem:Moving Dots in Cone} and a strategy similar to the proof of Theorem \ref{thm:Kh preserved under handle slides}.  In addition, the stabilization of the truncated complexes $\KCsimp_{>-2k_i}(k_i)$ is already well-known via the formula of Equation \ref{eq:two strand trunc cx}.  In particular, there is no need in this case to appeal to functoriality for cobordism maps commuting with very strong deformation retracts in large multicones in order to arrive at the desired result.  It is our hope that this viewpoint may be helpful in lifting this construction to provide a stable homotopy type $\mathcal{X}(\K_L)$ for such knotifications $\K_L\subset M^r$ as in the work of Lipshitz-Sarkar and Lawson-Lipshitz-Sarkar \cite{LS, LLS}, using techniques similar to those used in the author's construction of a colored homotopy type $\mathcal{X}_c(L)$ in \cite{MW2}.  We look to pursue this further in a future paper.
\end{remark}

\end{proof}

\bibliographystyle{alpha}

\bibliography{References}

\end{document}